\documentclass[bj,preprint]{imsart}

\RequirePackage{amsthm,amsmath,amsfonts,amssymb}
\RequirePackage[numbers]{natbib}

\usepackage{hyperref}
\usepackage{pdfsync}
\usepackage[latin1]{inputenc} 
\usepackage{graphicx,epsfig}
\usepackage{color}
\usepackage{subcaption}
\usepackage{stmaryrd}
\usepackage{makeidx}
\makeindex
\def\R{\mathbb{R}}

\def\N{\mathbb{N}}
\def\P{\mathbb{P}}
\def\E{\mathbb{E}}
\def\L{\mathbb{L}}

\def\R{\mathbb{R}}

\def\Z{\mathbb{Z}}

\def\Lip{\mbox{Lip}}

\def\1{\mbox{I\hspace{-.6em}1}} 

\def\cov{\mbox{Cov}\,}

\def\Lip{\mbox{Lip}\,}
\def\1{\mbox{\hspace{.2em}I\hspace{-.6em}1}} 

\def\limiteasn{\renewcommand{\arraystretch}{0.5}
\begin{array}[t]{c}\stackrel{a.s.}{\longrightarrow} \\
{\scriptstyle
n\rightarrow+\infty}\end{array}\renewcommand{\arraystretch}{1}}

\def\limiten{\renewcommand{\arraystretch}{0.5}
\begin{array}[t]{c}\stackrel{}{\longrightarrow} \\
{\scriptstyle
n\rightarrow+\infty}\end{array}\renewcommand{\arraystretch}{1}}

\def\limiteloin{\renewcommand{\arraystretch}{0.5}
\begin{array}[t]{c}\stackrel{{\cal L}}{\longrightarrow} \\
{\scriptstyle
n\rightarrow+\infty}\end{array}\renewcommand{\arraystretch}{1}}

\def\egaleloi{\renewcommand{\arraystretch}{0.5}
\begin{array}[t]{c}\stackrel{{\cal L}}{\sim } \\
{}\end{array}\renewcommand{\arraystretch}{1}} 

\def\limiteproban{\renewcommand{\arraystretch}{0.5}
\begin{array}[t]{c}\stackrel{{\P}}{\longrightarrow} \\
{\scriptstyle
n\rightarrow+\infty}\end{array}\renewcommand{\arraystretch}{1}}

\def\limiteL1{\renewcommand{\arraystretch}{0.5}
\begin{array}[t]{c}\stackrel{{\L^1}}{\longrightarrow} \\
{\scriptstyle n\rightarrow+\infty}\end{array}\renewcommand{\arraystretch}{1}}

\def\limitel{\renewcommand{\arraystretch}{0.5}
\begin{array}[t]{c}\stackrel{}{\longrightarrow} \\
{\scriptstyle \ell \rightarrow+\infty}\end{array}\renewcommand{\arraystretch}{1}}

\startlocaldefs

  \newtheorem{prop}{Proposition}[section]
  \newtheorem{cor}{Corollary}[section]
\newtheorem{theo}{Theorem}[section]
 \newtheorem{lem}{Lemma}[section]
 
  \newtheorem{Rem}{Remark}[section]
\newtheorem{Def}{Definition}[section]

\endlocaldefs

\begin{document}
\begin{frontmatter}
\title{Contrast estimation of general locally stationary processes using coupling}
\runtitle{Contrast estimation of locally stationary processes}
\begin{aug}
\author[A]{\fnms{Jean-Marc} \snm{Bardet} \ead[label=e1]{bardet@univ-paris1.fr}},
\author[B]{\fnms{Paul} \snm{Doukhan} \ead[label=e2]{paul.doukhan@u-cergy.fr}}
\and
\author[C]{\fnms{Olivier} \snm{Wintenberger} \ead[label=e3]{olivier.wintenberger@upmc.fr}}
\address[A]{University Paris 1 Panth\'eon-Sorbonne, SAMM, France, \printead{e1}}
\address[B]{University Cergy-Pontoise, AGM, France, \printead{e2}}
\address[C]{Sorbonne University, LPSM, France, \printead{e3}}
\end{aug}
\begin{abstract}
This paper aims at providing statistical guarantees for a kernel based  estimation of time varying parameters driving the dynamic of very generals classes of local stationary processes.  We consider coupling arguments in order to extend the results of Dahlhaus {\it et al.} \cite{DRW} to the local stationary version of the  infinite memory processes in Doukhan and Wintenberger \cite{DW}. The estimators are computed as localized M-estimators of any contrast satisfying appropriate regularity conditions. 
We prove the uniform consistency and pointwise asymptotic normality of such kernel based estimators. We apply our results to usual contrasts such as least-square, least absolute value, or quasi-maximum likelihood contrasts. Various local-stationary  processes as ARMA, AR$(\infty$), GARCH, ARCH$(\infty)$, ARMA-GARCH, LARCH$(\infty)$, \dots, and  integer valued processes are also considered. Numerical experiments  demonstrate the efficiency of the estimators on both simulated and  real data sets.
\end{abstract}
\end{frontmatter}

\section{Introduction}
Following the seminal paper of Dahlhaus \cite{dahl96} local-stationarity is considered as a natural set of conditions  for introducing non-stationarity in times series. The chapter of Dahlhaus \cite{dah} yields an exhaustive survey about new results obtained between 1992 and 2012 on this topics.  Dahlhaus and its co-authors have developed a consistent framework studying definitions and properties of local stationary time-varying models (See  \cite{dsr06}, \cite{DP06} and \cite{pd} for instance) as well as associated statistical issues such as  identification and estimation (See \cite{dahl97}, \cite{dahl00}, \cite{dsr06}, \cite{DP06}, \cite{dahl2009} and  \cite{pd}). Note that, except in Dahlhaus and Subba-Rao \cite{dsr06}, the estimators introduced in the previous papers are based on a spectral approximation of the Gaussian likelihood, {\it i.e.} Whittle type estimators. Moreover models considered in this early literature are linear filters of independent inputs.
More recently a  general approach based on derivative processes has been developed in the important  paper of Dahlhaus et al. \cite{DRW}, allowing to  get rid off the linearity condition on the models. Time non-homogeneous  Markovian observations are the considered non-linear time varying models. Under some contraction condition, the observations forget their own past exponentially fast. Anyway many important models like GARCH-type models are not taken in consideration in the above considered literature. The reason is that the Markov representation, adjoining the volatility process, is not observed and, like Hidden Markov Models, do not fall into the setting considered by Dahlhaus et al. \cite{DRW}. Indeed such models may be considered as infinite memory processes, a general class that also models longer memories than the exponential decaying ones.
\noindent 
~\\
The current paper aims at   extending the work of Dahlhaus et al. \cite{DRW} to such an infinite memory setting. 
We first extend the concept of local stationarity beyond the Markovian case. More precisely, 
 the stationary  models introduced by Doukhan and Wintenberger in \cite{DW} are extended to time varying infinite memory causal processes defined as $(X_t^{(n)})_{1\le t\le n}$ a recursive solution to the equation 
\begin{equation*}
X_t^{(n)}=F_{{\boldsymbol{\theta}}_t^{(n)}}\big ((X_{t-k}^{(n)})_{k\in \N^*}; \xi_{t}\big ),\qquad  1\le t\le n,\qquad n \in \N^*=\N\setminus \{0\}
\end{equation*}
where\, $({\boldsymbol{\theta}}_t^{(n)})_{0\leq t \leq n,\, n\in \N^*}$ is a deterministic family with ${\boldsymbol{\theta}}_t^{(n)} \in \Theta \subset \R^d$ for any $0\leq t \leq n$ and $n \in \N^*$, $F_{\boldsymbol{\theta}}$ is a known real-valued function and the innovations $\xi_t$ constitute an independent and identically distributed (i.i.d.)  sequence. For ease of writing we will consider $X_t^{(n)}=0$ for $t\leq 0$, but the arbitrary choice of any deterministic initial values does not change the asymptotic behavior. 
In order to make it tractable, a secondary aim of the paper is to keep the conditions as simple as possible. For instance, in our setting local-stationarity consists simply in the existence of a function ${\boldsymbol{\theta}}^*$ such as for $u\in (0,1)$,  ${\boldsymbol{\theta}}_t^{(n)} \simeq {\boldsymbol{\theta}^*}(u)$ when $t/n \simeq u$ (see the more precise condition \eqref{convtheta} in the so-called Assumption (LS$(\rho)$)). Under this assumption, we define a kernel based estimator $\widehat {\boldsymbol{\theta}}(u)$ of ${\boldsymbol{\theta}^*}(u)$ obtained by the minimization of a localized sum of contrast $\Phi$ (see its definition in \eqref{contrastEst}). We then establish the uniform consistency and the asymptotic normality of this estimator, which is minimax rate optimal, under sharp and general  conditions. 
~\\
The generality of the setting and the relative simplicity of the conditions allow us to recover  existing results on several classes of examples and extend them to infinite memory processes. Clearly  any Markov process with contractive kernel is an infinite memory model (with memory one if the observations are Markovian, with exponential decaying memory if there is a hidden state). The infinite memory representation for GARCH processes is quite appealing since it holds on the observations whereas the Markovian representation holds by adjoining the volatility process.  Processes with infinite memory may provide sharp conditions of convergence of estimation procedures in such models. Indeed, as any contrast is a function of the observations only, the contrast itself has infinite memory (see Bardet and Wintenberger \cite{BW} for a detailed discussion in the stationary setting). We first consider least squares and least absolute values contrasts  time-varying LARCH$(\infty)$ processes and we notably obtained an efficient asymptotic estimation for these infinite memory chains. Considering quasi  log-likelihood contrast also offers numerous estimation results for time-varying finite or infinite memory processes: we obtain the uniform consistency and the asymptotic normality for time varying AR$(\infty)$ and  ARCH$(\infty)$. For finite memory time varying ARMA$(p,q)$ or GARCH$(p,q)$,  our results recover previous ones of Dahlhaus and co-authors  \cite{dahl97}, \cite{dsr06} and \cite{DRW}, valid only for ARMA and ARCH models, and extend them to the important GARCH model class.
Dahlhaus et al. \cite{DRW} used functional dependence conditions, new coupling arguments developed from Dedecker and Prieur \cite{dp} are at the basis of the work, since no strong mixing condition may be expected and longer memories than exponential decaying ones have to be tackled. One advantage of our approach is that it goes around the derivative process, a drawback is that the rates of convergence in the asymptotic normality are sub-optimal. Finally we also apply our strategy to prove estimation convergence for time varying ARMA-GARCH processes and time varying integer valued Poisson-GLM type processes. 
~\\
Numerical studies are also proposed. Firstly, Monte-Carlo experiments show the accuracy of the estimator in several cases of time varying processes. However, these simulations also exhibit that such non-parametric estimate requires sufficient large sample sizes (at least one thousand in many cases). Secondly, an application to  financial data (the S\&P500 data from October 1990 to October 2020) demonstrates  the evolution of the parameters in case a typical GARCH$(1,1)$-model is used. \\
~\\
The forthcoming Section \ref{nonstat} is devoted to the definition and existence of new non-stationary models. In Section \ref{Mesti}, the definition of the nonparametric estimator as well as its uniform consistency and asymptotic normality are stated, while Section \ref{Examples} reviews several important cases. Numerical experiments are proposed in Section \ref{simu} and  proofs are postponed in Sections \ref{proofs1} and \ref{proofs2}.

\section{Preliminaries}\label{nonstat}
\subsection{Notation}
Some standard notation is used:
\begin{itemize}
\item The symbol $0$ denotes any null vector in any vector space;
\item If $V$ is a vector space then $V^{\infty}=\{(x_n)_{n\in \N}\in V^{\N};\exists N \in \N, \, x_k=0,~\mbox{for all}~k>N\}$;
\item The symbol $\| \cdot \|$ denotes the usual Euclidean norm
of a vector or the associated norm of a matrix;
\item For $p\geq 1$ and $Z$ a random vector in $\R^m$, denote:
$\|Z \|_p =  \big [\E (\|Z\|^p) \big ] ^{1/p};$  
\item For the measurable
vector- or matrix-valued function $g$ defined on some set $U$, $\| g
\|_U=\sup_{u\in U} \| g(u) \|$;
\item From now on $\Theta$ denotes a subset of $\R^d$, and $\stackrel{\circ}{\Theta}$ is the interior of $\Theta$. If $V$ is a Banach space then ${\cal C}(\Theta,V)$ denotes the
Banach space of $V$-valued continuous functions on $\Theta$
equipped with the uniform norm $\| \cdot \|_\Theta$ and $\L^p({\cal
C}(\Theta,V))$ ($p\ge 1$) denotes the Banach space of random a.e. continuous
functions $f$ such that $\E\big [\|f\|_\Theta^p\big ]<\infty$.
\item For $\boldsymbol{\theta}\in \Theta$ and $\Psi_{\boldsymbol{\theta}}:~\R^\infty \to V$ a Borel function with values in a finite dimensional vector space $V$,
 $\partial_{\boldsymbol{\theta}}^k\Psi_\theta(x)$ denotes respectively for $k=0,1,2$, in case they exist, $\Psi_{\boldsymbol{\theta}}(x)$, $\displaystyle 
 {\partial \Psi_{\boldsymbol{\theta}}(x)}/{\partial {\boldsymbol{\theta}}}$ and $\displaystyle 
  {\partial^2 \Psi_{\boldsymbol{\theta}}(x)}/{\partial {\boldsymbol{\theta}}^2}$ for $x \in \R^\infty$. 
\end{itemize}

\subsection{Stationary infinite memory processes}
In all the sequel, we will consider a given real number $p \geq 1$. \\
~\\
Set $\boldsymbol{\theta} \in \R^d$ and let a function $F_{\boldsymbol{\theta}}$ be defined as follows
\begin{equation}
F_{\boldsymbol{\theta}}: ~(x,y)\in \R^\infty\times \R \mapsto F_{\boldsymbol{\theta}}(x,y) \in \R.
\end{equation}  
Doukhan and Wintenberger proved in \cite{DW} the existence and uniqueness of the stationary solution of the recurrence  equation
\begin{equation}
\label{chaineinf}
X_t=F_{\boldsymbol{\theta}}\big ((X_{t-k})_{k\in \N^*}, \xi_{t}\big ), \qquad  \mbox{for all}~ t\in\Z,
\end{equation}
where $(X_t)$ is a process with values in $\R$ and where  $(\xi_t)_{t\in \Z}$ is a sequence of i.i.d. random variables (r.v.). This framework  provides a parametric representation of models such as nonlinear autoregressive or conditionally heteroskedastic time series for instance.\\

The existence of a stationary solution in $\L^p$ of the above equation relies on a contraction argument on the function $F_{\boldsymbol{\theta}}$. As a consequence, we define the following family of assumptions {\bf(A$_k$($\Theta$))} for $k=0,1,2$ and some compact subset $\Theta$ of $\R^d$:\\
\vspace{0.2cm}
\newline {\bf(A$_k$($\Theta$))} {\em For ${\boldsymbol{\theta}}\in \Theta$, we assume that the functions 
$\partial_{\boldsymbol{\theta}}^k F_{\boldsymbol{\theta}}$ exist on $\R^\infty \times \R$ for $k=0,1,2$. Moreover there exists a sequence
$\big(b^{(k)}_j(\Theta)\big)_{j}$ of nonnegative numbers such that  for all $ x$, $y\in\R^{\infty}$
\begin{eqnarray}
&\bullet & C_k(\Theta)= \Big \|\sup_{{\boldsymbol{\theta}}\in\Theta}\big \|\partial_{\boldsymbol{\theta}}^k F_{\boldsymbol{\theta}}(0 , \xi_{0})\big \| \Big \|_p <\infty\\
&\bullet  &\Big \|\sup_{{\boldsymbol{\theta}}\in \Theta} \big\|\partial_{\boldsymbol{\theta}}^k F_{\boldsymbol{\theta}}(x,\xi_0) -\partial_{\boldsymbol{\theta}}^k F_{\boldsymbol{\theta}}(y,\xi_0)\big\| \Big \|_p \le \displaystyle  \sum_{j=1}^\infty b^{(k)}_j(\Theta) \, \|x_j-y_j\|_p,
\end{eqnarray} }
$\displaystyle \qquad \qquad \qquad \qquad \qquad \qquad \qquad \qquad \qquad \mbox{with}~B_k(\Theta)=\sum_{j=1}^\infty b^{(k)}_j(\Theta)<\infty.$
\newline 
\noindent Thus, from \cite{DW}, under the uniform contraction conditions {\bf(A$_0$($\Theta$))} with $B_0(\Theta)<1$, there exists a unique stationary  solution of \eqref{chaineinf} in $\L^p$ (defined  almost surely).
\subsection{Local stationary infinite memory process}
If we replace now ${\boldsymbol{\theta}}$ by the time-varying ${\boldsymbol{\theta}}^{(n)}_t$ such that ${\boldsymbol{\theta}}^{(n)}_t \in \Theta$, then the uniform contraction conditions {\bf(A$_0$($\Theta$))} with $B_0(\Theta)<1$ ensure the existence of a non-stationary $\L^p$-process. More precisely, we define the triangular array $(X_t^{(n)})_{1\leq t \leq n,\,n\in \N^*}$ such as:
\begin{equation}  
\label{chaineinfloc} 
X_t^{(n)}=F_{{\boldsymbol{\theta}}_t^{(n)}}\big ((X_{t-k}^{(n)})_{k\in \N^*}; \xi_{t}\big ),\qquad  1\le t\le n,\, n \in \N^*\,,
\end{equation}
where $({\boldsymbol{\theta}}_t^{(n)})_{0\leq t \leq n,\, n\in \N^*}$ is a family of real numbers   ${\boldsymbol{\theta}}_t^{(n)} \in \Theta \subset \R^d$ for any $0\leq t \leq n$ and $n \in \N^*$. \\ 
In order to make this recursion possible we also set initial conditions
\begin{equation}\label{init}
X_t^{(n)}=0,\quad\mbox{for $t\le 0$}.
\end{equation}
The solution of the above equations is no longer stationary. However, we can establish the following result (its proof as well as all the other ones are postponed in the Sections \ref{proofs1} and \ref{proofs2}):
\begin{lem}\label{momm} 
Let $\Theta \subset \R^d$ such that {\bf(A$_0$($\Theta$))} holds with $B_0(\Theta)<1$. Then, under the assumption \eqref{init}, the nonstationary triangular array $(X_t^{(n)})_{0\leq t \leq n, \, n\in \N^*} $, solution of \eqref{chaineinfloc}, remains in $\L^p$ and it satisfies
$$
\sup_{\ \ n \in \N^*, ~0\le s\le n}\|X_s^{(n)}\|_p\le\frac  {C_0(\Theta)}{1-B_0(\Theta)}.
$$ 
\end{lem}

\subsection{The stationary version}
We introduce a function $u\mapsto {\boldsymbol{\theta}^*}(u)$ on $[0,1]$; this is a continuous time approximation of the triangular array of parameters $({\boldsymbol{\theta}}_t^{(n)})_{0\leq t \leq n,\, n\in \N^*}$. We consider $\rho \in (0,1]$ and assume the following local-stationarity assumption: \\
~\\
{\bf Assumption (LS$(\rho)$)}: {\em There exist $K_{\boldsymbol{\theta}}>0$ and a continuous function $\boldsymbol{\theta}^*:~u \in [0,1] \mapsto \boldsymbol{\theta}^*(u) \in \R^d$, such as
\begin{equation}\label{convtheta}
\big \| \boldsymbol{\theta}^{(n)}_t -\boldsymbol{\theta}^*(u) \big \| \leq K_{\boldsymbol{\theta}} \,\Big |u-\frac tn\Big |^\rho,\quad \mbox{for any \ $n \in \N^*$ and $1\leq t \leq n$}.
\end{equation}}
\newline 
\noindent This condition describes an H\"older type behavior for the approximation of  $\boldsymbol{\theta}^{(n)}_t$  by the function $\boldsymbol{\theta}^*$. When $t/n \limiten u $ then $\boldsymbol{\theta}^{(n)}_t \limiten  \boldsymbol{\theta}^*(u)$ and the behavior of $(X_{t}^{(n)})$ is similar to its so-called stationary version.
\begin{Def}
If it exists, we define $(\widetilde X_{t}(u))_{ t \in \Z} $ as any solution of the recursion
\begin{equation}\label{eq:statu}
\widetilde X_{t}(u)= F_{\boldsymbol{\theta}^*(u)}\big ((\widetilde X_{t-k}(u))_{k\geq 1}, \xi_t\big ),\qquad t\in \Z\,,
\end{equation}
and we call it the stationary version of $(X_{t}^{(n)})$  at $ u \in [0,1]$.
\end{Def}
Note that from \cite{DW} the existence of the stationary version is implied by contraction assumptions. Namely, if the function   $\boldsymbol{\theta}^*$ satisfies  $\boldsymbol{\theta}^*(u)\in\Theta \subset \R^d$ for each $u\in[0,1]$ and is such that {\bf(A$_0$($\Theta$))} holds with $B_0(\Theta)<1$, then there exists a.s. a unique stationary stationary version $(\widetilde X_{t}(u))_{t\in \Z}$ satisfying \eqref{eq:statu}  and  
$$
\sup_{t\in \Z}\|\widetilde X_{t}(u)\|_p\le \frac {C_0(\Theta)} {1-B_0(\Theta)}\,,\qquad u\in[0,1]\,.
$$

\section{M-estimation for infinite memory processes}\label{Mesti}
\subsection{The stationary case}
We recall the framework of contrast estimation for infinite memory chain as  in Bardet and Wintenberger \cite{BW}. Let $(X_t)$ be a stationary solution of the infinite memory model \eqref{chaineinf} with parameter $\boldsymbol{\theta}^* \in \Theta \subset \R^d$  such that {\bf(A$_0$($\Theta$))} holds with $B_0(\Theta)<1$. We estimate  $\boldsymbol{\theta}^* $ using an M-estimator based on an observed path $(X_1,\ldots,X_n)$. \\
 
We define a contrast function $\Phi(x,\boldsymbol{\theta}) $ that satisfies a set of regularity assumptions combined in the definition of the space $\Lip_{p}(\Theta)$ for $\Theta \subset \R^d$ (always with $1\leq p$):\\
~\\ 
{\bf Space} $\Lip_{p}(\Theta)$: {\em A Borel function $h:~\R^\infty \times \Theta \to \R$ belongs to $\Lip_{p}(\Theta)$ if there exist a sequence of non-negative numbers $(\alpha_i(h,\Theta))_{i\in \N}$ where $\sum_{s=1}^\infty \alpha_s(h, \Theta)<\infty$ and a function $g:[0,\infty)^2\to [0,\infty)$ such as for any sequences $U=(U_i)_{i\in \N^*}\in (\L^p)^\infty$ and $V=(V_i)_{i\in \N^*}\in (\L^p)^\infty$ satisfying $\sup_{s\geq 1} \{ \|U_s\|_p \vee \|V_s \|_p\} <\infty$, one obtains:
\begin{equation}\label{LipL}
\left \{ \begin{array}{l}
\E \big [ \sup_{\boldsymbol{\theta} \in \Theta}  \big |h(0,\boldsymbol{\theta})\big | \big ]<\infty; \\ \displaystyle
\E \big [ \sup_{\boldsymbol{\theta} \in \Theta} \big |h(U,\boldsymbol{\theta}) -h(V,\boldsymbol{\theta})\big |\big ] \leq g \big (\sup_{s\geq 1} \big \{\|U_s\|_p\vee\|V_s \|_p \big \} \big )\; \sum_{s=1}^\infty \alpha_s(h, \Theta) \, \| U_s-V_s\|_p.
\end{array} \right .
\end{equation}} 
\newline
\noindent Note that if $h \in \Lip_{p'}(\Theta)$ then $h \in \Lip_{p}(\Theta)$ when $p \leq p'$ thanks to Jensen's inequality. It is possible (see below the general case for non-stationary models) to prove that if $\Phi \in \Lip_{p}(\Theta)$ and if  the stationary solution $(X_t)$ admits finite $p$ moments, then $\Phi \big ((X_{-t})_{t \in \N},\boldsymbol{\theta} \big )$ exists in $\L^1$ for any $\boldsymbol{\theta} \in \Theta$. The existence of first order moments is crucial for ensuring that $\Phi$ is a proper score function which is implied by the following condition:\\
\newline {\bf Assumption} ${\bf (Co}(\Phi,\Theta))$: {\em The function $\Phi \in \Lip_{p}(\Theta)$ for $p\ge 1$ is such that for   $(X_t)_{t\in \Z}$ satisfying the infinite memory model \eqref{chaineinf} with parameter $\boldsymbol{\theta}^* \in \stackrel{\circ}{\Theta}$ and with ${\cal F}_0=\sigma\big ((X_{-k})_{k\in \N} \big )$,
\begin{eqnarray} 
\mbox{$\boldsymbol{\theta}^*$ is the unique minimum of the function } \boldsymbol{\theta} \in \Theta \mapsto \E \big [\Phi \big ((X_{1-k})_{k\in \N},\boldsymbol{\theta} \big ) | {\cal F}_0\big ]\mbox{ in } \stackrel{\circ}{\Theta}.\label{min} 
\end{eqnarray}}
\newline
\noindent As a consequence, this condition is depending on the function $\boldsymbol{\theta} \in \Theta \mapsto F_{\boldsymbol{\theta}}(\cdot)$ driving the infinite memory model \eqref{chaineinf}. Therefore the contrast function $\Phi$ will be chosen with respect to this function $F_{\boldsymbol{\theta}}$.  This  is thus natural to define the M-estimator of $\boldsymbol{\theta}$ by
$$
\widehat{\boldsymbol{\theta}}_n=\mbox{Arg}\! \min_{\! \! \! \! \boldsymbol{\theta} \in \Theta}  \frac1 n \, \sum_{t=1}^{n} \Phi \big ((X_{t-i})_{i\in \N},\boldsymbol{\theta}\big ).
$$
Note that the factor $1/n$ aims  at establishing a Law of Large Numbers. Indeed, if the almost sure convergence holds, {\it i.e.}  
\begin{equation*}
\sup_{\boldsymbol{\theta} \in \Theta} \Big | \frac1n \, \sum_{t=1}^{n} \Phi((X_{t-i})_{i\in \N},\boldsymbol{\theta})-\E \big [\Phi((X_{-i})_{i\in \N},\boldsymbol{\theta}) \big ]\Big |  \limiteasn 0,
\end{equation*}
then usual arguments  imply\ $\widehat {\boldsymbol{\theta}}_n \limiteasn \boldsymbol{\theta}^*$. 

\subsection{The local stationary case}
We extend  the notion of contrast function $\Phi$  to the non-stationary process $\left(X_t^{(n)}\right)_{t \in \Z}$ for time-varying parameters $ \boldsymbol{\theta}_t^{(n)}$. The first step below is to prove the integrability.
\begin{lem}\label{exiphi}
Let $(X_t^{(n)})_{t \in \Z}$ satisfy the non-stationary infinite memory model \eqref{chaineinfloc} under condition {\bf(A$_0$($\Theta$))} with $B_0(\Theta)<1$ and let $\Phi \in \Lip_{p}(\Theta)$ with $ p\ge1$. Then for any $\boldsymbol{\theta} \in \Theta$, the sequence of contrasts $\big (\Phi \big ( (X_{t-k}^{(n)})_{k\in \N}, \boldsymbol{\theta}   \big ) \big)_{t\in \Z}$ exists in $\L^1$. Moreover, under Assumption {\bf (LS($\rho$))},  and with $(\widetilde X_{t}(u))_{t\in \Z}$ the stationary version defined in \eqref{eq:statu}, $\big (\Phi \big ( (\widetilde X_{t-k}(u))_{k\in \N}, \boldsymbol{\theta}   \big ) \big)_{t\in \Z}$ is a stationary ergodic process.
\end{lem}
Under the local stationary assumption {\bf (LS($\rho$))}, we can expect estimating $\boldsymbol{\theta}^*(u)$ with $0<u<1$ defined in \eqref{convtheta} thanks to a M-estimator based on the observations $X_t^{(n)} $ for $t/n\simeq u$. 
The previous M-estimator has to be localized around $t$ such as $t \simeq nu$ using a convolution kernel $K$ with a compact support (for simplicity):
\begin{Def}
Let a kernel function $K:\R \to \R$ be such as:
\begin{itemize}
\item $K$ has a compact support, {\it i.e.} there exists $c>0$ such as $K(x)=0$ for $|x|\geq c$;
\item $K:\R\to\R $ is piecewise differentiable 
 with  $\int_\R K(x)dx=1$, $C_K=\sup_{x \in \R} |K(x)|<\infty$.
\end{itemize}
Then, with a bandwidth sequence $(h_n)_{n\in \N}$ of positive numbers, we define the kernel based estimator of $\theta^*(u)$ as
\begin{equation}
\label{contrastEst}
\widehat{\boldsymbol{\theta}}(u)= \arg \!  \min_{\! \! \! \! \!  \! \!\!\boldsymbol{\theta}\in \Theta}\frac 1 {nh_n} \, \sum_{j=1}^n \Phi\big ((X^{(n)}_{j-i})_{i\in \N},\boldsymbol{\theta }\big ) \, K\Big(\frac{\frac jn-u}{h_n}\Big )\,,\qquad u\in(0,1).
\end{equation}
\end{Def}
Under   weak conditions, this estimator is uniformly consistent:
\begin{theo} \label{theo1}
Let $(X_t^{(n)})_{t\in \N}$ be the solution of the non-stationary infinite memory model \eqref{chaineinfloc} which satisfies
 Assumption {\bf(A$_0$($\Theta$))} with $B_0(\Theta)<1$ and $\sum_{t=2}^\infty t\log(t) b_t(\Theta)<\infty$ and  Assumption {\bf(A$_1$($\Theta$))}, with also Assumption   {\bf (LS($\rho$))}. If for $p\geq 1$, $\Phi \in \Lip_{p}(\Theta)$  with  $\sum_{s\ge 0}s \, \alpha_s(\Phi, \Theta)<\infty$ satisfying Assumption ${\bf (Co}(\Phi,\Theta))$ then, for any $u\in (0,1)$, $\widehat{\boldsymbol{\theta}}(u)$ consistently estimates $\boldsymbol{\theta}^*(u)$:
$$
\widehat{\boldsymbol{\theta}}(u)\limiteproban\boldsymbol{\theta}^*(u),\quad\mbox{if}\quad h_n \limiten 0\quad\mbox{and}\quad nh_n \limiten \infty.$$ 
Moreover, if $p>1$ and $n^{1-1/p} h_n \limiten \infty$ then for any $\varepsilon>0$ then $\widehat{\boldsymbol{\theta}}$ uniformly consistently estimates $\boldsymbol{\theta}^*$:
\begin{equation}\label{convtheta2}
\sup_{u\in [\varepsilon, 1-\varepsilon]}\big\|\widehat{\boldsymbol{\theta}}(u)- \boldsymbol{\theta}^*(u)\big\| \limiteproban 0\,.
\end{equation}
\end{theo}
\begin{Rem}We notice that the uniform consistency of the kernel estimator was already obtained for Markov processes by Dahlhaus {\it et al.} in \cite{DRW} under a different set of assumptions. Our extra condition on the bandwidth $n^{1-1/p} h_n \to \infty$, $n\to\infty$,  corresponds to the extra condition in $(ii)$ of Theorem 5.2 of \cite{DRW} with $M=0$. We notice also that our extra condition implicitly  requires that $p>1$.
\end{Rem}
\textcolor{black}{
\begin{Rem}
Consistency of $\widehat{\boldsymbol{\theta}}(u)$ for any $u\in (0,1)$ can be achieved relaxing the uniform contraction condition $B_0(\{\Theta)<1$ of Theorem \ref{theo1}. Instead assume the pointwise contraction 
\begin{equation}\label{cond:weakcontr}
B_0(\{\boldsymbol{\theta}^*(u)\})<1\,,\qquad u\in (0,1)\,.
\end{equation}
Then a continuity argument under {\bf(A$_0$($\Theta$))} yeilds for any $u\in (0,1)$ the existence of $\varepsilon>0$ such that 
$$
B_0(\{\boldsymbol{\theta}^ \in \Theta: \, \|\boldsymbol{\theta}-\boldsymbol{\theta}^*(u)\|\le \varepsilon\})<1\,.
$$ 
The M-estimator defined in \eqref{contrastEst} with $\{\boldsymbol{\theta} \in \Theta: \, \|\boldsymbol{\theta}-\boldsymbol{\theta}^*(u)\|\le \varepsilon\}$ in place of $\Theta$ consistently estimates $\theta^*(u)$ by an application of Theorem \ref{theo1}. That this M-estimator coincides with $\widehat{\boldsymbol{\theta}}(u)$ for $n$ sufficiently large follows by a compacity argument of Pfanzagl \cite{Pf}; The compact set $\Theta$ is  covered by finitely many compact sets of diameter $\varepsilon$ among which only one contains $\boldsymbol{\theta}^* (u)$. Applying the SLLN of Lemma \ref{TLCK} 1., the uniqueness in Condition ${\bf (Co}(\Phi,\Theta))$ implies that the minimizer of \eqref{contrastEst} belongs to the unique compact set containing $\boldsymbol{\theta}^* (u)$ for $n$ large enough. See \cite{BW} for details. Note that we were not able to extend the uniform consistency of $\widehat{\boldsymbol{\theta}}$ under pointwise contraction, which is consistent with \cite{DRW} that worked under uniform contraction as well. Uniform contraction seems neceassary for proving  uniform moments   in Lemma \ref{mommbis}.
\end{Rem}}
For establishing the asymptotic normality of $\widehat{\boldsymbol{\theta}}(u)$, we analogously need extra assumptions on the differentiability of the contrast $\Phi$ and the integrability of its derivatives. 
We have:
\begin{theo} \label{theo3}
Let the assumptions of Theorem \ref{theo1} hold with $\Theta $ a compact set. Assume also that for any $x \in \R^\infty$, 
$\boldsymbol{\theta} \in \Theta \mapsto \Phi(x,\boldsymbol{\theta})$ is a ${\cal C}^2(\Theta)$ function such as $\partial  _{\boldsymbol{\theta}}\Phi \in \Lip_p(\Theta)$ with $ \sum_{s=1}^\infty s \, \alpha_s(\partial _{\boldsymbol{\theta}} \Phi, \Theta)<\infty$,
and for any $u \in (0,1)$, with $\rho \in (0,1]$ defined in Assumption   {\bf (LS($\rho$))},
\begin{itemize}
\item  
$\E \big [\big \| \partial _{\boldsymbol{\theta}} \Phi\big ((\widetilde X_k(u))_{k\leq 0}, \boldsymbol{\theta}^*(u) \big )\big \|^2 \big ]<\infty$ and $\Sigma\big (\boldsymbol{\theta}^*(u)\big )$ is a definite positive matrix with 
\begin{equation*}\label{SigmaMart}
\Sigma\big (\boldsymbol{\theta}^*(u)\big )= \int_\R K^2(x)dx \cdot 
\E\Big [\partial _{\boldsymbol{\theta}} \Phi \big ((\widetilde X_{-k}(u))_{k\in \N},\boldsymbol{\theta}^*(u)\big ) \, \big (\partial _{\boldsymbol{\theta}} \Phi \big ((\widetilde X_{-k}(u))_{k\in \N},\boldsymbol{\theta}^*(u)\big ) \big )^\top \Big ] ;
\end{equation*}
\item $\Gamma(\boldsymbol{\theta}^*(u)) =\E \big [ \partial^2 _{\boldsymbol{\theta}^2} \Phi\big ((\widetilde X_{-k}(u))_{k\in \N}, \boldsymbol{\theta}^*(u) \big )\big ]$ is a  positive definite  matrix.
\end{itemize}
If $(h_n)_n$ is a sequence of positive numbers such that 
\begin{equation}\label{condbn}
nh_n \limiten \infty\quad\mbox{and} \quad nh_n^{1+2\rho}\limiten 0,
\end{equation}
then, for any $u\in (0,1)$,
\begin{equation}\label{tlctheta}
 \sqrt{n h_n}\, \big ( \widehat{\boldsymbol{\theta}}(u)- \boldsymbol{\theta}^*(u) \big ) \limiteloin  {\cal N}_d \Big ( 0 \, , \, \Gamma^{-1}(\boldsymbol{\theta}^*(u))  \Sigma\big (\boldsymbol{\theta}^*(u)\big )\Gamma^{-1}(\boldsymbol{\theta}^*(u)) \Big ).
\end{equation}
\end{theo}
\begin{Rem}
A first consequence of this result is that the  convergence rate of $ \widehat{\boldsymbol{\theta}}(u)$ is $o(n^{-\rho/(2\rho +1)})$, which is just below the classical minimax convergence rate in a non parametric framework  for any $\rho \in (0,1]$. Then the optimal choice of the bandwidth satisfies $h_n=o(n^{-1/(2\rho +1)})$ and the estimator is also uniformly consistent whenever $p>(2\rho +1)/(2\rho)$. 
Under additional conditions, Rosenblatt \cite{ros1} and Dahlhaus {\em et al.} \cite{DRW} derive expressions for an equivalent of the bias; in this case {\it i.e.} $ nh_n^{1+2\rho}\limiten \ell\ne0$, one may use the classical minimax bandwidth and, then,  a non-centred Gaussian limit theorem   occurs.
\end{Rem}
\textcolor{black}{\begin{Rem}\label{rem:weakcontr}
A precise inspection of the proof of Theorem \ref{theo3} shows that only the consistency of $\widehat{\boldsymbol{\theta}}(u)$, $u\in (0,1)$, is used and not the uniform consistency of $\widehat{\boldsymbol{\theta}}$. Thus the asymptotic normality \eqref{tlctheta} can be extending under the pointwise contraction \eqref{cond:weakcontr}.
\end{Rem}}
\begin{Rem}
Considering $(u_1,\ldots,u_m)$ instead of $u$, a multidimensional central limit theorem could also be obtained extending \eqref{tlctheta}. Such a result could be interesting for testing the goodness-of-fit ($H_0:~{\boldsymbol{\theta}}^*={\boldsymbol{\theta}}_0$) or the stationarity ($H_0:~{\boldsymbol{\theta}}^*=C_0 \in \R^d$) of the process. This will be the subject of a forthcoming paper.
\end{Rem}

\section{Examples} \label{Examples}  \label{subExamples}
Here we develop several examples of locally stationary  infinite memory models with contrast functions $\Phi \in \Lip_{p}(\Theta)$ for which the Assumption {\bf (Co}($\Phi,\Theta)$) is satisfied. We also check the conditions of Theorems \ref{theo1} and \ref{theo3} in order to assert the uniform consistency and the asymptotic normality  of the localized M-estimator.

\subsection{Time varying AR$(1)$ processes}
In the case of time varying AR$(1)$ (denoted further as tvAR$(1)$) processes defined by
\begin{equation}\label{tvAR1}
X^{(n)}_t=\theta^{(n)}_{t}\,X^{(n)}_{t-1}+ \xi_t, \quad\mbox{for $1\leq t \leq n$, $n\in \N^*$},
\end{equation}
with $X^{(n)}_t=0$ for any $t\leq 0$, and  $ \Theta=[-r,r]$ with $0<r<1$.

\subsubsection{Least Square contrast}
When $\Phi_{LS}$ is the Least Square (LS) contrast defined as
\begin{equation}\label{phiAR1}
\Phi_{LS}(x,\theta)=(x_1-\theta \, x_2)^2\,,
\end{equation}
we obtain the usual Yule-Walker (or Least Square) estimator of $\theta^* (u)$ if the stationary version $(\widetilde X_t(u))$ were observed. Clearly Assumption ($Co(\Phi,\Theta)$) holds and using H\"older Inequality, we obtain than $\Phi_{LS} \in \Lip_{p}(\Theta)$ with $p=2$, 
\begin{multline*}
\E \big [| \sup_{\theta \in \Theta}\! \big |\Phi_{LS}\!(U,\theta) -\Phi_{LS}\!(V,\theta)\big | \big ]
 \leq (1\!+\!r)\max_{1\leq s \leq 2}\!\! \big \{ \big \|U_i\|_{2}\vee\big \|V_i\|_{2} \}   \big ( \|U_1\!-\!V_1\|_{2}+r  \|U_2\!-\!V_2 \|_{2} \big ),
\end{multline*}
and therefore $\alpha_1(\Phi_{LS}, \Theta)=1$, $\alpha_2(\Phi_{LS}, \Theta)=r$ and $\alpha_j(\Phi_{LS}, \Theta)=0$ for $j\geq 3$.
From basic calculation we also have
\begin{equation*}
\partial_\theta \Phi_{LS}(x,\theta)=2x_2 (\theta x_2 -x_1)\quad \mbox{and} \quad \partial^2_{\theta^2} \Phi_{LS}(x,\theta)=2x_2^2.
\end{equation*}
After elementary algebra, we obtain:
\begin{multline*}
\E \Big [\sup_{\theta \in [-r,r]}\big | \partial_\theta \Phi_{LS}(U,\theta)-\partial_\theta \Phi_{LS}(U,\theta) \big | \Big ]
\\
\leq  4 \, \big ( \|U_1-V_1\|_2 + \|U_2-V_2\|_2 \big ) \, \big ( \|U_1\|_2+ \|V_1\|_2+ \|U_2\|_2+ \|V_2\|_2 \big )
\end{multline*}
from H\"older's inequality. Analogously, 
\begin{equation*}
\E \Big [\sup_{\theta \in [-r,r]} \big | \partial^2_{\theta^2} \Phi_{LS}(U,\theta)-\partial^2_{\theta^2} \Phi_{LS}(V,\theta)\big | \Big ]\leq 2 \, \big (  \|U_2-V_2\|_2 \big ) \, \big ( \|U_2\|_2+ \|V_2\|_2 \big ),
\end{equation*}
ensuring that  $\partial_\theta \Phi_{LS}$ and $\partial^2_{\theta^2} \Phi_{LS}$ are included in $\Lip_2 \big ([-r,r]\big )$. If $\E [\xi_0^2 ]<\infty$, both the matrix $\Sigma\big ({\theta}^*(u)\big )= 4 \, \Big ( \int_\R K^2(x)dx \Big ) \, \sigma_\xi^4 (1-\theta^*(u)^2)^{-1} $ and $\Gamma({\theta}^*(u))=2 \sigma_\xi^2 (1-\theta^*(u)^2)^{-1}$ are positive definite.  Then, by an application of Theorem \ref{theo3} we obtain:
\begin{cor} \label{LSAR1}
If $\E [\xi_0^2 ]<\infty$ and if $\theta_t^{(n)}\in \Theta=[-r,r]$ satisfies Assumption {\bf (LS($\rho$))}, 
the localized least square estimator is asymptotically normal when the sequence $(h_n)_n$ satisfies \eqref{condbn} and we obtain for any $u\in (0,1)$
$$
\sqrt{n\, h_n}\, \big ( \widehat{\theta}(u)- {\theta}^*(u) \big ) \limiteloin  {\cal N} \Big ( 0 \, , \, \big (1-\theta^*(u)^2\big ) \int_\R K^2(x)dx \Big ).
$$
\end{cor}
Here, we recover for $0<\rho \le1$ the results on tvAR$(1)$ models obtained by Bardet and Doukhan in \cite{bd2017}, which are also valid for functions $u \in (0,1) \to \widehat{\theta}(u)$ with H\"olderian derivatives. 
\subsubsection{Least Absolute Value contrast}
In the framework of tvAR$(1)$ processes  \eqref{tvAR1}  a classical  alternative of the LS contrast, known for its robustness, is the Least Absolute Values (LAV) contrast defined as follows on $\theta \in \Theta=[-r,r]$ with $0<r<1$, 
\begin{equation}\label{phiLAV}
\Phi_{LAV}(x,\theta)=\big |x_1-\theta \, x_2 \big |.
\end{equation}
If the stationary version $(\widetilde X_t(u))$ were observed, we obtain the usual  estimator of $\theta$ and Assumption ($Co(\Phi_{LAV},\Theta)$) holds. In such a case, $\Phi_{LAV} \in \Lip_{p}(\Theta)$ for any $1\leq p$, and we obtain 
$$
\E \big [ \sup_{\theta \in \Theta} \big |\Phi_{LAV}(U,\theta) -\Phi_{LAV}(V,\theta)\big |\big ]\leq \|U_1-V_1\|_p+ r \, \|U_2-V_2\|_p.
$$
implying $\alpha_1(\Phi_{LAV}, \Theta)=1$ and $\alpha_2(\Phi_{LAV}, \Theta)=r$ and $\alpha_j(\Phi_{LAV}, \Theta)=0$ for $j\geq 3$. Since $\Phi_{LAV}$ is not a differentiable function, we will restrict our purpose to the uniform consistency of $\widehat \theta(u)$. We obtain the following result from Theorem \ref{theo1}:
\begin{cor}\label{LAV}
If $\|\xi_0\|_1<\infty$ and if $(\theta_t^{(n)})\in \Theta=[-r,r]$ satisfies Assumption {\bf (LS($\rho$))}, then 
the localized LAV estimators $\widehat \theta_n$ satisfies \eqref{convtheta2}.
\end{cor}

\subsection{Causal affine processes and Gaussian QMLE}  
We consider the general class of causal affine processes $(X_t)$ defined by Bardet and Wintenberger in \cite{BW} as
\begin{equation}\label{causal}
X_t=M_{\boldsymbol{\theta} }\big ((X_{t-i})_{i\geq 1}\big )\, \xi_t + f_{\boldsymbol{\theta} }\big ((X_{t-i})_{i\geq 1}\big ), \qquad \mbox{ for any } t\in \Z, \boldsymbol{\theta}  \in \Theta
\end{equation}
with $\Theta$ a compact subset of $\R^d$. We assume the existence of Lipschitz coefficient sequences $\big ( \beta_i(f,\Theta)\big )_{i\in \N}$ and $\big ( \beta_i(M,\Theta)\big )_{i\in \N}$ such as for $K_{\boldsymbol{\theta} }=f_{\boldsymbol{\theta} }$ or $M_{\boldsymbol{\theta} }$,
\begin{equation} \label{lipsh}
\sup_{\boldsymbol{\theta} \in \Theta} \big |K_{\boldsymbol{\theta} }(x)-K_{\boldsymbol{\theta} }(y) \big | \leq \sum_{i=1}^\infty \beta_i(K,\Theta) \, \big |x_i-y_i\big |,
\end{equation}
for any $x, \, y \in \R^\infty$. Then, $(X_t)$ satisfies the infinite memory model \eqref{chaineinf} with $F_{\boldsymbol{\theta} }(x,\xi_0)=f_{\boldsymbol{\theta} }(x)+ \xi_0 \, M_{\boldsymbol{\theta} }(x)$ and {\bf(A$_0$($\Theta$))} holds when $\sum _j \beta_j(f,\Theta)<\infty$ and $\sum _j \beta_j(M,\Theta)<\infty$ since we have
\begin{equation} \label{lipab}
b^{(0)}_j(\Theta)\leq \beta_j(f,\Theta)+ \| \xi_0 \|_p \, \beta_j(M,\Theta)\quad \mbox{for any $j \in \N^*$}.
\end{equation}
Therefore, $(X_t)$ is a stationary and $\L^p$ solution of the causal affine model \eqref{causal} when 
\begin{equation} \label{causalstat}
\sum _{j=1}^\infty \beta_j(f,\Theta)+ \| \xi_0 \|_p \, \beta_j(M,\Theta)<1. 
\end{equation}
In such a case this is interesting to consider $\Phi$ as $(-2)$ times the Gaussian conditional log-density, inducing
\begin{equation}\label{phiGQMLE}
\Phi_G(x,\boldsymbol{\theta} )=\log \big (M_{\boldsymbol{\theta} }^2\big ((x_i)_{i\geq 2}\big )\big )+\frac {\big (x_1-f_{\boldsymbol{\theta} }\big ((x_i)_{i\geq 2}\big ) \big )^2}{M_{\boldsymbol{\theta} }^2\big ((x_i)_{i\geq 2}\big )}. 
\end{equation}
The M-estimator resulting from this contrast is the Gaussian Quasi-Maximum Likelihood estimator (QMLE), notably used for estimating the parameters of GARCH processes, but also  for ARMA, APARCH, ARMA-GARCH processes\dots\\
As this was already done in \cite{BW} (proof of Theorem 1), under identifiability conditions on $f_{\boldsymbol{\theta} }$ and $M_{\boldsymbol{\theta} }$ (see below), Assumption ($Co(\Phi_{G},\Theta)$) holds. 
Moreover,  using Bardet {\it et al.} \cite{BKW} (proof of Lemma 6.3), we obtain with  $p=3$ assuming the existence of $\underline M>0$ such as $M_{\boldsymbol{\theta} } \geq \underline M$ and with $C$ a constant real number
\begin{multline}
\nonumber \sup_{\theta \in \Theta} \big |\Phi_G(U,\boldsymbol{\theta} ) -\Phi_G(V,\boldsymbol{\theta} )\big | \leq \\
C \, \big (1 +|U_1|^2+|V_1|^2+f^2_{\boldsymbol{\theta} }\big ((U_i)_{i\geq 2}\big )+f^2_{\boldsymbol{\theta} }\big ((V_i)_{i\geq 2}\big )  \big ) \\
\nonumber \hspace{-1cm}  \times  \Big ( |U_1-V_1|+\big | f_{\boldsymbol{\theta} }\big ((U_i)_{i\geq 2}\big ) -f_{\boldsymbol{\theta} }\big ((V_i)_{i\geq 2}\big ) \big | +\big | M_{\boldsymbol{\theta} }\big ((U_i)_{i\geq 2}\big ) -M_{\boldsymbol{\theta} }\big ((V_i)_{i\geq 2}\big ) \big | \Big ) \end{multline}
hence, there exists a function $g$ such that
\begin{multline}
\label{lipPhi}   \E \big [ \sup_{\theta \in \Theta} \big |\Phi_G(U,\boldsymbol{\theta} ) -\Phi_G(V,\boldsymbol{\theta} )\big |\big ]\\
 \leq g\big (\sup_{i \geq 1} \big \{ \big \|U_i\|_3 \vee\big \|V_i\|_3 \}\big ) \, \Big (\|U_1-V_1\|_3 +\sum_{i=2}^\infty b_k^{(0)}(\Theta) \, \|U_i-V_i\|_3 \Big ),
\end{multline}
using H\"older inequality and with $b_k^{(0)}(\Theta)=\beta_k(f_{\boldsymbol{\theta}},\Theta)+\| \xi_0\|_p \, \beta_k(M_{\boldsymbol{\theta}},\Theta)$ the Lipschitz coefficients of the function $F_{\boldsymbol{\theta} }$ given in {\bf(A$_0$($\Theta$))}. Therefore, according to \eqref{LipL} with $\alpha_k(\Phi_G, \Theta)= b_k^{(0)}(\Theta) $ for $k\geq 2$ and $\alpha_1(\Phi_G, \Theta)=1 $, we check that $\Phi_G \in \Lip_3(\Theta)$ since {\bf(A$_0$($\Theta$))} holds and $B_0(\Theta)=\sum_k b_k^{(0)}(\Theta)<\infty$.

Now we consider a time varying causal affine processes, that is the local stationary extension of causal affine processes defined in \eqref{causal}, {\it i.e.} 
\begin{equation}\label{causal2}
X^{(n)}_t=M_{\boldsymbol{\theta^{(n)}_t} }\big ((X^{(n)}_{t-i})_{1 \leq i}\big )\, \xi_t + f_{\boldsymbol{\theta^{(n)}_t} }\big ((X^{(n)}_{t-i})_{1 \leq i}\big ), \qquad \mbox{ for any } t\in \Z, 
\end{equation}
with $\boldsymbol{\theta^{(n)}_t} \in \Theta$ a compact set of $\R^d$ and $X^{(n)}_t=0$ for $t \leq 0$. \\
~\\
In the sequel, we will provide general sufficient  conditions for the asymptotic normality of $\widehat{\boldsymbol{\theta}}(u)$ in terms of  the functions $f_{\boldsymbol{\theta}}$ and $M_{\boldsymbol{\theta}}$ and of their derivatives.
\begin{prop}\label{propcausal}
Let $(X^{(n)}_t)$ satisfy \eqref{causal2} where $f_{\boldsymbol{\theta}}$, $M_{\boldsymbol{\theta}}$, $\partial_{\boldsymbol{\theta}}f_{\boldsymbol{\theta} }$, $\partial_{\boldsymbol{\theta}}M_{\boldsymbol{\theta} }$, $\partial^2_{\boldsymbol{\theta}^2}f_{\boldsymbol{\theta} }$ and $\partial^2_{\boldsymbol{\theta}^2}M_{\boldsymbol{\theta} }$ satisfy Lipschitz inequalities \eqref{lipsh} and under Assumption  {\bf (LS($\rho$))}. Assume also:
\begin{enumerate}
\item $\|\xi_0\|_4<\infty$ where the probability distribution of $\xi_0$ is absolutely continuous with respect to the Lebesgue measure and $\Theta$ is a bounded set included in $ \big\{ \boldsymbol{\theta} \in \R^d, ~\sum_{j=1}^\infty \big (\beta_j(f_{\boldsymbol{\theta} },\{ \boldsymbol{\theta} \}) + \|\xi_0\|_4 \,\beta_j(M_{\boldsymbol{\theta} },\{ \boldsymbol{\theta} \})\big ) <1 \big \}$;
\item  there exists $\underline M>0$ such as $M_{\boldsymbol{\theta} } \geq \underline M$ for any $\boldsymbol{\theta} \in \Theta$;
\item For all $\boldsymbol{\theta}, \,\boldsymbol{\theta'} \in \Theta$,
\begin{equation} \label{ident}
\big (f_{\boldsymbol{\theta}}=f_{\boldsymbol{\theta'}}\quad \mbox{and} \quad M_{\boldsymbol{\theta}}=M_{\boldsymbol{\theta'}} \big ) \implies \boldsymbol{\theta}=\boldsymbol{\theta'};
\end{equation}
\item We have
\begin{multline}\label{definite} 
\Big ( \sum_{j=1}^d \mu_j \, \frac {\partial}{\partial {\boldsymbol{\theta}}_j} f_{\boldsymbol{\theta}^*(u)}\big ((\widetilde X_{-k}(u))_{k\in \N} \big )= 0\quad a.s.~\implies ~\mu_j=0, ~j=1,\ldots,d \Big ) ,\\
\mbox{or}\qquad \Big ( \sum_{j=1}^d \mu_j \, \frac {\partial}{\partial {\boldsymbol{\theta}}_j} M_{\boldsymbol{\theta}^*(u)}\big ((\widetilde X_{-k}(u))_{k\in \N} \big )= 0\quad a.s.~\implies ~\mu_j=0, ~j=1,\ldots,d \Big ).
\end{multline}
\end{enumerate}
Consider $\Phi=\Phi_G$ as $(-2)$ times the Gaussian conditional log-density  \eqref{phiGQMLE}. Then, if
\begin{multline}\label{conddPhi}
\sum_{j=1}^\infty j \left(\log j \, \big (\beta_j(f,\Theta)+\beta_j(M,\Theta) \big ) + j \, \big (\beta_j(\partial _{\boldsymbol{\theta}}  f,\Theta)  \right)
\\
+\sum_{j=1}^\infty  \left(\beta_j(\partial _{\boldsymbol{\theta}}  M,\Theta) \big)+\beta_j(\partial^2 _{\boldsymbol{\theta}^2}  f,\Theta) +\beta_j(\partial^2 _{\boldsymbol{\theta}^2}  M,\Theta)\right) <\infty,
\end{multline}
with $\beta_j(\cdot ,\Theta)$ defined in \eqref{lipsh}, the localized QMLE
$\widehat{\boldsymbol{\theta}}(u)$ satisfies the central limit theorem \eqref{tlctheta}.
\end{prop}
Note that the QML contrast $\Phi$ depends on $f_{\boldsymbol{\theta} }$ and $M_{\boldsymbol{\theta} }$. This explains why the asymptotic normality can be obtained from conditions on $f_{\boldsymbol{\theta} }$ and $M_{\boldsymbol{\theta} }$ and their derivatives. Note also that the conditions required in Proposition \ref{propcausal}  are essentially the same than those requested in Theorem 2 of \cite{BW} in the stationary framework. The asymptotic normality of the (localized) QMLE holds under natural conditions, the main difference here is the convergence rates which is  $\sqrt n$ in the stationary case but $\sqrt{nh_n}$  in the non-stationary one, which  follows from localisation. The minimax  rate is $o(n^{1/3})$ is obtained for $\rho=1$ for local stationary causal affine models; it  is  smaller than the usual parametric rate $O(\sqrt n)$ achieved by the QMLE in the stationary case.\\
~\\
In the sequel we will state with details the assumptions for three important specific models, tvAR($\infty$), tvARCH($\infty$) and tvARMA-GARCH models.
\subsubsection{Time varying AR$(\infty)$ and time varying ARMA$(p,q)$ processes} 
In such the case of time varying AR$(\infty)$ (or tvAR$(\infty)$) or time varying (invertible) ARMA$(p,q)$ processes, we have $M_{\boldsymbol{\theta} }=\sigma(\boldsymbol{\theta})>\underline \sigma>0$ and $f_{\boldsymbol{\theta} }((x_i)_{i\geq 1})=\sum_{j=1}^\infty a_j(\boldsymbol{\theta})\, x_j$, where $(a_j(\boldsymbol{\theta}))_{j\geq 1}$ is a sequence of real numbers, implying
\begin{equation}\label{tvAR}
X^{(n)}_t= \sigma(\boldsymbol{\theta}_t^{(n)})\,\xi_t + \sum_{j=1}^\infty a_j(\boldsymbol{\theta}_t^{(n)})\,X^{(n)}_{t-j},\quad\mbox{for $1\leq t \leq n$, $n\in \N^*$},
\end{equation}
with $X^{(n)}_t=0$ for any $t\leq 0$.
Thus, the Lipschitz coefficients satisfy $\beta_j(f,\Theta)=\sup_{\boldsymbol{\theta} \in \Theta}|a_j(\boldsymbol{\theta})|$ and $\beta_j(M,\Theta)=0$. Then we obtain the asymptotic normality of $\widehat{\boldsymbol{\theta}}(u)$ from primitive conditions on functions $a_j$ and $\sigma$ by an application of Proposition \ref{propcausal}:
\begin{cor}
Let $(X_t^{(n)})$ be a tvAR$(\infty)$ process defined in \eqref{tvAR}. 
If $\| \xi_0\|_4<\infty$, let $\Theta$ be a bounded subset of $\R^d$ included in $\big \{\boldsymbol{\theta} \in \R^d, ~  \sum_{j=1}^\infty a_j(\boldsymbol{\theta} )<1 \big \}$. If for each $j \in \N^*$ the functions $\boldsymbol{\theta} \in \Theta \mapsto a_j(\boldsymbol{\theta} ) \in \R$ and $\boldsymbol{\theta} \in \Theta \mapsto \sigma(\boldsymbol{\theta} )\in [\underline \sigma,\infty)$ are ${\cal C}^2(\Theta)$ functions such as $\big ( a_j(\boldsymbol{\theta} ) =a_j(\boldsymbol{\theta}' ),~\forall j\in \N^*$ and $\sigma(\boldsymbol{\theta} )=\sigma(\boldsymbol{\theta}' ) \big )$ imply $(\boldsymbol{\theta} =\boldsymbol{\theta} ')$, if $\boldsymbol{\theta}_t^{(n)}$ satisfies the assumption of local stationarity  {\bf (LS($\rho$))}, and if 
\[ 
\sum_{j=1}^\infty  j \, \log j \, \sup_{\boldsymbol{\theta}\in \Theta}\big |a_j(\boldsymbol{\theta})\big |+ j \, \sup_{\boldsymbol{\theta}\in \Theta} \big |\partial_{\boldsymbol{\theta}} a_j(\boldsymbol{\theta})\big |+ \sup_{\boldsymbol{\theta}\in \Theta} \big |\partial^2_{\boldsymbol{\theta}^2} a_j(\boldsymbol{\theta})\big |<\infty,
\]
then the central limit \eqref{tlctheta} holds for any $u\in (0,1)$ under condition \eqref{condbn}.
\end{cor}
This result is new, essentially because it deals with two difficulties: an infinite memory and also with a non exponential decrease memory. As an illustrative example consider $\boldsymbol{\theta}=(\mu,\kappa,\sigma)'$ and $a_j(\boldsymbol{\theta})=\mu \, j^{-\kappa}$ for any $j\geq 1$, with $\kappa\geq \underline \kappa>2$, $\mu\leq (\sum_{j=1}^\infty j^{-\underline \kappa})^{-1}$ and $\sigma\geq \underline \sigma>0$. Then the previous corollary implies the asymptotic normality \eqref{tlctheta} of $\big (\widehat \mu(u),\widehat \kappa(u),\widehat \sigma(u) \big )$ under condition \eqref{condbn} when $\big (\mu_t^{(n)}, \kappa_t^{(n)},\sigma_t^{(n)} \big )$ satisfies Assumption {\bf (LS($\rho$))}.\\

An important subclass of tvAR$(\infty)$ models is the one of  an invertible tvARMA$(p,q)$ models defined as
\begin{equation}\label{TVARMA}
X^{(n)}_t+\phi^{(n)}_{1,t}\,X^{(n)}_{t-1}+ \cdots + \phi^{(n)}_{p,t}\,X^{(n)}_{t-p}= \sigma^{(n)}_{t}\xi_t +\psi^{(n)}_{1,t}\, \xi_{t-1}+ \cdots +  \psi^{(n)}_{q,t}\, \xi_{t-q}
\end{equation}
(for $1\leq t \leq n$, $n\in \N^*$),
with $X^{(n)}_t=0$ for any $t\leq 0$, as it was introduced in \cite{dahl96}.  We consider the set of parameters $\boldsymbol{\theta}_t^{(n)}=\big ( \phi^{(n)}_{1,t},\ldots,\phi^{(n)}_{p,t},\psi^{(n)}_{1,t},\ldots,\psi^{(n)}_{q,t}, \sigma^{(n)}_{t} \big )'$ and the subset $\Theta_{ARMA}^{(p,q)}$ of $\R^d$ with $d=p+q+1$ defined by:
\begin{multline*}
\Theta_{ARMA}^{(p,q)}=\Big \{ (\phi_1,\ldots,\phi_p,\psi_1,\ldots,\psi_q,\sigma) \in \R^{p+q+1},\\ 1+\phi_1\,  z+\cdots + \phi_p\,  z^p\neq 0~\mbox{and}~1+\psi_1\,  z+\cdots + \psi_q\,  z^q\neq 0~\mbox{for all $|z|\leq 1$} \Big \}.
\end{multline*}
Then if $\boldsymbol{\theta}_t^{(n)} \in \Theta$ for any $1\leq t\leq n,\, n\in \N^*$ and $\Theta$ a compact subset of $\Theta_{ARMA}^{(p,q)}$, then $\sup_{n,t} \|X^{(n)}_t \|_p<\infty$ for any $p\geq 1$ when $\|\xi_0\|_p<\infty$ since $X^{(n)}_t$ can be written as a tvAR$(\infty)$ process \eqref{tvAR} with finite sum of absolute values of coefficients. 
Moreover, from classical analytic arguments it is well known that the corresponding Lipschitz coefficients $\beta_j(f,\Theta)$, $\beta_j(\partial _{\boldsymbol{\theta}}  f,\Theta)$ and $\beta_j(\partial^2 _{\boldsymbol{\theta}^2}  f,\Theta)$  decrease exponentially fast so that the condition \eqref{conddPhi} is automatically satisfied.\\ 

As a consequence of Proposition \ref{propcausal} we obtain:
\begin{cor}
If $(X_t^{(n)})$ is a tvARMA$(p,q)$ process defined in \eqref{TVARMA}, $\Theta$ is a bounded subset of $\Theta_{ARMA}^{(p,q)}$, if $\| \xi_0\|_4<\infty$ and $\boldsymbol{\theta}_t^{(n)}$ satisfies the assumption of local stationarity  {\bf (LS($\rho$))}, which is implied by the local stationarity {\bf (LS($\rho$))} of all functions  
$\phi^{(n)}_{1,t},\ldots,\phi^{(n)}_{p,t}$, $\psi^{(n)}_{1,t},\ldots,\psi^{(n)}_{q,t}, \sigma^{(n)}_{t}$, then the central limit \eqref{tlctheta} holds for any $u\in (0,1)$ under condition \eqref{condbn}.
\end{cor}
This result is a QMLE version of the results obtained by Dahlhaus in \cite{dahl00} for Gaussian tvARMA processes (using Whittle likelihood approximation) and by Azrak and M\'elard in \cite{am2006}; we use quasi likelihood contrasts similarly as those authors.

\subsubsection{Time varying ARCH$(\infty)$ and time varying GARCH$(p,q)$ processes} 
Time varying ARCH$(\infty)$ (tvARCH$(\infty)$) or time varying GARCH$(p,q)$ (tvGARCH$(p,q)$) processes correspond to $f_{\boldsymbol{\theta} }((x_i)_{i\geq 1})=0$ and $M_{\boldsymbol{\theta} }((x_i)_{i\geq 1})=\big ( a_0(\boldsymbol{\theta})+\sum_{j=1}^\infty a_j(\boldsymbol{\theta})\, x_j^2 \big )^{1/2}$, where $(a_j(\boldsymbol{\theta}))_{j\geq 0}$ is a sequence of non negative real numbers and $a_0(\cdot) \geq \underline a>0$ implying
\begin{equation}\label{tvARCH}
X^{(n)}_t=\xi_t \, \Big ( a_0(\boldsymbol{\theta}^{(n)}_t ) + \sum_{i\geq 1} a_i(\boldsymbol{\theta}^{(n)}_t )\big (X^{(n)}_{t-i} \big )^2 \Big )^{1/2} \quad\mbox{for $1\leq t \leq n$, $n\in \N^*$},
\end{equation}
with $X^{(n)}_t=0$ for any $t\leq 0$, $\boldsymbol{\theta}^{(n)}_t \in \R^d$ for any $1\leq t \leq n$, $n\in \N^*$ satisfying Assumption {\bf (LS($\rho$))}. For more details about stationary ARCH$(\infty)$ or GARCH$(p,q)$ processes, or for the transition from GARCH$(p,q)$ to ARCH$(\infty)$, see \cite{fz}. We are going to specify again the conditions of Proposition \ref{propcausal} in such a case. Firstly,  we consider Lipschitz properties on $\big ((X^{(n)}_t)^2\big )_t$ rather than $(X^{(n)}_t)_t$ as in \cite{BW} in order to recover natural constraint on the parameters set $\Theta$ (similar as the moments condition of \cite{gkl})
\begin{equation}\label{ThetaARCH}
\Theta~\mbox{is a compact subset of}~ \Big \{\boldsymbol{\theta}\in \R^d,~\| \xi_0\|^2_4 \, \sum_{j=1}^\infty a_j(\boldsymbol{\theta})<1 \Big \}\,.
\end{equation}
Secondly, the Lipschitz coefficients of $\Phi$, $\partial_{\boldsymbol{\theta}}  \Phi$ and $\partial^2_{\boldsymbol{\theta}^2}  \Phi$ can be expressed in terms of $|U_i^2-V_i^2|$ following the same computations than in \eqref{lipPhi} and in the proof of Proposition \ref{propcausal}. But since  $|U_i^2-V_i^2|=|U_i-V_i| \, |U_i+V_i|$ and each time $\big (M_{\boldsymbol{\theta} }((U_i)_{i\geq 1}) \times M_{\boldsymbol{\theta} }((V_i)_{i\geq 1})\big )^{-1}$ appears in the function $g$, we deduce that $\Phi$, $\partial_{\boldsymbol{\theta}}  \Phi$ and $\partial^2_{\boldsymbol{\theta}^2}  \Phi$ are respectively included in $\Lip_3(\Theta)$, $\Lip_4(\Theta)$ and $\Lip_4(\Theta)$ with coefficients $\alpha_s(\cdot,\Theta)$ defined in \eqref{LipL} satisfying for $s\geq 2$,
$$
\left\{ \begin{array}{lcl}
\alpha_s(\Phi,\Theta)&=&\sup_{\boldsymbol{\theta} \in \Theta} \big |a_s(\boldsymbol{\theta}) \big |\\
\alpha_s(\partial_{\boldsymbol{\theta}} \Phi,\Theta)&=&\sup_{\boldsymbol{\theta} \in \Theta} \big ( \big |a_s(\boldsymbol{\theta}) \big |+\big |\partial_{\boldsymbol{\theta}} a_s(\boldsymbol{\theta}) \big |\big ) \\
\alpha_s(\partial^2_{\boldsymbol{\theta}^2} \Phi,\Theta)&=&\sup_{\boldsymbol{\theta} \in \Theta} \big ( \big |a_s(\boldsymbol{\theta}) \big |+\big |\partial_{\boldsymbol{\theta}} a_s(\boldsymbol{\theta}) \big |+\big |\partial^2_{\boldsymbol{\theta}^2} a_s(\boldsymbol{\theta}) \big |\big )
\end{array}
\right . .
$$
From an application of Proposition \ref{propcausal} we obtain the asymptotic normality of $\widehat{\boldsymbol{\theta}}(u)$ from primitive conditions on functions $a_j$:
\begin{cor}
Let $(X_t^{(n)})$ be a tvARCH$(\infty)$ process defined in \eqref{tvARCH}. We assume that $\| \xi_0\|_4<\infty$ 
and we consider $\Theta$ a compact subset of  $\Big \{\boldsymbol{\theta}\in \R^d,~\| \xi_0\|^2_4 \, \sum_{j=1}^\infty a_j(\boldsymbol{\theta})<1 \Big \}$. 
If for $j \in \N^*$ the functions $\boldsymbol{\theta} \in \Theta \mapsto a_j(\boldsymbol{\theta} ) \in [0,\infty)$ and
 $\boldsymbol{\theta} \in \Theta \mapsto a_0(\boldsymbol{\theta} )\in [\underline a,\infty)$ are ${\cal C}^2(\Theta)$ functions such that $\big ( a_j(\boldsymbol{\theta} ) =a_j(\boldsymbol{\theta}' ),~\forall j\in \N \big )$ implies $(\boldsymbol{\theta} =\boldsymbol{\theta} ')$, if $\boldsymbol{\theta}_t^{(n)}$ satisfies the assumption of local stationarity  {\bf (LS($\rho$))}, and if 
$$
\sum_{j=1}^\infty  j \, \log j \, \sup_{\boldsymbol{\theta}\in \Theta}\big |a_j(\boldsymbol{\theta})\big |+
 j \, \sup_{\boldsymbol{\theta}\in \Theta} \big |\partial_{\boldsymbol{\theta}} a_j(\boldsymbol{\theta})\big |+
  \sup_{\boldsymbol{\theta}\in \Theta} \big |\partial^2_{\boldsymbol{\theta}^2} a_j(\boldsymbol{\theta})\big |<\infty,
$$
then the localized QMLE $\widehat{\boldsymbol{\theta}}(u)$ is asymptotically normal and \eqref{tlctheta} holds for any $u\in (0,1)$ and any $(h_n)$ satisfying  \eqref{condbn}.
\end{cor}
To the best of our knowledge, this result is new. In  \cite{dsr06}  the existence of tvARCH$(\infty)$ processes has been studied and the asymptotic normality has been obtained for tvARCH$(p)$ processes.\\

We specialized the previous result to the cases where  $(X^{(n)}_t)$ is a tvGARCH$(p,q)$ process. We assume that $\|\xi_0\|_4<\infty$ and we consider the model
\begin{equation}\label{TVGARCH}
\left \{ \begin{array}{ccl}
X^{(n)}_t&=& \sigma^{(n)}_{t} \,\xi_t \\
\big (\sigma^{(n)}_{t} \big )^2 & = & c^{(n)}_{0,t}+ c^{(n)}_{1,t}\,\big (X^{(n)}_{t-1}\big )^2 + \cdots + c^{(n)}_{p,t}\,\big (X^{(n)}_{t-p}\big )^2+  d^{(n)}_{1,t}\,\big (\sigma^{(n)}_{t-1}\big )^2 + \cdots + d^{(n)}_{q,t}\,\big (\sigma^{(n)}_{t-q}\big )^2
\end{array} \right .
\end{equation}
where $(c^{(n)}_{i,t})_{0\leq i \leq p}$ and $(d^{(n)}_{j,t})_{1\leq j \leq q}$ are non negative real number for any $1\leq t \leq n$, $n\in \N^*$, with $X^{(n)}_t=0$ for any $t\leq 0$.  Consider $\boldsymbol{\theta}_t^{(n)}=\big (c^{(n)}_{0,t}, c^{(n)}_{1,t},\ldots,c^{(n)}_{p,t},d^{(n)}_{1,t},\ldots,d^{(n)}_{q,t} \big )'$ with $c^{(n)}_{0,t}\geq \underline c >0$ and the subset $\Theta_{GARCH}^{(p,q)}$ of $\R^{p+q+1}$ defined by:
\begin{equation*}
\Theta_{GARCH}^{(p,q)}=\Big \{ (c_0,c_1,\ldots,c_p,d_1,\ldots,d_q) \in \R^{p+q+1},~ \sum_{j=1}^q d_j+ \| \xi_0 \|^2_4 \,\sum_{i=1}^p c_i<1  \Big \}.
\end{equation*}
If $\boldsymbol{\theta}_t^{(n)} \in \Theta$ for any $1\leq t\leq n,\, n\in \N^*$ with $\Theta$ a bounded set included in $\Theta_{GARCH}^{(p,q)}$ then we have $\sup_{t,n}\|X_{t}^{(n)}\|_4 <\infty$. Moreover, $(X_{t}^{(n)})_t$ can be written as a tvARCH$(\infty)$ (see \cite{fz}). Moreover the coefficients $a_i(\boldsymbol{\theta}^{(n)}_t )$ in the tvARCH$(\infty)$  decrease exponentially fast. As $\alpha_s(\Phi,\Theta)$, $
\alpha_s(\partial_{\boldsymbol{\theta}} \Phi,\Theta)$ and $\alpha_s(\partial^2_{\boldsymbol{\theta}^2} \Phi,\Theta)$ can be expressed from  $a_s(\cdot )$ and their derivatives, which are also exponentially decreasing, this implies the following corollary:
\begin{cor}
Let $(X_t^{(n)})$ be a tvGARCH$(p,q)$ process defined in \eqref{TVGARCH} and $\Theta$ be a bounded set included in $\Theta_{GARCH}^{(p,q)}$ where $\| \xi_0\|_4<\infty$. If $\boldsymbol{\theta}_t^{(n)}$ satisfies the assumption of local stationarity  {\bf (LS($\rho$))}, which is implied by the local stationarity {\bf (LS($\rho$))} on all functions  $c^{(n)}_{0,t},\ldots,c^{(n)}_{p,t},d^{(n)}_{1,t},\ldots,d^{(n)}_{q,t}$, then the central limit \eqref{tlctheta} holds for any $u\in (0,1)$ under condition \eqref{condbn}.
\end{cor}
This result can be compared for instance with those of \cite{dsr06} for tvARCH$(p)$, which are obtained under the same procedure but under the condition  $\| \xi_0\|_{4(1+\delta)}<\infty$, or those of  \cite{ro} for tvGARCH$(p,q)$, which are obtained from a local polynomial estimation and under the condition  $\| \xi_0\|_8<\infty$. Note that  \cite{truq17} also obtained asymptotic normality under very sharp conditions in a special case of tvARCH$(p)$ process.

\subsubsection{Time varying ARMA$(p,q)$-GARCH$(p',q')$ processes}
This model was introduced in the stationary framework by \cite{Ding1993} and developed in \cite{Ling2003}. The model consists on a tvARMA$(p,q)$ where the pure white noise $(\xi_t)$ is replaced by a weak white noise $(\varepsilon^{(n)}_t)$ that is a tvGARCH$(p',q')$ process, {\it i.e.} $(X^{(n)}_t)$ is defined by
\begin{equation}\label{TVARMAGARCH}
\left \{ \begin{array}{ccl}
X^{(n)}_t& = & -\phi^{(n)}_{1,t}\,X^{(n)}_{t-1}- \cdots -\phi^{(n)}_{p,t}\,X^{(n)}_{t-p}+ \varepsilon^{(n)}_t +\psi^{(n)}_{1,t}\, \varepsilon^{(n)}_{t-1}+ \cdots +  \psi^{(n)}_{q,t}\, \varepsilon^{(n)}_{t-q} \\
\varepsilon^{(n)}_t&=& \sigma^{(n)}_{t} \,\xi_t \\
\big (\sigma^{(n)}_{t} \big )^2 & = & c^{(n)}_{0,t}+ c^{(n)}_{1,t}\,\big (\varepsilon^{(n)}_{t-1}\big )^2 + \cdots + c^{(n)}_{p',t}\,\big (\varepsilon^{(n)}_{t-p'}\big )^2+  d^{(n)}_{1,t}\,\big (\sigma^{(n)}_{t-1}\big )^2 + \cdots + d^{(n)}_{q',t}\,\big (\sigma^{(n)}_{t-q'}\big )^2
\end{array} \right .
\end{equation} 
for any $1\leq t \leq n$, $n\in \N^*$, with $X^{(n)}_t=0$ for any $t\leq 0$. As previously, $c^{(n)}_{0,t} \geq \underline c >0$ and $(c^{(n)}_{j,t})$ is a family of non-negative real numbers. Consider
$$
\boldsymbol{\theta}_t^{(n)}=\big (\phi^{(n)}_{1,t},\ldots,\phi^{(n)}_{p,t},\psi^{(n)}_{1,t},\ldots,\psi^{(n)}_{q,t} , c^{(n)}_{0,t}, c^{(n)}_{1,t},\ldots,c^{(n)}_{p',t},d^{(n)}_{1,t},\ldots,d^{(n)}_{q',t} \big )'.
$$
If $\| \xi_0\|_4<\infty$, define the subset $\Theta_{ARMAGARCH}^{(p,q,p',q')}$ of $\R^d$ with $d={p+q+p'+q'+1}$ by 
\begin{multline*}
\Theta_{ARMAGARCH}^{(p,q,p',q')}=\Big \{\boldsymbol{\theta}\in \R^d,~\sum_{j=1}^{q'} d_j+\| \xi_0|^2_4 \, \sum_{j=1}^{p'} c_j<1,\\ ~\mbox{and }\ ~\big (1+\sum_{j=1}^{p} \phi_j z^j\big ) \, \big (1+\sum_{j=1}^{q} \psi_j z^j\big ) \neq 0~\mbox{for all $|z|\leq 1$} \Big \}.
\end{multline*}
Then, if $\Theta$ is a bounded set included in $\Theta_{ARMAGARCH}^{(p,q,p',q')}$ then the GARCH$(p',q')$ process $(\varepsilon^{(n)}_t)_t$ satisfies $\sup_{t,n} \| \varepsilon^{(n)}_t \|_4<\infty$ (see previously) when $\boldsymbol{\theta}_t^{(n)}\in \Theta$ for any $1\leq t \leq n, \, n \in \N^*$. Moreover, $X_t^{(n)}$ can be written as a linear filter of $(\varepsilon^{(n)}_t)_t$ when the ARMA coefficients satisfies the condition required in $\Theta_{ARMAGARCH}^{(p,q,p',q')}$, and these coefficients decrease exponentially fast. Therefore, when $\boldsymbol{\theta}_t^{(n)}\in \Theta$ for any $1\leq t \leq n, \, n \in \N^*$ with $\Theta$ a bounded set included in $\Theta_{ARMAGARCH}^{(p,q,p',q')}$, then $\sup_{t,n} \|X_t^{(n)}\|_4 <\infty$. \\
Moreover, following Lemma 2.1. of \cite{bb2017}, we know that a stationary ARMA$(p,q)$-GARCH$(p',q')$ process is a stationary affine causal process with functions $f_{\boldsymbol{\theta}}$ and $M_{\boldsymbol{\theta}}$ satisfying the Lipschitz condition \eqref{lipsh} with Lipschitz  coefficients decreasing exponentially fast, as well as their derivatives. This is also the same case for a time varying ARMA$(p,q)$-GARCH$(p',q')$ process. Therefore, we obtain the following result:
\begin{cor}
Let $(X_t^{(n)})$ be a time varying ARMA$(p,q)$-GARCH$(p',q')$ process defined in \eqref{TVARMAGARCH} with $\|\xi_0\|_4<\infty$ and $\Theta$ be a bounded set included in $\Theta_{ARMAGARCH}^{(p,q,p',q')}$. 
Moreover, if $\boldsymbol{\theta}_t^{(n)}$ satisfies the assumption of local stationarity  {\bf (LS($\rho$))}, which is implied by the same local stationarity property satisfied by $\phi^{(n)}_{1,t},\ldots,\phi^{(n)}_{p,t}$, $\psi^{(n)}_{1,t},\ldots,\psi^{(n)}_{q,t} $, $ c^{(n)}_{0,t}, c^{(n)}_{1,t},\ldots,c^{(n)}_{p',t}$, $d^{(n)}_{1,t},\ldots,d^{(n)}_{q',t}$, 
then the localized QMLE is asymptotically normal as \eqref{tlctheta} holds for any $u\in (0,1)$  and $(h_n)$ satisfying \eqref{condbn}.
\end{cor}

\subsection{Time varying LARCH$(\infty)$ processes and LS-contrast}   \label{LARCH}
Here we consider a LARCH$(\infty)$ process introduced by Robinson in \cite{r91} and studied intensively by Giraitis {\it et al.} in \cite{glrs}. The model is defined as
\begin{equation}\label{larch}
X_t=\xi_t \, \Big(a_0(\boldsymbol{\theta})+\sum_{j=1}^{\infty}a_j(\boldsymbol{\theta}) \, X_{t-j}\Big)\quad\mbox{for any $t\in \Z$},
\end{equation}
where $\boldsymbol{\theta}\in \R^d$ and assume $\| \xi_0\|_2=1$. Assume also that $j\in \N$, $\boldsymbol{\theta} \in \R^d \mapsto a_j(\boldsymbol{\theta}) \in \R$ are continuous functions and without lose of generality assume $a_0(\boldsymbol{\theta})\ge 0$ for any $\boldsymbol{\theta}\in \R^d$. Moreover, for ensuring the stationarity of $(X_t)$ and the existence of $\|X_t\|_r$ with $r\geq 1$, assume that for any $\boldsymbol{\theta}\in \Theta$,
\begin{equation}\label{statioLARCH}
\| \xi_0\|_r \, \sum_{j=1}^{\infty} |a_j(\boldsymbol{\theta})| <1.
\end{equation} 
Even if a LARCH$(\infty)$ process is an affine causal process, the Gaussian QML contrast can not be used for estimating $\boldsymbol{\theta}$. Indeed, the conditional variance of $X_t$ can not be bounded close to $0$ and this does not allow asymptotic results for such contrasts (see more details in Francq and Zako\"ian  \cite{FZ10}). Even if weighted least square estimators can also be defined (see \cite{FZ10}), we consider here the following ordinary LS contrast of square values: for $x \in \R^\infty$, define
\begin{equation}\label{phiLARCH}
\Phi_{LARCH}(x,\boldsymbol{\theta})=\Big (x^2_1-\big(a_0(\boldsymbol{\theta})+\sum_{j=1}^{\infty}a_j(\boldsymbol{\theta}) \, x_{j+1}\big)^2\Big )^2.
\end{equation}
If the stationary version $(\widetilde X_t(u))$ were observed, for any $\boldsymbol{\theta} \in \Theta$ the score associated to the LS-contrast is
\begin{multline*}
\E \big [\Phi_{LARCH}((\widetilde X_{1-k}(u))_{k\geq 0},\boldsymbol{\theta}) ~| ~{\cal F}_0 \big ]= \E \big [ |\xi_1|^4-1 \big ] \, \Big (a_0(\boldsymbol{\theta}^*)+\sum_{j=1}^{\infty}a_j(\boldsymbol{\theta}^*(u)) \, \widetilde X_{1-j}(u)\Big)^4 \\
 + \Big ( \big ( a_0(\boldsymbol{\theta}^*(u))+\sum_{j=1}^{\infty}a_j(\boldsymbol{\theta}^*(u)) \, \widetilde X_{1-j}(u) \big )^2-\big ( a_0(\boldsymbol{\theta})+\sum_{j=1}^{\infty}a_j(\boldsymbol{\theta}) \,\widetilde X_{1-j}(u) \big )^2 \Big )^2.
\end{multline*}
We notice that the first term at the right side of the last equality does not depend on $\boldsymbol{\theta}$. Then since $a_0(\cdot)$ is supposed to be non negative,   if we assume 
\begin{equation}\label{larchcond}
\Big (a_0(\boldsymbol{\theta})+\sum_{j=1}^{\infty}a_j(\boldsymbol{\theta}) \, X_{1-j} =a_0(\boldsymbol{\theta}')+\sum_{j=1}^{\infty}a_j(\boldsymbol{\theta}) \, X_{1-j} ~a.s.\Big ) ~\implies ~\boldsymbol{\theta}=\boldsymbol{\theta}',
\end{equation}
then $\E \big [\Phi_{LARCH}((X_{1-k})_{k\geq 0},\boldsymbol{\theta}) ~| ~{\cal F}_0 \big ]$ has a unique minimum  that is $\boldsymbol{\theta}^*$ and Assumption ($Co(\Phi_{LARCH},\Theta)$) holds. Moreover, after computations and use of H\"older Inequalities, if $r=4$,
\begin{multline*}
 \sup_{\theta \in \Theta} \big |\Phi_{LARCH}(U,\boldsymbol{\theta} ) -\Phi_{LARCH}(V,\boldsymbol{\theta} )\big |\\ \leq 
 \Big (U^2_1+V^2_1+\big(a_0(\boldsymbol{\theta})+\sum_{j=1}^{\infty}a_j(\boldsymbol{\theta}) \, U_{j+1}\big)^2+\big(a_0(\boldsymbol{\theta})+\sum_{j=1}^{\infty}a_j(\boldsymbol{\theta}) \, V_{j+1}\big)^2\Big )
\\
 \times  \Big (|U_1+V_1|\, |U_1-V_1|+\Big |2a_0(\boldsymbol{\theta})+\sum_{j=1}^{\infty}a_j(\boldsymbol{\theta}) \, (U_{j+1} +V_{j+1})\Big |\, \sum_{j=1}^{\infty}|a_j(\boldsymbol{\theta}) |\, |U_{j+1} -V_{j+1}| \Big ) 
\end{multline*}
Hence
\begin{multline*}
  \E \big [ \sup_{\theta \in \Theta} \big |\Phi_L(U,\boldsymbol{\theta} ) -\Phi_L(V,\boldsymbol{\theta} )\big | \big ]  \leq g\big (\sup_{i \geq 1} \big \{ \big \|U_i\|_4 \vee \big \|V_i\|_4 \}\big ) \\
 \hspace{2cm}\times \Big ( \|U_1-V_1\|_4+ \sum_{j=2}^\infty \sup_{\boldsymbol{\theta} \in \Theta}|a_{j-1}(\boldsymbol{\theta})| \, \|U_j-V_j\|_4 \Big ),
\end{multline*}
and therefore $\Phi_{LARCH} \in \Lip_4(\Theta)$ with $\alpha_1(\Phi_{LARCH}, \Theta)=1 $ , and $\alpha_k(\Phi_{LARCH}, \Theta)= \sup_{\boldsymbol{\theta} \in \Theta}|a_{k-1}(\boldsymbol{\theta})| $, for $k\geq 2$ and $\sum_k  \alpha_k(\Phi_{LARCH}, \Theta)<\infty$, from \eqref{statioLARCH}.\\

We consider now the time varying LARCH($\infty$) process defined by:
\begin{equation}\label{tvLARCH}
X^{(n)}_t=\xi_t \, \Big ( a_0(\boldsymbol{\theta^{(n)}_t})+ \sum_{i=1}^\infty a_i(\boldsymbol{\theta^{(n)}_t})\, X^{(n)}_{t-i}\Big ), \qquad \mbox{ for any } t\in \Z, 
\end{equation}
with $\boldsymbol{\theta^{(n)}_t} \in \Theta$ a compact set of $\R^d$ and $X^{(n)}_t=0$ for $t \leq 0$. We also assume that $a_0(\cdot)$ is a non negative function. An application of Theorem \ref{theo1} implies that the localized LS estimator is uniformly consistent when $r=4$.\\
To assert the asymptotic normality, we assume $\|\xi_0\|_8<\infty$ and $\Theta$ is a bounded subset of the set 
\begin{equation}\label{thetaLARCH}
\Theta_{LARCH}=\Big \{ \boldsymbol{\theta} \in \R^d, ~\|\xi_0\|_8 \, \sum_{i=1} ^\infty |a_i(\boldsymbol{\theta}) | <\infty  \Big \}.
\end{equation}
Then, using classical computations and Hausdorff Inequalities (see the proof), we obtain the asymptotic behavior of the estimator:
\begin{prop}\label{propLARCH}
Assume that $\boldsymbol{\theta}\in \R ^d \mapsto a_j(\boldsymbol{\theta}) \in \R$ are ${\cal C}^2$ functions for any $j\in \R$, $a_0(\cdot)\geq 0$ and
\begin{enumerate}
\item $\|\xi_0\|_8<\infty$ where the probability distribution of $\xi_0$ is absolutely continuous with respect to the Lebesgue measure and $\Theta$ is a bounded set included in $\Theta_{LARCH}$;
\item For all $\boldsymbol{\theta}, \,\boldsymbol{\theta'} \in \Theta$,
 \begin{equation}\label{identLARCH}
(a_i(\boldsymbol{\theta})=a_i(\boldsymbol{\theta}'),~ \mbox{for all $i \in \N$})~\implies ~(\boldsymbol{\theta}=\boldsymbol{\theta'} );
\end{equation}
\item For all $\boldsymbol{\theta}, \,\boldsymbol{\theta'} \in \Theta$,
 \begin{equation}\label{sigmaLARCH}
(\partial _{\boldsymbol{\theta}}a_i(\boldsymbol{\theta})=\partial _{\boldsymbol{\theta}}a_i(\boldsymbol{\theta}'),~ \mbox{for all $i \in \N$})~\implies ~(\boldsymbol{\theta}=\boldsymbol{\theta'} ).
\end{equation}
\end{enumerate}
Let $(X^{(n)}_t)$ be a tvLARCH process defined following  \eqref{tvLARCH} where $\boldsymbol{\theta}_t^{(n)}$ satisfies Assumption  {\bf (LS($\rho$))}. 
Consider $\Phi=\Phi_{LARCH}$ as  in \eqref{phiLARCH}. If
\begin{equation}\label{conddPhiL}
\sum_{j=1}^\infty j \,\log j \, \sup_{\boldsymbol{\theta} \in \Theta} |a_j(\boldsymbol{\theta}|+ j \, \sup_{\boldsymbol{\theta} \in \Theta} \|\partial _{\boldsymbol{\theta}} a_j(\boldsymbol{\theta}\| +\sup_{\boldsymbol{\theta} \in \Theta} \|\partial^2 _{\boldsymbol{\theta}^2} a_j(\boldsymbol{\theta}\| <\infty,
\end{equation}
then $\widehat{\boldsymbol{\theta}}(u)$ is asymptotically normal as \eqref{tlctheta} for any $u\in (0,1)$ and $(h_n)$ satisfying \eqref{condbn}.
\end{prop}
To our knowledge, this result is new, even in its stationary $\boldsymbol{\theta}_t^{(n)}=\boldsymbol{\theta}^* \in \R^d$ version. The particular case of time varying GLARCH$(p,q)$ process, natural extension of stationary GLARCH$(p,q)$ processes (see for instance \cite{glrs}) is also interesting and straightforward:
\begin{cor}
If $\|\xi_0\|_8<\infty$ and the probability distribution of $\xi_0$ is absolutely continuous with respect to the Lebesgue measure and if $(X_t^{(n)})$ is a tvGLARCH$(p,q)$ process defined by
\begin{equation*}
X^{(n)}_t=\xi_t \,\sigma_t^{(n)} \quad 
\mbox{with}  \quad \sigma_t^{(n)}= c_{0,t}^{(n)}+ \sum_{i=1}^p c_{i,t}^{(n)}\, X^{(n)}_{t-i}+\sum_{j=1}^q d_{j,t}^{(n)}\, \sigma_{t-j}^{(n)}, \qquad \mbox{ for any } t \in \N^*, 
\end{equation*}
with $X^{(n)}_t=0$ for $t \leq 0$ and where $\boldsymbol{\theta}_t^{(n)}=(c_{0,t}^{(n)},\ldots,c_{i,t}^{(n)},d_{j,t}^{(n)},\ldots,d_{j,t}^{(n)}) \in \Theta$, with $\Theta$ a bounded set in $$\big \{ (c_0,c_1,\ldots,c_p,d_1,\ldots,d_q) \in [0,\infty)\times \R^{p+q},~\sum_{i=1}^q |d_i|+ \|\xi_0\|_8 \sum_{i=1}^p |c_i|<1 \big \} ,$$ which satisfies the assumption of local stationarity  {\bf (LS($\rho$))}. Then the central limit \eqref{tlctheta} holds for any $u\in (0,1)$ under condition \eqref{condbn}.
\end{cor}
This result is due to the exponential decay of the sequences $(\alpha_s(\Phi_{LARCH}, \Theta))_s$ and $(\alpha_s(\partial _{\boldsymbol{\theta}} \Phi_{LARCH}, \Theta))_s$ in such as case (see \cite{glrs}). 

\subsection{Time varying integer valued processes and Poisson QMLE}   \label{Poisson}
Finally, we consider the integer valued process $(X_t)$ defined so that the conditional distribution of $X_t$ is a Poisson distribution with parameter $\lambda_{\boldsymbol{\theta}}\big ((X_{t-i})_{i\geq 1}\big )$, {\it i.e.}
\begin{equation}\label{PoissonGen}
X_t ~ | ~((X_{t-i})_{i\geq 1})~\egaleloi {\cal P} (\lambda_{\boldsymbol{\theta}}\big ((X_{t-i})_{i\geq 1}\big ) \big ), \quad \mbox{for any $t \in \Z$},
\end{equation}
where for any $\boldsymbol{\theta} \in \Theta$, $U\in \R^\infty \mapsto \lambda_{\boldsymbol{\theta}}(U ) \in [\underline \lambda,\infty)$, $\underline \lambda>0$,  is once again a Lipschitz function on $\Theta$ satisfying \eqref{lipsh} with Lipschitz coefficients $(\beta_i(\lambda_{\boldsymbol{\theta}},\Theta))_{i\geq 1}$ such that  $\sum_{i=1}^\infty \beta_i(\lambda_{\boldsymbol{\theta}},\Theta)<1$  (see for instance Doukhan and Kengne \cite{dk2015}). \\ 
Then, we consider as a contrast the opposite of the log-likelihood of the process, {\it i.e.}
\begin{equation}\label{phiPoisson}
\Phi_P(x,\theta)=-x_1 \,\log \big (\lambda_{\boldsymbol{\theta}}\big ((x_i)_{i\geq 2}\big)\big )+\lambda_{\boldsymbol{\theta}} \big ((x_i)_{i\geq 2}\big). 
\end{equation}
Then, after classical computations we obtain for $p= 2$:
\begin{multline*}
\E \big [\sup_{\boldsymbol{\theta}  \in \Theta} \big |\Phi_P(U,\boldsymbol{\theta} ) -\Phi_P(V,\boldsymbol{\theta} )\big | \big ] \\
\leq  C \Big ( \sup_{i \geq 1} \big \{ \big \|U_i\|_2 \vee \big \|V_i\|_2 \} \, \|U_1-V_1\|_2+(1+\|V_1\|_2^2 )^{1/2} \sum_{i=2}^\infty \beta_i(\lambda_{\boldsymbol{\theta}},\Theta)\, \|U_i-V_i\|_2 \Big ) \\ 
\leq g\big (\sup_{i \geq 1} \big \{ \big \|U_i\|_2 \vee \big \|V_i\|_2 \}\big ) \, \sum_{i=1}^\infty \alpha_i(\Phi, \Theta) \, \|U_i-V_i\|_2,
\end{multline*}
with $\alpha_1(\Phi_P, \Theta)=1$ and $\alpha_i(\Phi_P, \Theta)=\beta_{i-1}(\lambda_{\boldsymbol{\theta}},\Theta)$ for $i\geq 2$, inducing $\Phi_P\in \Lip_2(\Theta)$.\\

We extend the structural recursive equation \eqref{PoissonGen} and $\Phi_P$ to the time varying framework as follows. This example shows that our results   apply to integer valued local stationary processes, which is an original and interesting extension. Hence, we consider
\begin{equation}\label{tvPoisson}
X_t^{(n)} ~ | ~(X^{(n)}_{t-i})_{i\geq 1}~\egaleloi {\cal P}\Big (\lambda_{\boldsymbol{\theta}_t^{(n)}}\big ((X^{(n)}_{t-i})_{i\geq 1}\big ) \Big ), \quad \mbox{for any $1\leq t \leq n$, and all $n\in \N^*$},
\end{equation}
where $X_t^{(n)}=0$ for $t\leq 0$.  If $\lambda_{\boldsymbol{\theta}}(\cdot)$ satisfies the uniform Lipschitz property \eqref{lipsh} with Lipschitz coefficients $(\beta_i(\lambda,\Theta))_{i\geq 1}$, we define
\begin{equation}\label{thetap}
\Theta_{{\cal P}}=\Big \{\boldsymbol{\theta} \in \R^d,~\sum_{i=1}^\infty \beta_i(\lambda,\{ \boldsymbol{\theta}\})<1\Big \}.
\end{equation}
Then, using Theorem 2.1 of Doukhan {\it et al.} \cite{dft} and Lemma \ref{momm}, we deduce that if $\boldsymbol{\theta} \in \Theta_{{\cal P}}$ for any $1\leq t \leq n$ and $n\in \N^*$, then $\sup_{t,n} \|X_t^{(n)}\|_p <\infty$, for any $p\geq 1$. Using the Poisson QMLE defined by \eqref{phiPoisson}, we obtain the following asymptotic result:
\begin{prop}\label{propPoisson}
Let $(X^{(n)}_t)$ satisfy \eqref{causal2} where $\lambda_{\boldsymbol{\theta}}$,  $\partial_{\boldsymbol{\theta}}\lambda_{\boldsymbol{\theta} }$ and $\partial^2_{\boldsymbol{\theta}^2}\lambda_{\boldsymbol{\theta} }$ satisfy Lipschitz inequalities \eqref{lipsh} and under Assumption  {\bf (LS($\rho$))}. Assume also:
\begin{enumerate}
\item $\Theta$ is a bounded set included in $ \Theta_{{\cal P}}$;
\item  there exists $\underline \lambda>0$ such as $\lambda_{\boldsymbol{\theta} } \geq \underline \lambda$ for any $\boldsymbol{\theta} \in \Theta$;
\item For all $\boldsymbol{\theta}, \,\boldsymbol{\theta'} \in \Theta$,
 $\lambda_{\boldsymbol{\theta}}=\lambda_{\boldsymbol{\theta'}}$ implies $\boldsymbol{\theta}=\boldsymbol{\theta'} $;
\item For any $i=1,\ldots,d$, $\frac \partial {\partial {\boldsymbol{\theta}_i}}\, \lambda_{\boldsymbol{\theta}^*(u)}\big ((\widetilde X_{-k}(u))_{k\in \N} \big )\neq 0$ a.s.
\end{enumerate}
Consider $\Phi=\Phi_P$ as it was defined in \eqref{phiPoisson}. Then, if
\begin{equation}\label{conddPhiP}
\sum_{j=1}^\infty (j \,\log j) \, \beta_j(\lambda,\Theta)+  j \, \beta_j(\partial _{\boldsymbol{\theta}}  \lambda,\Theta) +\beta_j(\partial^2 _{\boldsymbol{\theta}^2}  \lambda,\Theta)  <\infty,
\end{equation}
for any $u\in (0,1)$ and under condition \eqref{condbn}.
\begin{equation}\label{tlcingarch}
 \sqrt{n\, h_n}\big ( \widehat{\boldsymbol{\theta}}(u)- \boldsymbol{\theta}^*(u) \big ) \limiteloin  {\cal N} \Big ( 0 , \\E \Big [\frac {\partial_{\boldsymbol{\theta}} \lambda_{\boldsymbol{\theta}^*(u)}\big ((\widetilde X_{-k}(u))_{k\in \N} \big )  {}^t \partial_{\boldsymbol{\theta} }\lambda_{\boldsymbol{\theta}^*(u)}\big ((\widetilde X_{-k}(u))_{k\in \N} \big )} { \lambda^{2}_{\boldsymbol{\theta}^*(u)}\big ((\widetilde X_{-k}(u))_{k\in \N} \big ) }\Big ] \Big ).
\end{equation}
\end{prop}
This central limit theorem can notably be applied for:
\begin{itemize}
\item  Time varying integer-valued GARCH$(p,q)$ processes (tvINGARCH$(p,q)$ processes), where
\begin{equation}\label{inGARCH}
\lambda_{\boldsymbol{\theta}_t^{(n)}}\big ((X^{(n)}_{t-j})_{j\geq 1}\big ) =a^{(n)}_{0,t}+\sum_{i=1}^p a^{(n)}_{i,t} \, X^{(n)}_{t-i} +\sum_{i=1}^q b^{(n)}_{i,t}\,\lambda_{\boldsymbol{\theta}_t^{(n)}}\big ((X^{(n)}_{t-i-j})_{j\geq 1}\big )  , 
\end{equation}
\mbox{for any $1\leq t\leq n$ and $n\in \N^*$}, with $X^{(n)}_t=0$ for all $t\leq 0$. Here 
$$
\boldsymbol{\theta}_t^{(n)}=\big (a^{(n)}_{0,t},\ldots,a^{(n)}_{p,t},b^{(n)}_{1,t},\ldots,b^{(n)}_{q,t} \big ).
$$
Here, following \cite{dft}, it is possible to consider a sharper set of parameters than the one given by \eqref{thetap}; hence let $\Theta$  be a bounded subset of 
\begin{equation}\label{thetaINGARCH}
\Theta_{INGARCH}=\Big \{\boldsymbol{\theta}=\big (a_0,a_1,\ldots,a_p,b_1,\ldots,b_q\big ) \in \R^{p+q+1},~\sum_{i=1}^p a_i +\sum_{i=1}^q b_i <1\Big \}.
\end{equation}
In such a case, the Lipschitz coefficients $\beta_j(\lambda,\Theta)$, $\beta_j(\partial _{\boldsymbol{\theta}}  \lambda,\Theta)$ and  $\beta_j(\partial^2 _{\boldsymbol{\theta}^2}  \lambda,\Theta)$  exist and decrease exponentially fast (see \cite{dft}), and all the conditions are satisfied for obtaining \eqref{tlcingarch}.
\item   Integer valued threshold GARCH$(p,q)$ processes where, with  $\ell$ a positive fixed integer,
\begin{multline*}
\lambda_{\boldsymbol{\theta}_t^{(n)}}\big ((X^{(n)}_{t-j})_{j\geq 1}\big ) \\ =a^{(n)}_{0,t}+\sum_{i=1}^p a^{(n)}_{i,t} \,\lambda_{\boldsymbol{\theta}_t^{(n)}}\big ((X^{(n)}_{t-i-j})_{j\geq 1}\big )+\sum_{i=1}^q \big (b^{(n)}_{i,t}\, \max(X^{(n)}_{t-i}-\ell,0)-c^{(n)}_{i,t}\, \min(X^{(n)}_{t-i},\ell)\big ), 
\end{multline*}
\mbox{for any $1\leq t\leq n$ and $n\in \N^*$}, with $X^{(n)}_t=0$ for all $t\leq 0$. Here 
$$
\boldsymbol{\theta}_t^{(n)}=\big (a^{(n)}_{0,t},\ldots,a^{(n)}_{p,t},b^{(n)}_{1,t},\ldots,b^{(n)}_{q,t},c^{(n)}_{1,t},\ldots,c^{(n)}_{q,t} \big ).
$$
As previously, the Lipschitz coefficients $\beta_j(\lambda,\Theta)$, $\beta_j(\partial _{\boldsymbol{\theta}}  \lambda,\Theta)$ and  $\beta_j(\partial^2 _{\boldsymbol{\theta}^2}  \lambda,\Theta)$  exist and decrease exponentially fast (see \cite{dft}), and all the conditions are satisfied for obtaining \eqref{tlcingarch}.
\end{itemize}

\section{Numerical experiments} \label{simu}
In the sequel we are going to apply our kernel based estimator in several different cases of local stationary processes. \\
~\\
The window bandwidth $h_n$ is a tuning parameter that requires to be chosen.  In order to neglect the bias we chose $h_n=n^{-\lambda}$ with $\lambda=0.35$, inducing $n\,h^3_n \limiten 0$, which is the uniform consistency and the asymptotic normality condition required for $Lip_p$-contrast and ${\cal C}^\rho$ functions when $p>3/2$ and $\rho=1$. 
\subsection{Monte Carlo simulations}
Here we will consider three cases:
\begin{enumerate}
\item \underline{ An example of tvGARCH$(1,1)$.} Here, with the notation of equation  \eqref{TVGARCH}, assume:
$$
c_{0,t}^{(n)}= 1+0.5  \, \sin \big (5\, \frac t n \big ),\quad c_{1,t}^{(n)}=0.1 + 0.4  \cos^2 \big (4\, \frac t n \big )
\quad \mbox{and}\quad  d_{1,t}^{(n)}=0.1+ 0.4 \,  \frac t n,
$$ 
for any $1\leq t \leq n$ and $n\in \N^*$. Clearly, $c^*_{0}(u)= 1+0.5\, \sin (5\, u ),~c^*_{1}(u)=0.1+0.4  \,  \cos^2 (4\, u )$ and  $ d_1^*(u)=0.1+0.4\, u$. Moreover, we assume that $(\xi_t)$ is a sequence of i.i.d.r.v. following ${\cal N}(0,1)$ distribution. \\
We independently replicated $1000$ trajectories of such process, for $n=2000, ~5000$ and $10000$ and computed the Gaussian QMLE estimators  $\widehat c_{0}(u),~\widehat c_{1}(u)$ and  $ \widehat d_1(u)$ for $u=k/50$ with $k=1,\ldots,49$. Finally, we used the two well known kernels, the uniform kernel $U(x)=\frac 1 2 \, \1 _{x\in [-1,1]}$ and the Epanechnikov one $E(x)=\frac 3 4 \, (1-x^2)\, \1 _{x\in [-1,1]}$ and denote respectively $\widehat \theta^U(u)$ and $\widehat \theta^E(u)$.  \\
~\\
Table \ref{Table1} contains the results of these Monte Carlo experiments where we computed the square root of the mean integrated squared error (RSMISE). In Figure \ref{Fig_GARCHt} exhibits an example of particular trajectories of these estimators for $n=10000$, while  Figure \ref{Fig_GARCHm} we also present the average trajectories of $\widehat c_0^E(u)$, $\widehat c_1^E(u)$ and $\widehat d_1^E(u)$ when $n=5000$.\\
\begin{table*}
\centering
\begin{tabular}{l||c|c||c|c||c|c|}
&   \multicolumn{2}{c|}{$c_0$} & \multicolumn{2}{c|}{$c_1$} & \multicolumn{2}{c|}{$d_1$}  \\
\hline 
$n$ & $\widehat c^U_0$ &$\widehat c^{E}_0$ &$\widehat c^U_1$ &$\widehat c^{E}_1$ &$\widehat d^U_1$ &$\widehat d^{E}_1$  \\ 
\hline
$1000$ & 0.493&0.455 &0.126 &0.122 &0.230 &0.208 \\
$3000$ & 0.363&0.323 &0.081 &0.077 &0.167 &0.146 \\
$10000$ & 0.259&0.224 & 0.052&0.048 &0.118 &0.101 \\
\hline
\end{tabular}
\caption{Square root of the MISE for tvGARCH$(1,1)$ processes for $n=1000$, $3000$ and $10000$ computed from $1000$ independent replications.}
\label{Table1}
\end{table*}




\begin{figure}
  \hspace*{-1.7cm}
  \begin{tabular}{lll}
     \subcaptionbox{Trajectory of $\widehat c_0^E$.\label{figc0t}}[0.2\linewidth]{\includegraphics[width=0.33\linewidth]{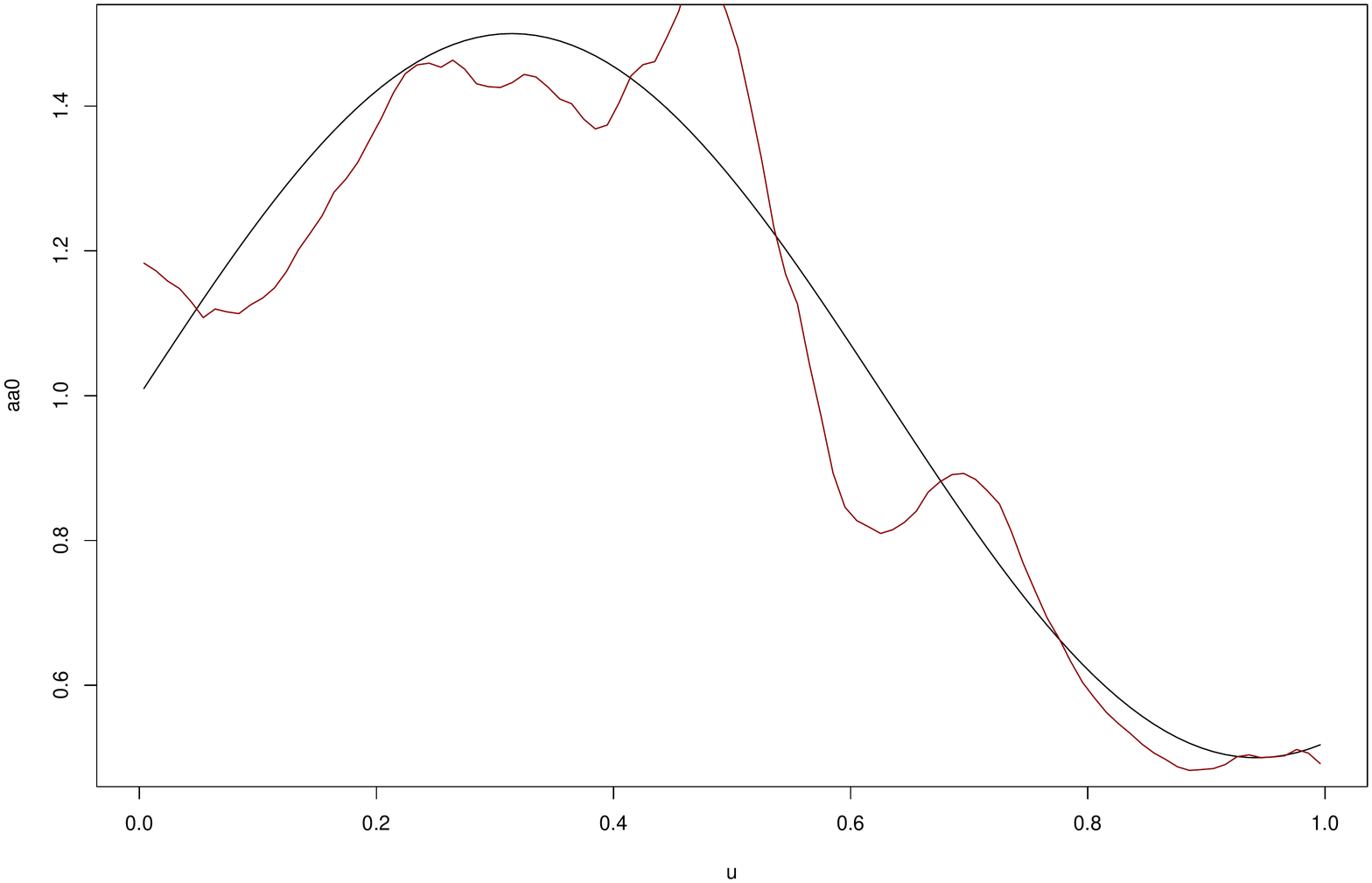}} &
   \hspace*{1.2cm} 
   \subcaptionbox{ Trajectory of $\widehat c_1^E$. \label{figc1t}}[0.2\linewidth]{\includegraphics[width=0.33\linewidth]{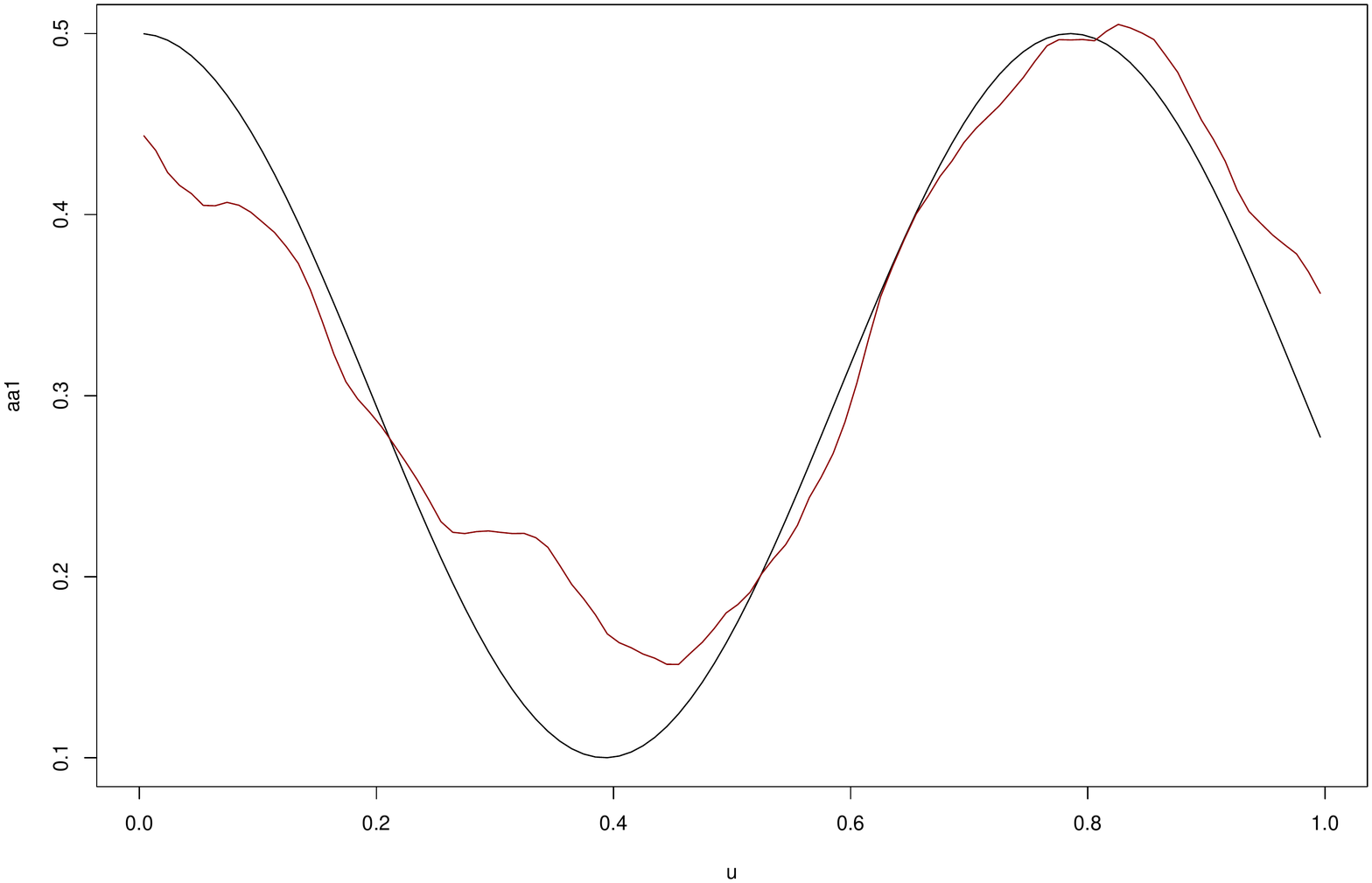}} &
	\hspace*{1.2 cm}	
	\subcaptionbox{ Trajectory of $\widehat d_1^E$. \label{figd1t}}[0.2\linewidth]{\includegraphics[width=0.33\linewidth]{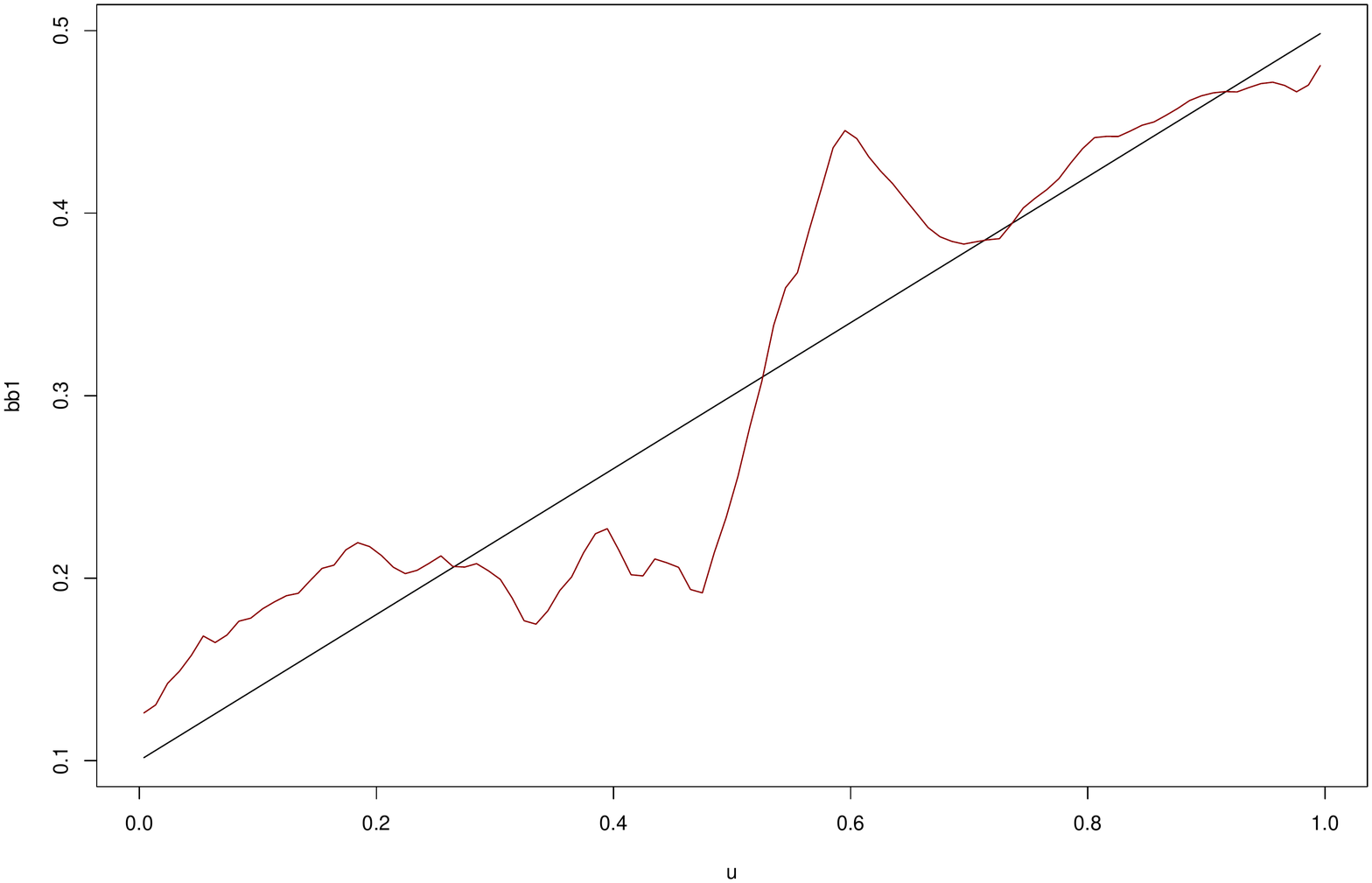}} \\
  \end{tabular}
	\caption{Paths of functions $c_0$, $c_1$, $d_1$ (in black), and a path of $\widehat c_0^E$, $\widehat c_1^E$ and $\widehat d_1^E$ (in red) for $n=10000$\label{Fig_GARCHt}} 
\end{figure}

\begin{figure}
  \hspace*{-1.7cm}
  \begin{tabular}{lll}
     \subcaptionbox{\small Mean of $\widehat c_0^E$.\label{figc0m}}[0.2\linewidth]{\includegraphics[width=0.33\linewidth]{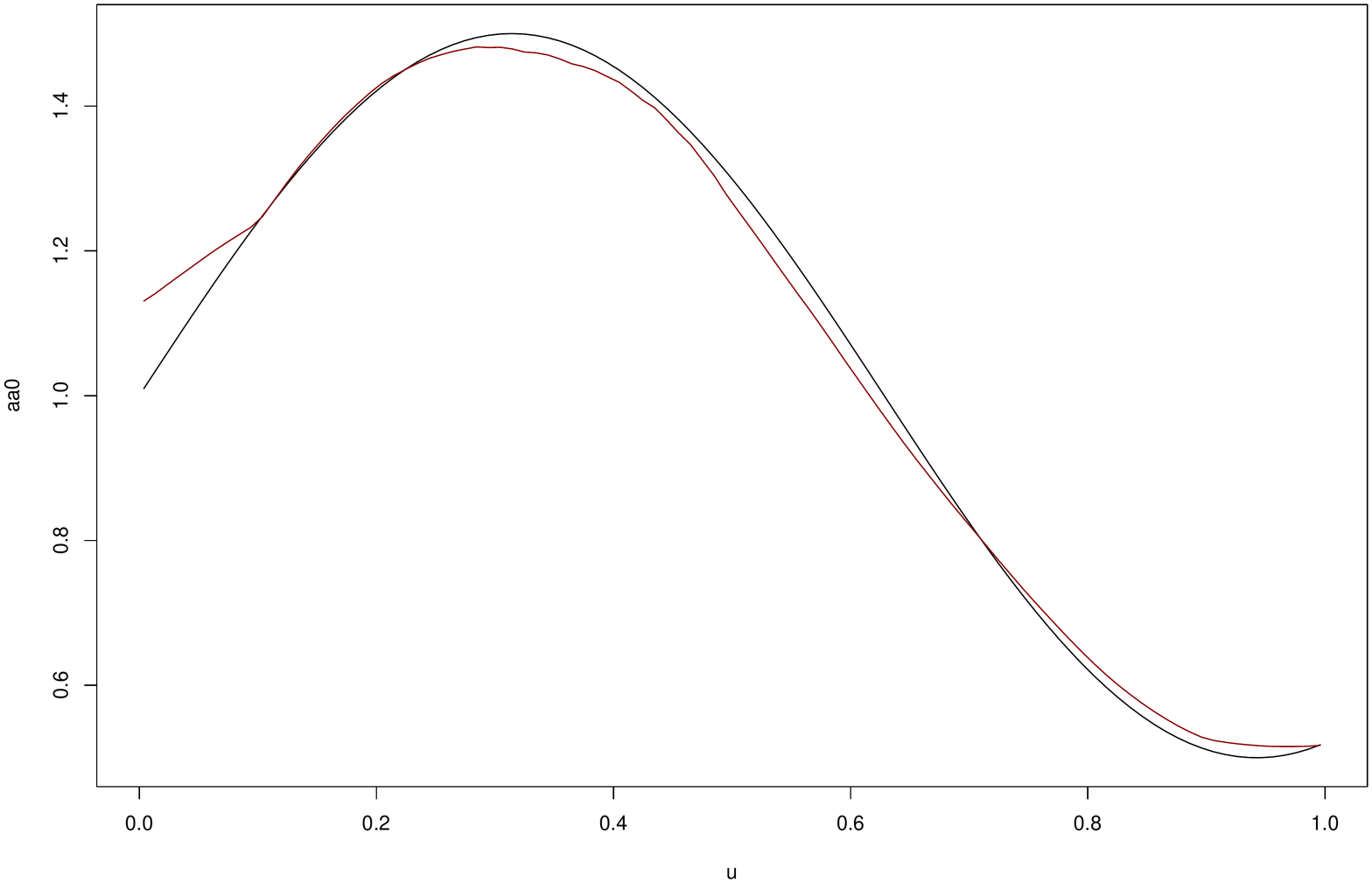}} &
   \hspace*{1.2cm} \subcaptionbox{\small Mean of $\widehat c_1^E$. \label{figc1m}}[0.25\linewidth]{\includegraphics[width=0.33\linewidth]{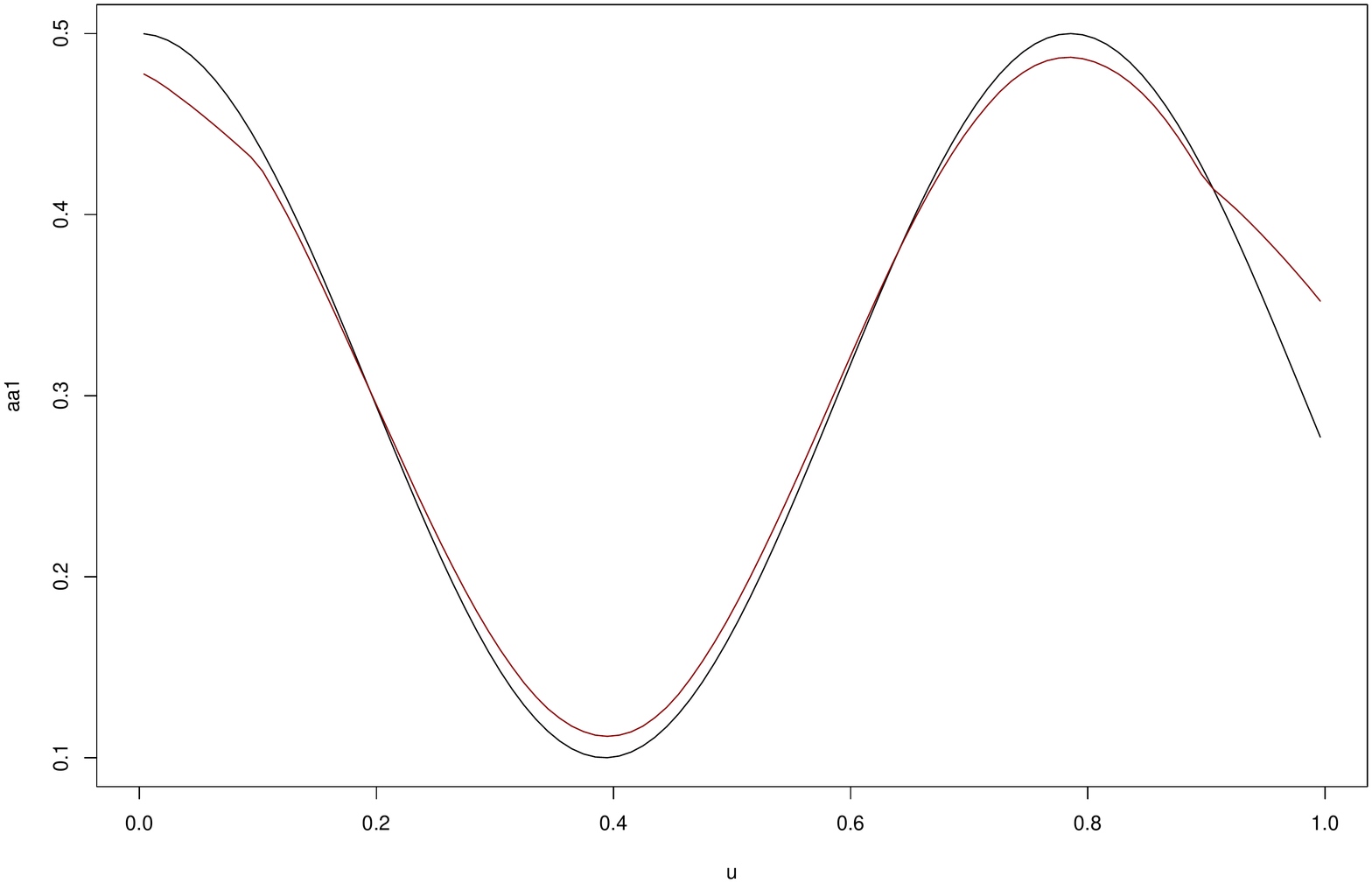}} &
	\hspace*{.5cm}
		\subcaptionbox{\small Mean of $\widehat d_1^E$. \label{figd1m}}[0.2\linewidth]{\includegraphics[width=0.33\linewidth]{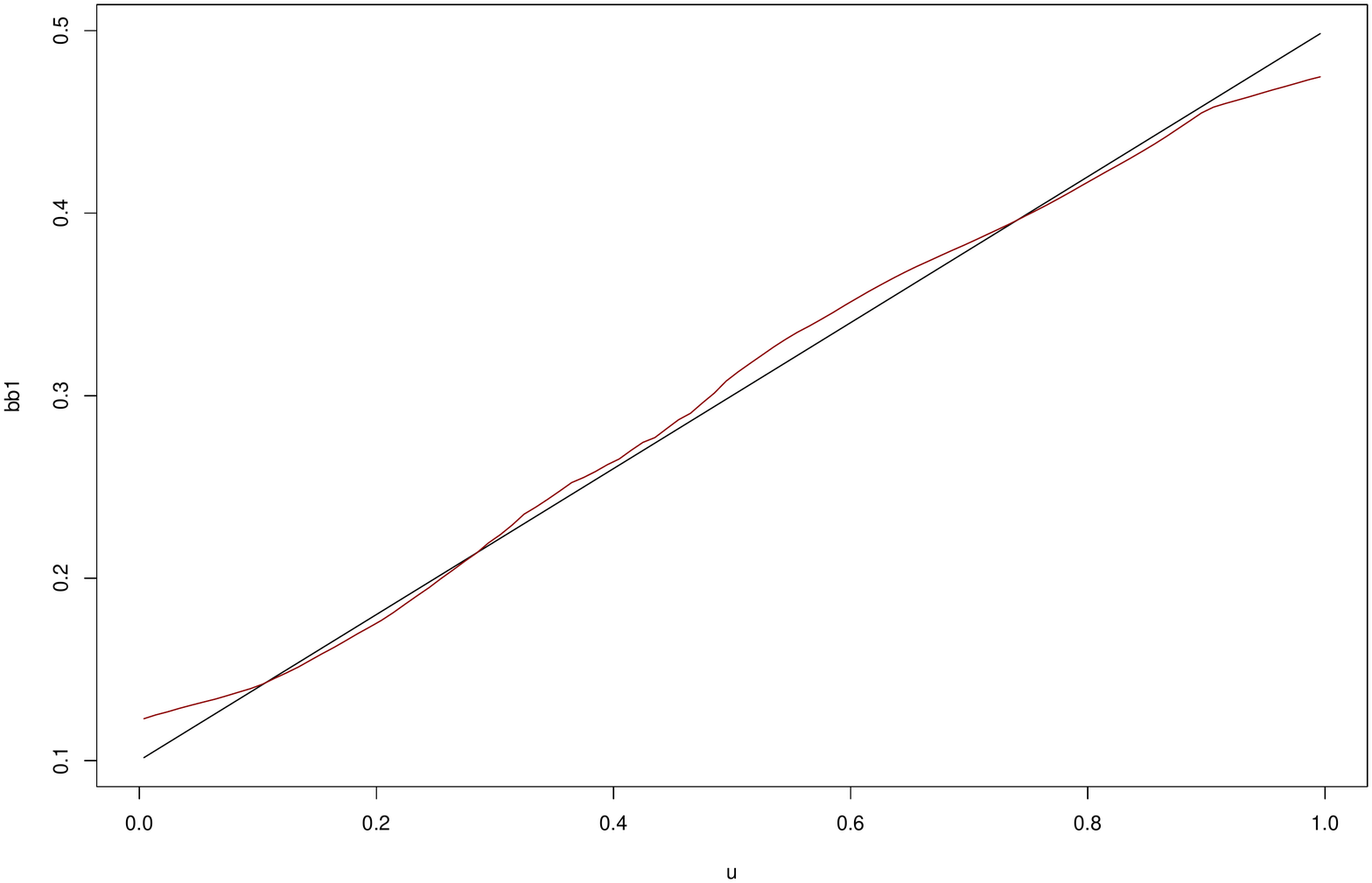}} \\
  \end{tabular}
	\caption{Paths of functions $c_0$, $c_1$, $d_1$ (in black), and the mean trajectories  over $1000$ replications of $\widehat c_0^E$, $\widehat c_1^E$ and $\widehat d_1^E$ (in red) for $n=5000$ 
	\label{Fig_GARCHm}} 
\end{figure}

\item \underline{A tvARCH$(\infty)$ example.} With the notation of equation  \eqref{tvARCH}, assume:
\begin{multline*}
\boldsymbol{\theta}=\big (c_0,c_1,p\big ),\quad a_0(\boldsymbol{\theta})=c_0 \quad \mbox{and} \quad a_j(\boldsymbol{\theta})=c_1 \, j^{-p}\quad\mbox{for $j\in \N^*$} \\
\mbox{with} \quad c_{0,t}^{(n)}=  1+0.5\, \sin \big (5\, \frac t n \big ), \quad c_{1,t}^{(n)}=0.1+0.5\, \cos^2\big (4 \, \frac j n \big ) \quad \mbox{and}\quad  p^{(n)}_t=2.1+ \frac j n.
\end{multline*}
for any $1\leq t \leq n$ and $n\in \N^*$. Therefore $c^*_{0}(u)= 1+0.5 \, \sin (5\, u ),~c^*_{1}(u)=0.1 +0.5 \, \cos^2 (4\, u )$ and  $p^*(u)=2.1 + u$. Moreover, we assume that $(\xi_t)$ is a sequence of i.i.d.r.v. following ${\cal U}([-\sqrt 3, \sqrt 3 ])$ (uniform) distribution. \\
As previously, we replicated $1000$ trajectories of such process (see for instance one trajectory in Figure \ref{Fig_ARCH}), for $n=2000, ~5000$ and $10000$ and computed the Gaussian QMLE estimators with Epanechnikov kernel $\widehat c^E_{0}(u),~\widehat c^E_{1}(u)$ and  $ \widehat p^E(u)$ for $u=k/50$ with $k=1,\ldots,49$.  \\
\begin{figure}
  \hspace*{0cm}
	\centering
  \includegraphics[height=4cm,width=12cm]{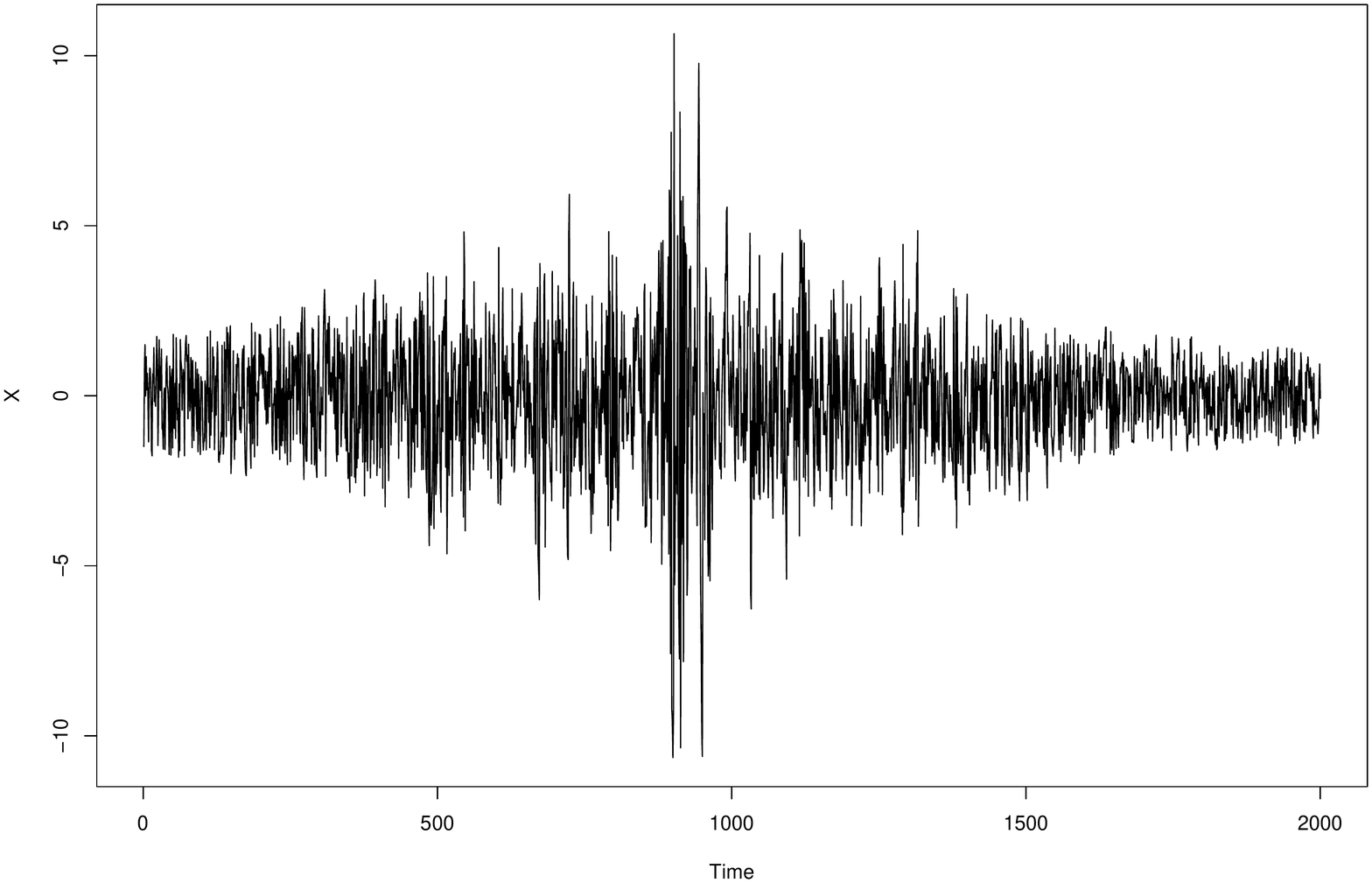}
	\caption{On trajectory of a tvARCH$(\infty$)-process for $n=2000$\label{Fig_ARCH}} 
\end{figure}
~\\
Table \ref{Table2} contains the results of these Monte-Carlo experiments where we computed the square root of mean integrated squared error (RSMISE).\\
\begin{center}
\begin{table*}
\centering
\begin{tabular}{l||c|c||c|c||c|c|} 
$n$ & $\widehat c^{U}_0$  & $\widehat c^{E}_0$ &$\widehat c^{U}_1$&$\widehat c^{E}_1$ &$\widehat p^{U}$&$\widehat p^{E}$   \\ 
\hline
$1000$ & 0.271 & 0.294 & 0.089 & 0.089& 0.789& 0.770 \\
$2000$ & 0.227& 0.261 & 0.066 & 0.069 &0.712 & 0.696 \\
$5000$ & 0.164& 0.187 & 0.047 & 0.049 &0.624 & 0.591  \\
$10000$ & 0.128 & 0.144 &0.037 & 0.039 & 0.555 & 0.517 \\
\hline
\end{tabular}
\caption{Square root of the MISE of $\widehat c^{U}_0$, $\widehat c^{E}_0$, $\widehat c^{U}_1$ , $\widehat c^{E}_1$, $\widehat p^{U}$  and $\widehat p^{E}$ for the tvARCH($\infty$) processes for $n=1000$, $2000$, $5000$ and $10000$ computed from $1000$ independent replications.}
\label{Table2}
\end{table*}
\end{center}
\item \underline{The example of one integer valued process.} \\ A tvINGARCH$(1,0)$-process (tvINARCH$(1)$) as  defined in \eqref{inGARCH}. Here we choose 
$$
a_{0,t}^{(n)}=1+0.5 \, \sin \big (5\, \frac t n \big )\quad \mbox{and}\quad  a_{1,t}^{(n)}=0.3 + 0.5  \,  \frac t n
$$
for any $1\leq t \leq n$ and $n\in \N^*$. Note that we only consider here the Poisson distribution case. Figure \ref{FigIN} exhibits a trajectory of such a process for $n=1000$. 
\begin{figure}
  \hspace*{0cm}
	\centering
  \includegraphics[height=4cm,width=12cm]{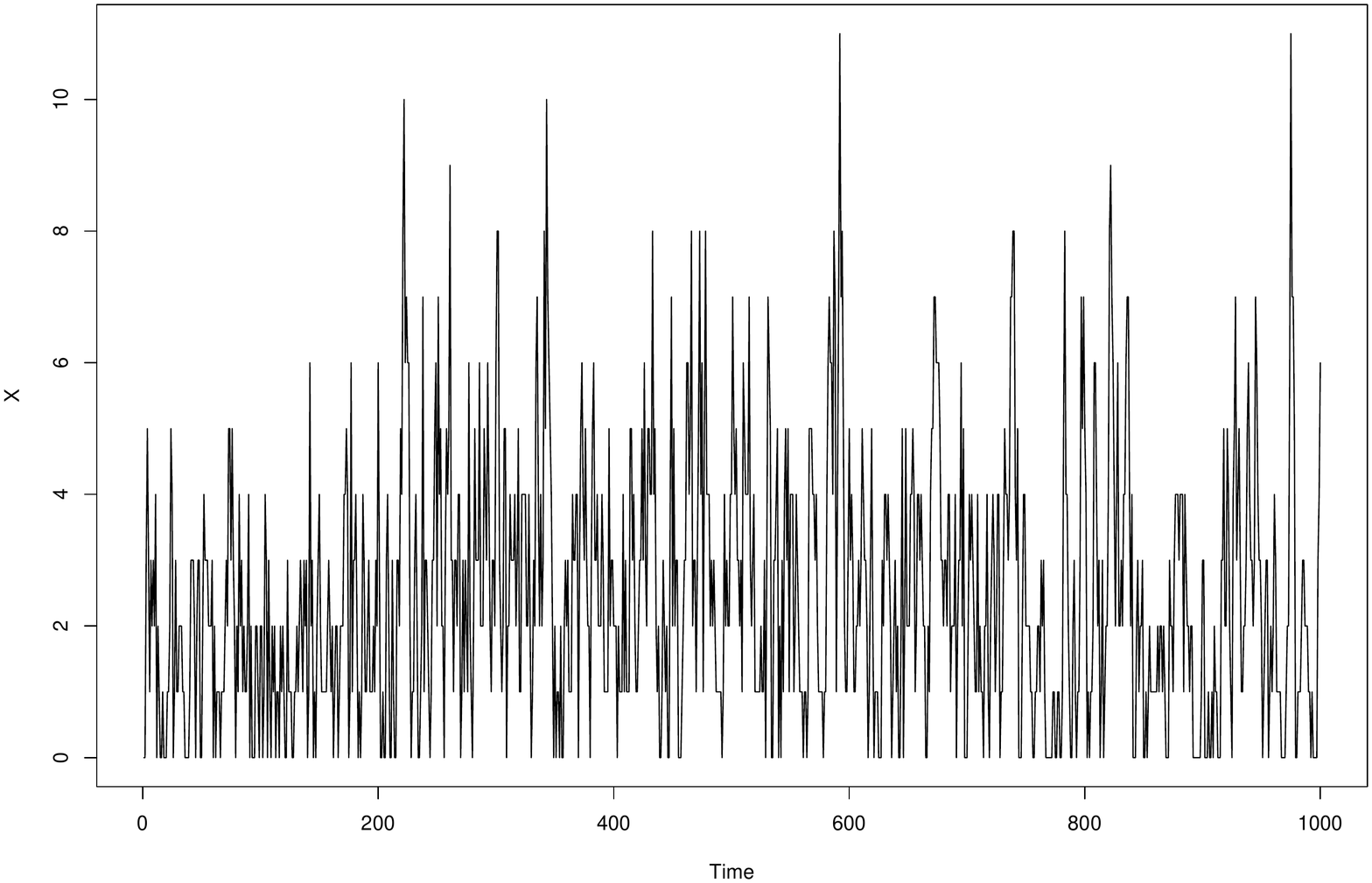}
	\caption{A trajectory of a tvINGARCH$(1,0)$-process for $n=1000$\label{FigIN}} 
\end{figure}
~\\
Using the same procedure than in the previous examples, Table \ref{Table3} contains the RSMISE of the estimators computed with Uniform and Epanechnikov kernels, $\widehat a^{U}_0$, $\widehat a^{E}_0$, $\widehat a^{U}_1$  and $\widehat a^{E}_1$.  \\
\begin{table*}
\centering
\begin{tabular}{l||c|c|| c | c|} 
$n$ & $\widehat a^{U}_0$ & $\widehat a^{E}_0$ &$\widehat a^{U}_1$ &$\widehat a^{E}_1$  \\ 
\hline
$1000$ & 0.144 & 0.135 & 0.051 & 0.058\\
$2000$ & 0.111 & 0.103&0.045 & 0.041 \\
$5000$ & 0.079 & 0.073&0.032 & 0.030 \\
$10000$ &0.061 & 0.056&0.025 & 0.022\\
\hline
\end{tabular}
\caption{Square root of the MISE of $\widehat a^{U}_0$, $\widehat a^{E}_0$, $\widehat a^U_1$ and $\widehat a^E_1$ for tvINGARCH($1,0$) processes for $n=1000$, $2000$, $5000$ and $10000$ computed from $1000$ independent replications.}
\label{Table3}
\end{table*}
\

\noindent
{\bf Conclusion:} We can globally conclude from these Monte Carlo experiments:
\begin{itemize}
\item The consistency of the estimators and their convergence rate are established;
\item The Epanechnikov's kernel is preferable to the uniform one.
\end{itemize}
\end{enumerate}
\subsection{Application to financial data}
We apply our local non-parametric estimator to a trajectory of financial data. More precisely, we consider the log-returns of the daily closing values of S\&P500 index between July 1999 and July 2019 (therefore $n=5031$, see also Figure \ref{FigSP} for the graph of this trajectory).  
\begin{figure}
  \hspace*{0cm}
	\centering
  \includegraphics[height=4cm,width=12cm]{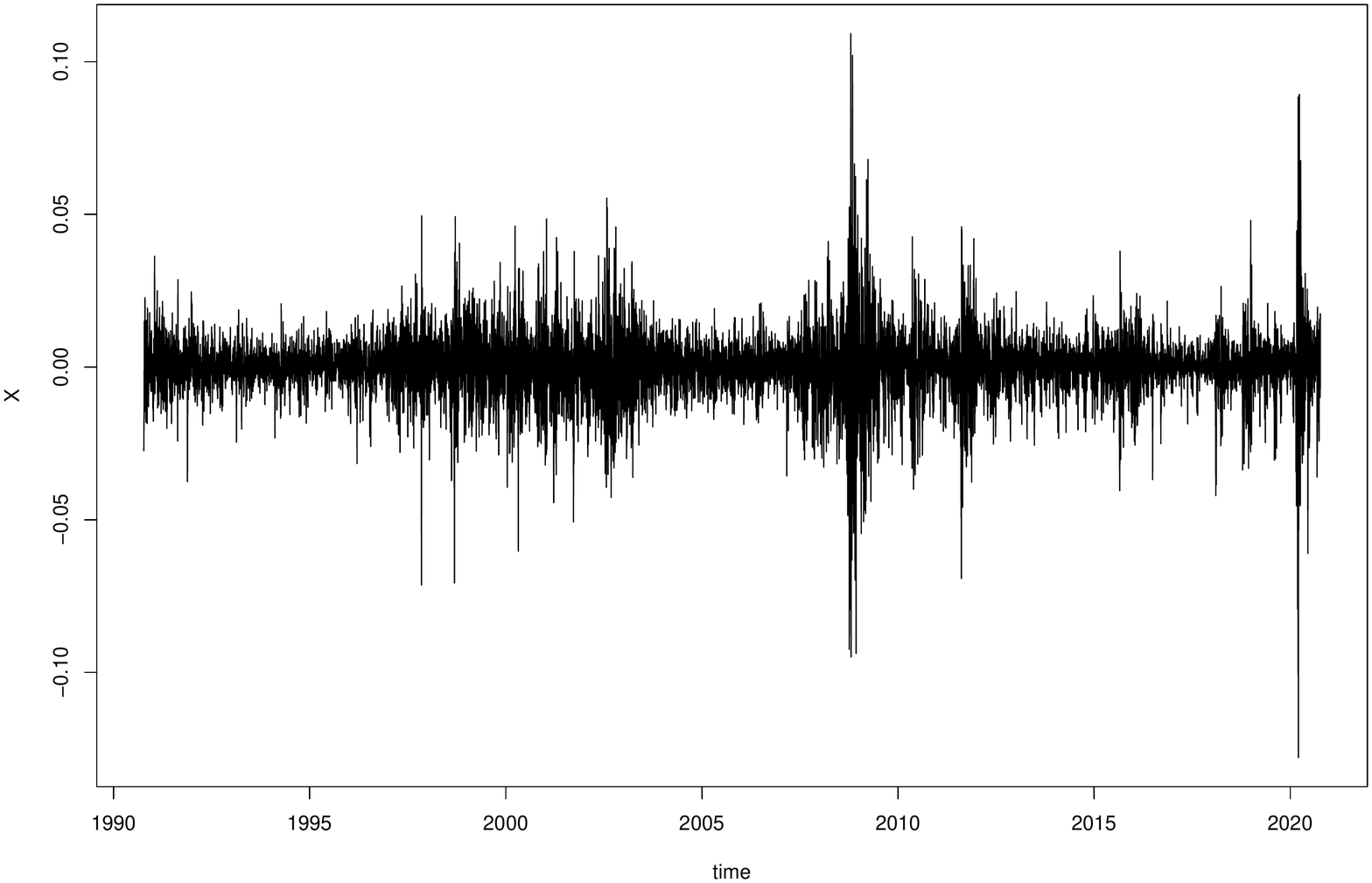}
	\caption{The log-returns of daily closing values of S\&P500 index between October 1990 and October 2020\label{FigSP}} 
\end{figure}
Many studies have shown that the GARCH$(1,1)$ process is a relevant model for this type of data (We refer to the monograph of Francq and Zako\"ian \cite{fz} for more details). As a consequence, we used a tvGARCH$(1,1)$ process (see \eqref{TVGARCH}) to take into account the changes in economic and financial conjectures over 20 years on such a model (think in particular of the September 2008 crisis and of the  spring 2020 COVID crisis). Figure \ref{FigSP33} exhibits the evolution of the three estimators computed with Uniform and Epanechnikov kernels, {\it i.e.} $\widehat c^U_0$, $\widehat c^U_1$, $\widehat d^U_1$, $\widehat c^E_0$, $\widehat c^E_1$, $\widehat d^E_1$ from Gaussian QMLE. \\
\begin{figure}
 	\begin{center}
  \begin{tabular}{ccc}
     \subcaptionbox{\textcolor{red}{$\widehat c^U_0$} and $\widehat c^E_0$ }[0.33\linewidth]{\includegraphics[width=0.42\linewidth]{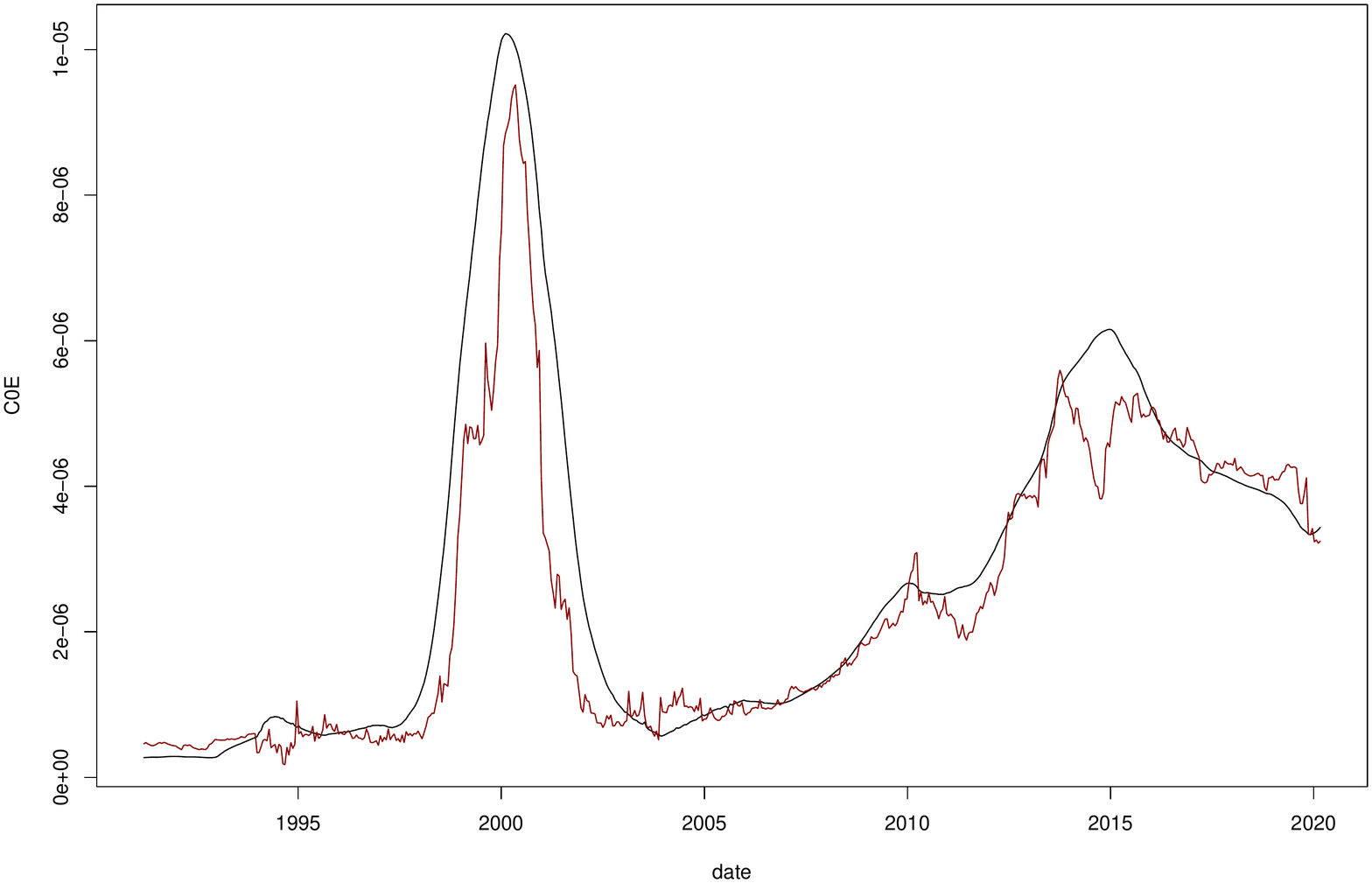}} &
   \hspace*{1.5cm}
    \subcaptionbox{\textcolor{red}{$\widehat c^U_1$} and $\widehat c^E_1$}[0.33\linewidth]{\includegraphics[width=0.42\linewidth]{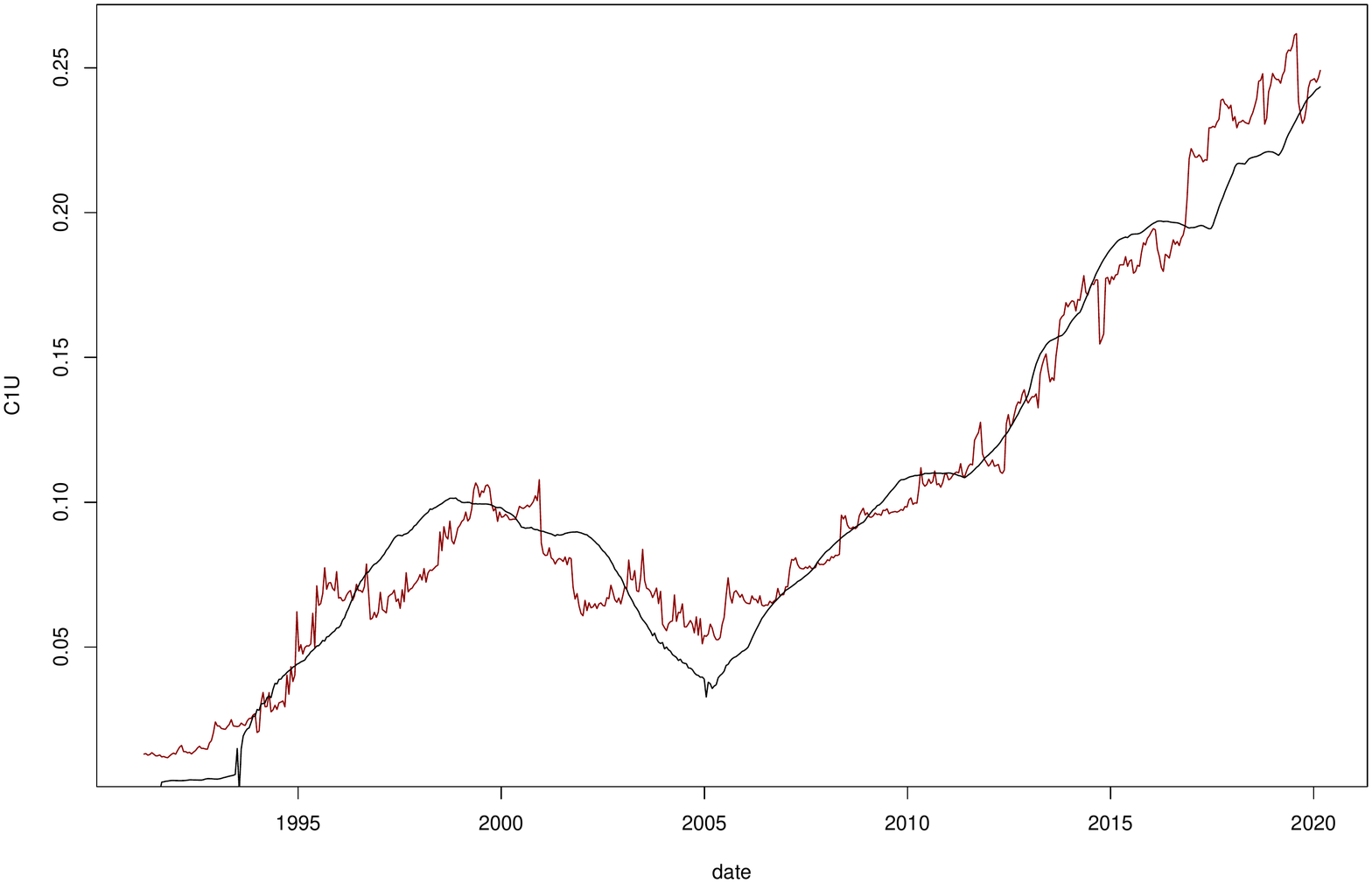}}  \\
	\subcaptionbox{\textcolor{red}{$\widehat d^U_1$} and $\widehat d^E_1$}[0.33\linewidth]{\includegraphics[width=0.42\linewidth]{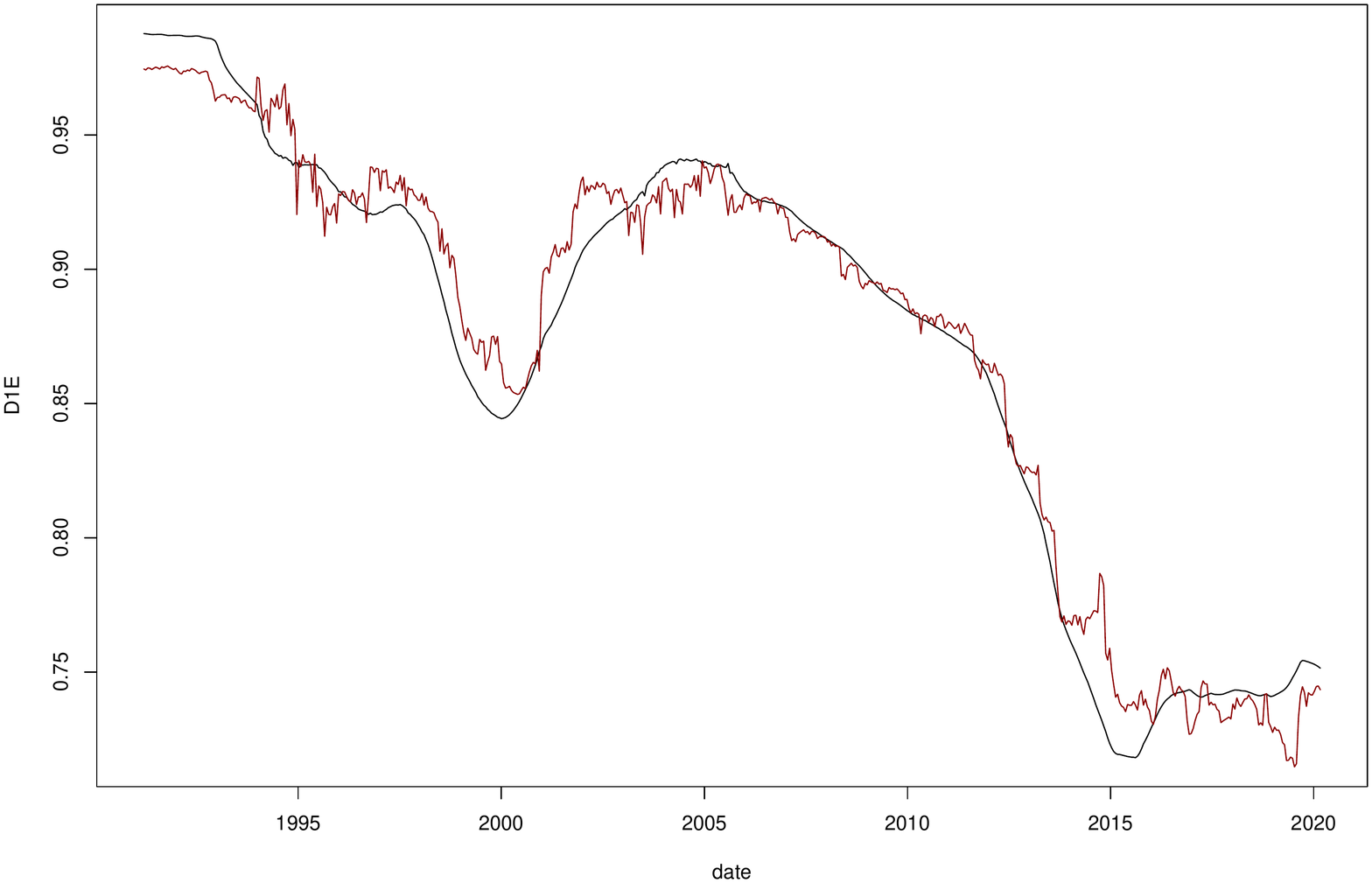}} &
	\hspace*{1.5cm}
	 \subcaptionbox{\textcolor{red}{$\widehat c^U_1+\widehat d^U_1$}
	 and $\widehat c^E_1+\widehat d^E_1$
	  }[0.3\linewidth]{\includegraphics[width=0.42\linewidth]{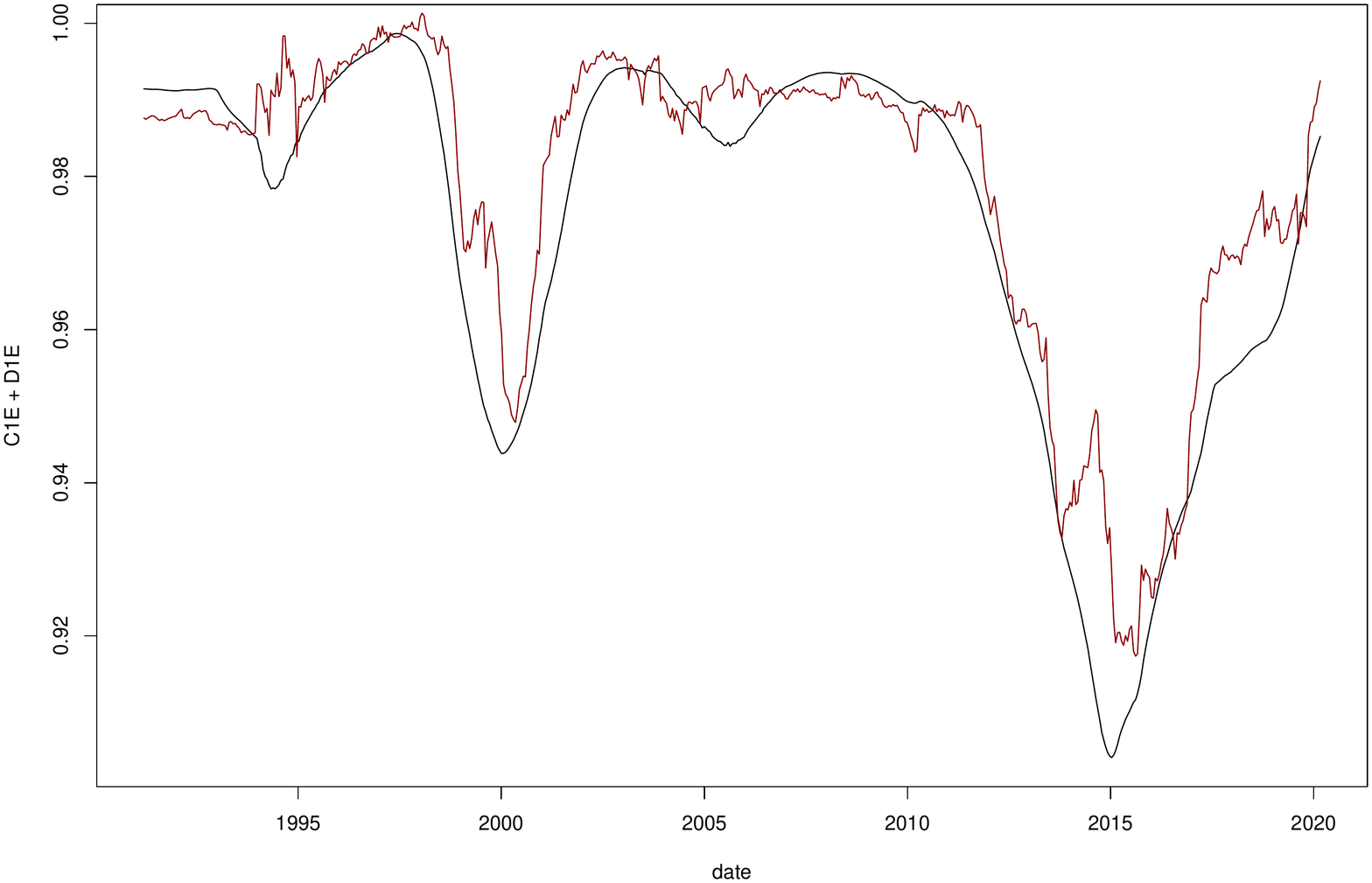}}
  \end{tabular}
	\end{center}
	\caption{Estimators of  $c_0$, $c_1$, $d_1$ and $c_1+d_1$ for the log-returns of daily closing values of S\&P500 index from October 1990 and October 2020\label{FigSP33}} 
\end{figure}

We draw the evolution of $c_1+d_1$ in order to get a visual indicator of the variability of the S\&P500 index. The larger $c_1+d_1$ the worst  the moment properties of the tvGARCH$(1,1)$. The variability may be seen as an indicator of instability of the financial markets and thus of the crisis. Indeed the maximum of the $c_1+d_1$ is achieved at the chore of the September 2008. More surprisingly, there is also a peak of variability as early as 2003. There the financial markets were renewing at their climate between 1998 and 2008 crisis. In order to distinguish between the two peaks of variability, one should observe that the curves of the coefficients $c_1$ and $d_1$ separately. Then we observe that 2003 corresponds to a higher value for the coefficient $d_1$ and 2008 to a higher value for the coefficient $c_1$. We note that $d_1$ is the coefficient of persistence in the volatility whereas $c_1$ transfers external shocks in the volatility. Finally, also surprisingly, the COVID crisis does not seem to change the last five years evolution of $c_1$ and $d_1$.

\section{Moments and coupling properties of non stationary infinite memory processes} \label{proofs1}
\subsection{Proof of the moments properties in Lemma \ref{momm}}
We start this section with the proof of Lemma \ref{momm} which follows similar arguments than in \cite{DW} and that we give here for completeness.
\begin{proof}[Proof of Lemma \ref{momm}]
Under the assumption {\bf(A$_0$($\Theta$))}, for any $n\in \N^*$ and $0\leq t \leq n$ we have
$$
\|X_t^{(n)}-F_{{\boldsymbol{\theta}}_t^{(n)}}(0; \xi_t)\|_p\le \sum_{s=1}^\infty b^{(0)}_s\|X_{t-s}^{(n)}\|_p\,.
$$
Thus from the triangle inequality we obtain
$$
\|X_t^{(n)}\|_p\le \sum_{s=1}^\infty b^{(0)}_s\|X_{t-s}^{(n)}\|_p+\sup_{{\boldsymbol{\theta}} \in \Theta}\|F_{\boldsymbol{\theta}}(0,\xi_0)\|_p.
$$
As a consequence  we get
\begin{equation*}\label{reccontr}
\|X_t^{(n)}\|_p\le \sum_{s=1}^\infty b^{(0)}_s\max_{j\leq  t-1}\|X_{j}^{(n)}\|_p+\sup_{\boldsymbol{\theta}}\|F_{\boldsymbol{\theta}}(0,\xi_0)\|_p\le  B_0(\Theta) \, \max_{j< t}\|X_{j}^{(n)}\|_p+C_0(\Theta).
\end{equation*}
With $M_t=\max_{j\le  t}\|X_{j}^{(n)}\|_p$, a recursion entails that $M_t\le B_0(\Theta) \, M_{t-1}+C_0(\Theta)$ where $0\leq B_0(\Theta)<1$, which implies with $M_0=0$ that for any $0\leq t \leq n$,  $$M_t\leq C_0(\Theta) \sum_{k=0}^tb^{(0)}_k \leq \frac {C_0(\Theta)} {1-B_0(\Theta)} <\infty$$ and this achieves the proof. We refer to \cite{DW} for more details.
\end{proof}
\subsection{Weak dependence properties of the stationary version}
Any stationary infinite memory process \eqref{chaineinf} admits  fruitful coupling properties. In this section we quantify them thanks to the  $\tau$-weak dependence properties introduced in \cite{dp}. The reader is deferred to the  lecture notes \cite{DDLLLP}  for complements and details on coupling, based on the Wasserstein distance between probabilities.
Conditional coupling for stationary time series is defined as follows.
\begin{Def}[\cite{dp}]\label{defwd}
Let $(\Omega,\mathcal{C}, \P)$ be a probability space,
$\mathcal{M}$ a $\sigma$-subalgebra of $\mathcal{C}$ and $Z$ a
random variable with values in $E$. Denote $\Lambda_1(E)$  the space of $1$-Lipschitz functions from $E$ to
$\R$. Assume that $\|Z\|_p<\infty$ and define the coupling coefficient $\tau^{(p)}$ as
\begin{equation*}
\tau^{(p)}(\mathcal{M},Z)=
\Big \|\sup_{f\in\Lambda_1(E)}\Big\{\Big|\int f(x)\P_{Z|\mathcal{M}}(dx)-\int
f(x)\P_{Z}(dx) \Big |\Big \}\Big \|_p.
\end{equation*}
\end{Def}
The dependence between the past of the sequence $(Z_t)_{t\in\Z}$ and its future $k$-tuples may be assessed by using the coupling coefficient $\tau^{(p)}$: Consider the norm
$\|x-y\|=\|x_1-y_1\|+\cdots+\|x_k-y_k\|$ on $E^k$,
set $\mathcal{M}_p=\sigma(Z_t,t\le p)$ and define
\begin{eqnarray*}
\tau_Z^{(p)}(r)&=&\sup_{k>0} \Big \{ \max_{1\le l\le k}
\frac1l\sup\Big\{\tau^{(p)}(\mathcal{M}_p,(Z_{j_1},\ldots,Z_{j_l}))\mbox{ with }p+r\le
j_1<\cdots <j_l\Big\} \Big \}.
\end{eqnarray*}
Finally, the time series $(Z_t)_{t\in\Z}$ is said to be {\it $\tau_Z^{(p)}$-weakly dependent} in case the sequence of coefficients $\tau_Z^{(p)}(r)$ tend to $0$ as $r$ tends to infinity.\\

The $\tau$-dependence coefficients of the stationary process $(\widetilde X_t(u))_{t\in \Z}$ are bounded above by using the following coupling scheme.  Hence, if  $(\xi^\circ_t)_{t\in \Z}$ is an independent replication of $(\xi_t)_{t\in \Z}$, define $(\widetilde X^\circ_{t}(u))_{t\in \Z}$ such as:
\begin{equation}\label{couple}
\widetilde X^\circ_{t}(u)=\left \{ \begin{array}{ll}
F_{\boldsymbol{\theta}^*(u)}\big ((\widetilde X^\circ_{t-k}(u))_{k\geq 1}, \xi^\circ_t\big ),& \mbox{for $t\leq 0$}; \\
F_{\boldsymbol{\theta}^*(u)}\big ((\widetilde X^\circ_{t-k}(u))_{k\geq 1}, \xi_t\big ),& \mbox{for $t> 0$}.
\end{array} \right . 
\end{equation}
Then for $s\geq 0$, we have the upper-bound
\begin{equation}\label{taucouple}
\tau_{\widetilde X(u)}^{(p)} (s) \leq \big \| \widetilde X_{s}(u)- \widetilde X^\circ_{s}(u)\big \|_p\,.
\end{equation}
In the following, we mimic the proof of Theorem 3.1 of \cite{DW} in order to get an $\L^p$-estimate uniform over $u\in [0,1]$ of the approximation of $\widetilde X_{s}(u)$ by $\widetilde X^\circ_{s}(u)$. We start by estimating the moments $ \|\sup_{u\in[0,1]} |\widetilde X_{s}(u)|\|_p$ in the following Lemma.
\begin{lem}\label{mommbis}
Let $\Theta \subset \R^d$ be such that {\bf(A$_0$($\Theta$))} holds with $B_0(\Theta)<1$ and assume that (LS$(\rho)$) also holds. Then the stationary version  $(\widetilde X_{t}(u))_{t\in \Z}$ solution of \eqref{eq:statu} satisfies
$$
\Big\|\sup_{u\in [0,1]}|\widetilde X_{t}(u)|\Big\|_p \leq \frac {C_0(\Theta)} {1-B_0(\Theta)}\,,\qquad t\in\Z\,.
$$
\end{lem}
\begin{proof}[Proof of Lemma \ref{mommbis}]
We adapt the fixed point approach of \cite{DT} to our setting. We refer to \cite{DT} for details. We consider $\L^p(\mathcal C([0,1],\R))$ the Banach space of random continuous functions $H:\, [0,1]\to \R$ that admits finite $p$ moments equipped with the norm $H\mapsto \|\sup_{u\in[0,1]}|H_u|\|_p$. The underlying probability space is the one of the probability distribution of the iid sequence $(\xi_t)_{t\in\Z}$, i.e. $H_u$ is a measurable function of $(\xi_t)_{t\in\Z}$ such that $\E[\sup_{u\in[0,1]}|H_u((\xi_t)_{t\in\Z})|^p]<\infty$. We denote $L$ the lag operator on sequences $(x_t)_{t\in\Z}$ of $\R^\Z$ such that $L((x_t)_{t\in\Z})=(x_{t-1})_{t\in\Z}$. We denote $\Phi$ the function from $\L^p(\mathcal C([0,1],\R))$ such that
$$
\Phi(H)(u)=F_{\boldsymbol{\theta}^*(u)}((H_u\circ L^j)_{j\ge 0},\pi_0)\,,\qquad u\in[0,1]\,,
$$
where $\pi_0$ is the projection $\pi_0 \big ((x_t)_{t\in\Z}\big )=x_0$. That $u\mapsto
\Phi(H)(u)$ is continuous follows from the continuity of $\boldsymbol{\theta}\mapsto F_{\boldsymbol{\theta}}$ and $u\mapsto \boldsymbol{\theta}^*(u)$ under {\bf(A$_0$($\Theta$))} and (LS$(\rho)$). That $\sup_{u\in[0,1]}|\Phi(H)|(u)$ admits finite moments of order $p$ follows from similar arguments than in Lemma 1 of \cite{DT} under {\bf(A$_0$($\Theta$))} that holds uniformly in $u\in [0,1]$. One can apply the Picard fixed point theorem to $\Phi$ which is a contraction under $B_0(\Theta)<1$. We derive the existence of $\widetilde X_{t}(u)$ in the Banach space $\L^p(\mathcal C([0,1],\R))$ and the desired estimate on its norm.
\end{proof}
Notice that the same uniform estimate also holds on the coupling version $\widetilde X^\circ_{s}(u)$ so that one can consider the approximation, for any $s\in \N^*$ and any $r\in \N^*$. Now let us  set the uniform coupling $\tau$-coefficients as $\tau_{\widetilde X}^{(p)} (s)\equiv \big \| \sup_{u\in[0,1]}|\widetilde X_{s}(u)-  \widetilde X^\circ_{s}(u)| \big\|_p$, then,
\begin{align}
\tau_{\widetilde X}^{(p)} (s)&\leq \big \|  \sup_{u\in[0,1]}\big |F_{\boldsymbol{\theta}^*(u)}\big ((\widetilde X_{s-k}(u))_{k\geq 1}, \xi_t\big )- F_{\boldsymbol{\theta}^*(u)}\big ((\widetilde X^\circ_{s-k}(u))_{k\geq 1}, \xi^\circ_t\big )\big |\big  \|_p\nonumber\\
&\leq \big \|  \sup_{u\in[0,1]}\big  |F_{\boldsymbol{\theta}^*(u)}\big ((\widetilde X_{s-k}(u))_{k\geq 1}, \xi_s\big )- F_{\boldsymbol{\theta}^*(u)}\big ((\widetilde X^\circ_{s-k}(u))_{k\geq 1}, \xi^\circ_s\big )\big  |\big \|_p\nonumber\\
&\leq \sum_{k=1}^\infty b^{(0)}_t(\Theta) \big \|  \sup_{u\in[0,1]}\big  |\widetilde X_{s-k}(u)- \widetilde X^\circ_{s-k}(u)\big  |\big \|_p\nonumber\\
&\leq B_0(\Theta)\max_{s-r\le t\le s-1} \big \| \sup_{u\in[0,1]}\big  | \widetilde X_{t}(u)- \widetilde X^\circ_{t}(u)\big  |\big \|_p
\\ & + 2  \sum_{k=r+1}^\infty b^{(0)}_t(\Theta) \big \| \sup_{u\in[0,1]}| \widetilde X_0(u)|\big  \|_p\nonumber.
\end{align}
By a recursive argument we easily derive that $\max_{t\ge 0} \big \|\sup_{u\in[0,1]}\big | \widetilde X_{t}(u)- \widetilde X^\circ_{t}(u)\big |\big \|_p<\infty$. Then we extend  the bound so that 
\begin{multline}\label{recurcoupl}
\max_{t\ge s} \big \|\sup_{u\in[0,1]}| \widetilde X_{t}(u)- \widetilde X^\circ_{t}(u)|\big \|_p\\ \le
 B_0(\Theta)\max_{ t \ge s-r } \big \|\sup_{u\in[0,1]}| \widetilde X_{t}(u)- \widetilde X^\circ_{t}(u)|\big \|_p + 2  \sum_{k=r+1}^\infty b^{(0)}_t(\Theta)\big  \|\sup_{u\in[0,1]}| \widetilde X_0(u)| \big \|_p\,.
\end{multline}
Given any $r\in \N^*$, a recursive argument on the sequence $\big (\max_{ t \ge k r } \big \|\sup_{u\in[0,1]}| \widetilde X_{t}(u)- \widetilde X^\circ_{t}(u)|\big \|_p \big )_{k\ge 0}$ based on \eqref{recurcoupl} together with the coupling properties of the $\tau$-dependent coefficient \eqref{taucouple} yields the following Lemma that is analog to Theorem 3.1 of \cite{DW}. We refer to \cite{DW} for more details.
\begin{lem}\label{momm2} 
Let $\Theta \subset \R^d$ such that {\bf(A$_0$($\Theta$))} holds with $B_0(\Theta)<1$ and assume that (LS$(\rho)$) also hold. Then the stationary version  $(\widetilde X_{t}(u))_{t\in \Z}$ is uniformly approximated by its coupling version $(\widetilde X^\circ_{t}(u))$ so that  we have
\begin{equation}\label{tauX}
\tau_{\widetilde X}^{(p)} (s)
   \le C \,\lambda_s
    \qquad\mbox{where}\quad \lambda_s=  \inf_{1\leq r\le s} \Big(B_0(\Theta)^{s/r}+\sum_{t= r+1}^\infty b^{(0)}_t(\Theta) \Big)\quad \mbox{for $s \ge 1$}.
\end{equation}
\end{lem}

Remark that the bound on the $\tau$-coefficients does not depend on $u\in (0,1)$. Notice also that:
\begin{equation}\label{taudet}
\begin{array}{cl} 
\mbox{$\bullet$ if} & \mbox{$b^{(0)}_t(\Theta) =O\big (t^{-\kappa}\big )$ with $\kappa>1$, then $\tau_{\widetilde X}^{(p)} (s)\leq \lambda_s =O\big (s^{1-\kappa} \log s\big )$;} \\
\mbox{$\bullet$ if} &\mbox{$b^{(0)}_t(\Theta) =O\big (r^{t}\big )$ with $0<r<1$, then $\tau_{\widetilde X}^{(p)} (s) \leq \lambda_s=O\big (e^{\sqrt{ s \, \log (r) \,\log(B_0(\Theta))}} \big )$.} 
\end{array}
\end{equation}

The SLLN in \cite{DW} is implied by the summability of the $\tau$-dependence coefficients. We will use the following Lemma on the $\tau$-dependence coefficients $\tau_{\widetilde X}^{(p)} (s)$, $s\ge 1$.
\begin{lem}\label{lem:lambda}
If $\sum_{t=2}^\infty t\log(t) b_t^{(0)}(\Theta)<\infty$ then $\sum_{s=1}^\infty \lambda_s<\infty$.
\end{lem}
\begin{proof}
Choosing $r= \lfloor s/C \log(s)\rfloor$ for $s\ge 2$ and $C>0$  we have
$$
\lambda_s\le s^{-C\log (1/B_0(\Theta))} + \sum_{t= \lfloor s/C \log(s)\rfloor }^\infty b_t^{(0)}(\Theta).
$$
For $C>0$ sufficiently large and since $B_0(\Theta)<1$, we get $\sum_{s=1}^\infty s^{-C\log  (1/B_0(\Theta))}<\infty$. Moreover, for $s>e$ quote that if $t=s/C \log (s)$ we have  $s>Ct\log (t)$ and inverting of sums yields:
$$
\sum_{s=3}^\infty\sum_{t= \lfloor s/C \log(s)\rfloor }^\infty b_t^{(0)}(\Theta)\le C \sum_{t=1}^\infty t\log(t) b_t^{(0)}(\Theta)<\infty\,,
$$
and the desired result follows.
\end{proof}

\subsection{Coupling of local-stationary processes}

The $\tau$-dependence properties of the stationary process comes from the coupling scheme where we uniformly  approximate the stationary version $(\widetilde X_t(u))$ with a copy $(\widetilde X^*_t(u))$ that is independent of the past $(\widetilde X_t(u))_{t\le 0}$. The weak dependence notion will be used in order to get the uniform SLLN over functional of $(\widetilde X_t(u))$. The goal of this section is to extend such coupling approach  to non-stationary processes $(X_t^{(n)})$ in the sense that one approximates $(X_t^{(n)})$ in $\L^p$ with a certain coupled version. A useful remark is that we do not use the stationarity of the coupling process $(\widetilde X^*_t(u))$ for obtaining the recursion \eqref{recurcoupl}. Thus a similar coupling scheme can be extended to the local stationary process $( X_t^{(n)})$ but  locally only. In order to localize, we define $u\in [\varepsilon ,1-\varepsilon]$, $\varepsilon>0$ and   $n$ large enough the quantities
\begin{equation}\label{c}
i_n(u):=[n(u-ch_n)]\ge 1 \quad  \mbox {and}\quad j_n(u):=[n(u+ch_n)]\le n,
\end{equation}
where we recall that the compact support of the kernel $K$ is included in $[-c,c]$. 
\begin{Def}[Coupling process]
In the time-window $\{i_n(u),i_n(u)+1,\ldots,j_n(u)\}$ define the process $(X_{t}^*(u))_{i_n\leq t\le j_n}$ by
\begin{equation}
\label{eq:tangent}
 X_{t}^*(u)=\left\{\begin{array}{lll}
  X_t^{(n)},& t< i_n(u),\\
 F_{\boldsymbol{\theta}^*(u)}\big ((X_{t-k}^*(u))_{k\geq 1},\xi_t\big),&i_n(u) \leq t\le j_n(u). \end{array}
 \right.
 \end{equation}
 \end{Def}
Notice that for the ease of notation and as $n$ is fixed sufficiently large in this section, we suppress the dependence in $n$ on the coupling process. We first have to prove the existence of the process $( X_{t}^*(u))$. 
\begin{lem}\label{momm3} 
Let $\Theta \subset \R^d$ be such that {\bf(A$_0$($\Theta$))} holds with $B_0(\Theta)<1$ and assume that (LS$(\rho)$) also holds. Then, for any $u\in (0,1)$, there exists a.s. a unique coupling process $(X^*_{t}(u))_{t\in \Z}$ satisfying \eqref{eq:tangent} and there exists a positive constant $C^*>0$ such that 
$$
\big \| \! \! \!\sup_{u\in [\varepsilon ,1-\varepsilon]}\, \big |X^*_{ i_n(u)+s}(u) \big |\, \big \|_p\le  C^*\,n^{1/p}  \, \qquad\mbox{for all $~0\le  s\leq  2c\,nh_n$}.
$$ 
\end{lem}
\begin{proof}[Proof of Lemma \ref{momm3}] We use a chaining argument adapted to our framework. We denote $(u_k)$ where $u_k= (k+1/2+ch_n)/n$ for $k \in U_n(\varepsilon)=\big \{[\varepsilon n -ch_n-1/2],\ldots, [(1-\varepsilon)n -ch_n-1/2] \big \}$ a grid of points of the segment $[\varepsilon,1-\varepsilon]$ and therefore $\big |U_n(\varepsilon) \big |\simeq (1-2 \varepsilon)n\leq n$. Moreover, for each $u\in  [\varepsilon,1-\varepsilon]$ there exists $u_k$ such that $|u-u_k|\le 1/(2n)$ and therefore $i_n(u)=i_n(u_k)$ with $i_n(u)$ defined in \eqref{c}. Then a chaining argument shows that 
\begin{equation}\label{chaine}
\sup_{u\in [\varepsilon ,1-\varepsilon]} \big |X^*_{ i_n(u)+s}(u)\big |\le \max_{k\in U_n(\varepsilon)} \big |X^*_{ i_n(u_k)+s}(u_k)\big |+ \sup_{u,v:\, i_n(u)=i_n(v)}\big | X^*_{ i_n(u)+s}(u)- X^*_{ i_n(v)+s}(v)\big |\,.
\end{equation}
Using the inequality $\max(|x|,|y|)\leq |x|+|y|$ and similar argument than in the proof of Lemma \ref{momm}, we get that 
\begin{align}
\nonumber \big \|  \max_{k\in U_n(\varepsilon)} \big |X^*_{ i_n(u_k)+s}(u_k) \big |\big \|_p&\le  \Big \|  \sum_{k\in U_n(\varepsilon)}\big |X^*_{ i_n(u_k)+s}(u_k)\big |\Big \|_p\\
\nonumber &\le  \Big( \sum_{k\in U_n(\varepsilon)} \big \|X^*_{ i_n(u_k)+s}(u_k)\big \|_p^p\Big)^{1/p}\\
\label{np} &\le n^{1/p}\,  \dfrac{C_0(\Theta)}{1-B_0(\Theta)},
\end{align}
since $\|X^*_{ i_n(u)+s}(u)\|_p\le C_0(\Theta)/(1-B_0(\Theta))$ for any $u\in[\varepsilon,1-\varepsilon]$ and $0\le  s\leq  2c\,nh_n$. 
Set
$ \delta_n= \Big \|\sup_{u,v:\, i_n(u)=i_n(v)} \big| X^*_{ i_n(u)+s}(u)- X^*_{ i_n(v)+s}(v)\big| \Big \|_p$. We derive from an application of the chaining argument,
\begin{eqnarray*}
 \delta_n&\leq& \Big \|\sup_{u,v:\, i_n(u)=i_n(v)}\, \big| F_{\theta^*(i_n(u)+s)} \big ( (X^*_{ i_n(u)+s-k}(u))_{k\geq 1},\xi_{ i_n(u)+s} \big )
\\ && \qquad \qquad\qquad \qquad\qquad \qquad 
 - F_{\theta^*(i_n(u)+s)} \big ( (X^*_{ i_n(v)+s-k}(v))_{k\geq 1},\xi_{ i_n(v)+s} \big ) \big |\Big \|_p \\
&  +&\Big \|\sup_{u,v:\, i_n(u)=i_n(v)}\,\big| F_{\theta^*(i_n(u)+s)} \big ( (X^*_{ i_n(v)+s-k}(v))_{k\geq 1},\xi_{ i_n(v)+s} \big )\\
&&\qquad \qquad\qquad \qquad\qquad \qquad - F_{\theta^*(i_n(v)+s)} \big ( (X^*_{ i_n(v)+s-k}(v))_{k\geq 1},\xi_{ i_n(v)+s} \big ) \big | \Big \|_p\\
& \leq&  \sum_{k=1}^\infty b_k^{(0)}(\Theta) \, \Big \|\sup_{u,v:\, i_n(u)=i_n(v)}\big|X^*_{{ i_n(u)+s}-k}(u)-X^*_{{ i_n(v)+s}-k}(v)\big|\, \Big \|_p\\
& +&\sup_{u,v:\, i_n(u)=i_n(v)} \big\|{\boldsymbol{\theta}}^*(u)-{\boldsymbol{\theta}}^*(v)\big\|\\
&&\qquad \times \Big (\sum_{k=1}^\infty b^{(1)}_k(\Theta) \,\Big \| \! \sup_{u\in[\varepsilon ,1-\varepsilon]}\big|X^*_{{ i_n(u)+s}-k}(u)\big| \, \Big \|_p+  \Big \| \sup_{{\boldsymbol{\theta}} \in \Theta} \big \|\partial_{\boldsymbol{\theta}}^1F_{\boldsymbol{\theta}}(0;\xi_0)\big \|\, \Big \|_p \Big )\\
& \le &B_0(\Theta) \, \Big \|\sup_{u,v: i_n(u)=i_n(v)}\big|X^*_{{ i_n(u)+s}-k}(u)-X^*_{{ i_n(v)+s}-k}(v)\big| \,\Big \|_p
\\ & 
+& K_{\boldsymbol{\theta}} \, n^{-\rho}  \Big \| \sup_{{\boldsymbol{\theta}} \in \Theta} \big \|\partial_{\boldsymbol{\theta}}^1F_{\boldsymbol{\theta}}(0;\xi_0)\big \|\, \Big \|_p\\
& 
 +& K_{\boldsymbol{\theta}}  n^{-\rho}  B_1(\Theta) \Big(  \dfrac{C_0(\Theta)n^{1/p}}{1-B_0(\Theta)}
+\Big \|\sup_{u,v: i_n(u)=i_n(v)}\big|X^*_{{ i_n(u)+s}-k}(u)-X^*_{{ i_n(v)+s}-k}(v)\big|\,\Big \|_p\Big),
\end{eqnarray*}
from \eqref{convtheta}, \eqref{chaine} and \eqref{np}. Collecting all those bounds  for $n$ sufficiently large in order  that $n^{-\rho}$ is sufficiently small,  we get
\begin{multline*}
\Big \|\sup_{u,v:\, i_n(u)=i_n(v)} \big | X^*_{ i_n(u)+s}(u)- X^*_{ i_n(v)+s}(v) \big | \Big \|_p\\
\le\ \dfrac{\dfrac{K_{\boldsymbol{\theta}}B_1(\Theta)C_0(\Theta)}{1-B_0(\Theta)}\, n^{1/p-\rho}  + K_{\boldsymbol{\theta}} \, n^{-\rho}  \big \| \sup_{{\boldsymbol{\theta}} \in \Theta} |\partial_{\boldsymbol{\theta}}^1F_{\boldsymbol{\theta}}(0;\xi_0) |\big \|_p}{1-B_0(\Theta) -  K_{\boldsymbol{\theta}} \,  n^{-\rho}  B_1(\Theta)}\,.
\end{multline*}
Finally, applying again the chaining argument  we obtain
$$
\|   \max_{k\in U_n(\varepsilon)} \big |X^*_{ i_n(u_k)+s}(u_k)\big |\|_p \le  n^{1/p} \, \dfrac{C_0(\Theta)}{1-B_0(\Theta)} + O(n^{1/p-\rho})\,.
$$
\end{proof}
We point out that $( X_{t}^*(u))$ is not a copy of $(  X_t^{(n)})$ as it does not follow the same distribution. However this is a satisfactory non-stationary approximation of $(  X_t^{(n)})$ because we obtain the following coupling bound:
\begin{lem} \label{lem*}
Under  Assumptions ${\bf (A_0(\Theta))}$ with $B_0(\Theta)<1$, ${\bf (A_1(\Theta))}$ and ${\bf (LS(\rho))}$, with $(X_t^*(u))$ the coupling process defined in \eqref{eq:tangent}, there exists a positive constant $C'>0$ such that
\begin{equation}\label{couplebis}
\big \| \! \! \! \sup_{u\in[\varepsilon ,1-\varepsilon]}\big |X^{(n)}_{i_n(u)+s}-X_{i_n(u)+s}^*(u) \big | \, \big \|_p \leq C' \, h_n^\rho n^{1/p} \qquad\mbox{for all $0\le  s\leq  2c\,nh_n$}.
\end{equation}
\end{lem}
\begin{proof}[Proof of Lemma \ref{lem*}]
 Define  $\Delta^*_s=0$ for any $s\le 0$ and for $0\le  s\leq  2c\,nh_n$ set the quantity   
$ \Delta^*_s=\sup_{u\in[\varepsilon ,1-\varepsilon]} \big | X^{(n)}_{i_n(u)+s}-X_{i_n(u)+s}^*(u) \big |$. For $0\le  s\leq  2c\,nh_n$ we decompose
 \begin{eqnarray*}\nonumber
\Delta^*_s
  &=&\sup_{u\in[\varepsilon ,1-\varepsilon]} \big | F_{\boldsymbol{\theta}^{(n)}_{ i_n(u)+s}}\big ((X^{(n)}_{i_n(u)+s-k})_{k\geq 1},\xi_{ i_n(u)+s}\big )\\
  && \qquad  \qquad  \qquad\qquad \qquad -F_{\boldsymbol{\theta}^*(u)}\big ((X^*_{ i_n(u)+s-k}(u))_{k\geq 1},\xi_{ i_n(u)+s}\big ) \big |\\
  \label{inegcontr}
 &\leq &\sup_{u\in[\varepsilon ,1-\varepsilon]} \big | F_{\boldsymbol{\theta}^{(n)}_{ i_n(u)+s}}\big ((X^{(n)}_{{ i_n(u)+s}-k})_{k\geq 1},\xi_{ i_n(u)+s}\big )\\
  && \qquad  \qquad  \qquad \qquad \qquad-F_{\boldsymbol{\theta}^{(n)}_{ i_n(u)+s}}\big ((X^*_{{ i_n(u)+s}-k}(u))_{k\geq 1},\xi_{ i_n(u)+s}\big ) \big | \\
& +&\sup_{u\in[\varepsilon ,1-\varepsilon]} \big | F_{\boldsymbol{\theta}^{(n)}_{ i_n(u)+s}}\big ((X^*_{{ i_n(u)+s}-k}(u))_{k\geq 1},\xi_{ i_n(u)+s}\big )\\
  && \qquad  \qquad\qquad \qquad  \qquad -F_{\boldsymbol{\theta}^*  (u)}((X^*_{{ i_n(u)+s}-k}(u))_{k\geq 1},\xi_{ i_n(u)+s}\big ) \big |.
 \end{eqnarray*}
Then we derive
  \begin{eqnarray*}
  \|\Delta^*_s\|_p&\le& \Big \|\sup_{u\in[\varepsilon ,1-\varepsilon]} \big | F_{\boldsymbol{\theta}^{(n)}_{ i_n(u)+s}}\big ((X^{(n)}_{t-k})_{k\geq 1},\xi_t\big )\\
  && \qquad  \qquad  \qquad \qquad \qquad-F_{\boldsymbol{\theta}^{(n)}_{ i_n(u)+s}}\big ((X^*_{{ i_n(u)+s}-k}(u))_{k\geq 1},\xi_{ i_n(u)+s}\big ) \big | \Big\|_p \\
	& +& \Big \|\sup_{u\in[\varepsilon ,1-\varepsilon]} \big | F_{\boldsymbol{\theta}^{(n)}_{ i_n(u)+s}}\big ((X^*_{{ i_n(u)+s}-k}(u))_{k\geq 1},\xi_t\big )\\
  && \qquad  \qquad  \qquad\qquad \qquad -F_{\boldsymbol{\theta}^*  (u)}((X^*_{{ i_n(u)+s}-k}(u))_{k\geq 1},\xi_{ i_n(u)+s}\big ) \big | \Big \|_p
\\
&\le&\Big \|\sum_{k=1}^\infty b_k^{(0)}(\Theta) \, \sup_{u\in[\varepsilon ,1-\varepsilon]} \big |X^{(n)}_{{ i_n(u)+s}-k}-X^*_{{ i_n(u)+s}-k}(u) \big | \Big \|_p  \\
& +\!&\!\!\!\!\sup_{u\in[\varepsilon ,1-\varepsilon]}\!\! \big\|{\boldsymbol{\theta}}^{(n)}_{ i_n(u)+s}\!-\!{\boldsymbol{\theta}}^*(u)\big\| 
    \Big \| \sup_{{\boldsymbol{\theta}} \in \Theta} \sup_{u\in[\varepsilon ,1-\varepsilon]}\big \|  \partial_{\boldsymbol{\theta}}^1 F_{\boldsymbol{\theta}}((X^*_{{ i_n(u)+s}-k}(u))_{k\geq 1},\xi_{ i_n(u)+s}) \big \| \Big \|_p\\
 &\le&\Big \| \sum_{k=1}^\infty b_k^{(0)}(\Theta) \, \Delta^*_{s-k}\Big \|_p 
  +\sup_{u\in[\varepsilon ,1-\varepsilon]} \big\|{\boldsymbol{\theta}}^{(n)}_{ i_n(u)+s}-{\boldsymbol{\theta}}^*(u)\big\|
\\
 && \times \Big \|\sup_{u\in[\varepsilon ,1-\varepsilon]}\sum_{k=1}^\infty b^{(1)}_k(\Theta) \,\big |X^*_{{ i_n(u)+s}-k}(u) \big |+ \sup_{{\boldsymbol{\theta}} \in \Theta}  \big \| \partial_{\boldsymbol{\theta}}^1 F_{\boldsymbol{\theta}}(0;\xi_{ i_n(u)+s}) \big \| \Big \|_p\\
 &\le&\sum_{k=1}^\infty b_k^{(0)}(\Theta) \,  \big \|\Delta_{s-k}^* \big \|_p 
 +\sup_{u\in[\varepsilon ,1-\varepsilon]} \big\|{\boldsymbol{\theta}}^{(n)}_{ i_n(u)+s}-{\boldsymbol{\theta}}^*(u)\big\| 
\\
 && \times \Big (\sum_{k=1}^\infty b^{(1)}_k(\Theta) \,\Big \|  \!\sup_{u\in[\varepsilon ,1-\varepsilon]} \big |X^*_{{ i_n(u)+s}-k}(u) \big | \Big \|_p+  \Big \| \sup_{{\boldsymbol{\theta}} \in \Theta} \big \|\partial_{\boldsymbol{\theta}}^1F_{\boldsymbol{\theta}}(0;\xi_0) \big \|\Big \|_p \Big )
 \end{eqnarray*}
by using the Assumptions {\bf(A$_0$($\Theta$))} and {\bf(A$_1$($\Theta$))}. \\

Now define $ M_t^*= \max_{ s\le t} \big \|\Delta^*_s \big \|_p$. We derive from an application of Lemma \ref{momm3} that for any  $0\le  t\leq  2c\,nh_n$ it holds
$$
\big \|\Delta^*_t \big \|_p \leq B_0(\Theta) \, M_{t-1}^* +\sup_{u\in[\varepsilon ,1-\varepsilon]} \big\|{\boldsymbol{\theta}}^{(n)}_{ i_n(u)+t}-{\boldsymbol{\theta}}^*(u)\big\| \times \Big (C^*\, n^{1/p}+C_1(\Theta) \Big ).
$$
We have $\displaystyle \sup_{u\in[\varepsilon ,1-\varepsilon]} \big\|\boldsymbol{\theta}^{(n)}_{ i_n(u)+t}-\boldsymbol{\theta}^*(u)\big\| \leq c \, K_\theta \, h_n^\rho$ from condition \eqref{convtheta} of Assumption {\bf (LS($\rho$))}. 
As a consequence, for any $0\le  t\leq  2c\,nh_n$, we have
$$
M_{t}^*\leq B_0(\Theta) \, M_{t-1}^* +c \, K_\theta \,\Big (C^* n^{1/p}+C_1(\Theta) \Big )\,  {h_n^\rho}.
$$
By definition, $M_{0}^*=0$. Therefore we deduce for any $0\le  t\leq  2c\,nh_n$,
$$
M_{t}^*\leq \frac {c \, K_\theta}{1-B_0(\Theta)} \,\Big ( C^* n^{1/p}+C_1(\Theta) \Big )\,  {h_n^\rho}.
$$
This completes the proof of the Lemma \ref{lem*}. 
\end{proof}
Finally, the coupling process $(X^*_{t}(u))_{t\in \Z}$ is also used for estimating the approximation of $X_t^{(n)}$ with the stationary version $\widetilde X_t(u)$ for $t/n\simeq u$. 
\begin{lem} \label{lemtilde}
Under  Assumptions ${\bf (A_0(\Theta))}$ with $B_0(\Theta)<1$, ${\bf (A_1(\Theta))}$ and ${\bf (LS(\rho))}$ there exists a positive constant $C''>0$ such that
\begin{equation}\label{couple2}
\big \| \! \! \! \sup_{u\in[\varepsilon ,1-\varepsilon]}\big |X^{(n)}_{i_n(u)+s}-\widetilde X_{i_n(u)+s}(u) \big | \, \big \|_p \leq C'' \, n^{1/p} \big ( h_n^\rho +  \lambda_{s}\big ),\quad\mbox{for all }0\le  s\leq  2c\,nh_n.
\end{equation}
\end{lem}
\begin{proof}[Proof of Lemma \ref{lemtilde}]
The approximation is  derived since the coupling process is a non stationary coupling version of $\widetilde X^\circ_t(u)$ defined in \eqref{couple}. Indeed, using $(X_t^*(u))$ the coupling process defined in \eqref{eq:tangent} and repeating the same arguments than above, we obtain  a similar recursive relation than \eqref{recurcoupl} on the sequence 
$$
\Big (\max_{ t \ge k r } \Big \|\sup_{u\in[\varepsilon ,1-\varepsilon]} \big| \widetilde X_{i_n(u)+t}(u)-  X^*_{i_n(u)+t}(u)\big |\Big \|_p \Big )_{k\ge 0}
$$ given any $r\in \N^*$, using Lemma \ref{momm3} and the estimates
\begin{multline*}
\Big \|
\sup_{u\in[\varepsilon ,1-\varepsilon]} \big | \widetilde X_{i_n(u)+s}(u)\big |  \Big\|_p\le \Big\|\max_{1\le t\le n} \sup_{u\in[\varepsilon ,1-\varepsilon]}\big | \widetilde X_{t}(u)\big |  \Big\|_p \\
\le \Big(\sum_{1\le t\le n} \Big\|\sup_{u\in[\varepsilon ,1-\varepsilon]}\big |  \widetilde X_{t}(u)\big | \Big\|_p^p\Big)^{\frac 1p}\le C\,n^{\frac1p}.
\end{multline*} 
We obtain
$$
\Big \|  \!\sup_{u\in[\varepsilon ,1-\varepsilon]} \big |X_{i_n(u)+s}^{*}(u)-\widetilde X_{i_n(u)+s}(u) \big | \Big \|_p \le C \, n^{1/p} \, \lambda_s\,,\qquad \mbox{for}\quad 0\leq s \leq 2c\,nh_n,
$$
with $\lambda_s$ defined in  \eqref{tauX}. Combining this result with \eqref{couple2} we bound the $\L^p$ norm of the approximation  of $X_s^{(n)}$ with the stationary version $\widetilde X_s(u)$, namely for for all $0\leq s \leq 2c\,nh_n$,
\begin{align*}
\Big \| \! &\sup_{u\in[\varepsilon ,1-\varepsilon]} \big |X^{(n)}_{i_n(u)+s}-\widetilde X_{i_n(u)+s}(u) \big | \Big \|_p
\\
&\le \Big \| \! \sup_{u\in[\varepsilon ,1-\varepsilon]} \big |X^{(n)}_{i_n(u)+s}-X_{i_n(u)+s}^*(u) \big | \Big \|_p
\nonumber  +\Big \| \! \sup_{u\in[\varepsilon ,1-\varepsilon]} \big |X_{i_n(u)+s}^*(u)-\widetilde X_{i_n+s}(u) \big | \Big \|_p\\
 &\le C\, n^{1/p} \big ( h_n^\rho +  \lambda_{s}\big ).\end{align*}
\end{proof}

\section{Proofs for the Section \ref{Mesti}}\label{proofs2}
\subsection{Some useful lemmas}
\begin{proof}[Proof of Lemma \ref{exiphi}]
For $\boldsymbol{\theta} \in \Theta$, $t \in \Z$ and $m \in \N$, define 
$$\phi_{t,m}=\Phi \big (X_t^{(n)},X_{t-1}^{(n)},\ldots, X_{t-m}^{(n)},0,0,\ldots;\boldsymbol{\theta} \big ).$$ 
As $\Phi \in \Lip_{p}(\Theta)$ the sequence $(\phi_{t,m})_{m\in \N}$ is a Cauchy sequence in $\L^q$ since for any $m_2> m_1$
\begin{eqnarray}
\big \| \phi_{t,m_2}-\phi_{t,m_1} \big \|_1 &\leq &g \big (\sup_{0\leq s\leq m_2} \big \{\| X_{t-s}^{(n)} \|_p \big \} \big ) \,\sum_{k=m_1+1}^{m_2}\alpha_k(h, \Theta) \, \| X_{t-k}^{(n)}  \|_p \\
& \leq & C \, \sum_{k=m_1+1}^{m_2}\alpha_k(h, \Theta)  
\end{eqnarray}
from Lemma \ref{momm}, since if $s<0$ then $X_{s}=0$, thus the corresponding supremum bound extends over each $s\le n$. \\
As $\sum_{k=1}^{\infty}\alpha_k(h, \Theta) <\infty$ we deduce that for any $\varepsilon >0$,   $ \sum_{k=m_1+1}^{m_2}\alpha_k(h, \Theta) \leq \varepsilon$ for $m_1$ and $m_2$  large enough. Using the completeness of $\L^1$ we deduce the consistency  of the sequence $(\phi_{t,m})_{m\in \N}$ and the existence in $\L^1$ of its limit $\Phi \big ((X_{t-k}^{(n)})_{k\geq 0},\boldsymbol{\theta} \big )$. \\
~\\
When $\boldsymbol{\theta} _t^{(n)}=\boldsymbol{\theta} ^*(u)$ for any $t,n$, we consider  $\Phi \big ((\widetilde X_{t-k}(u))_{k\in \N},\boldsymbol{\theta} \big )$ that also exists in $\L^1$. Moreover, as $(\widetilde X_{t-k}(u))_{k\in \N}$ is a stationary ergodic process, this is also the case for  $ \big (\Phi \big ((\widetilde X_{t-k}(u))_{k\in \N},\boldsymbol{\theta} \big )\big )_{t\in \Z}$ (see Corollary 2.1.3. in \cite{Stra05}). 
\end{proof}
\begin{lem} \label{TLCK}\
Let $i_n(u))$ and $j_n(u)$ defined in \eqref{c}. 
\begin{enumerate}
\item Let $Z(u)=(Z_t(u))_{t\in \N}$ be a centered stationary process on a Banach space $(\mathbb B, \|\cdot\|)$ for any $0\le u\le 1$. If $Z(u)$ is an ergodic process continuous with respect to $u$ and satisfying $\E[\sup_{0\le u\le 1}\|Z_0(u)\|]<\infty$ then   we have
\begin{equation}\label{LGNnoyau} 
\sup_{0< u< 1} \left\| \frac 1 {n\, h_n } \sum_{t=i_n(u)}^{j_n(u)}  Z_t(u)  \, K\Big(\frac{\frac tn-u}{h_n}\Big)\right\|  \limiteasn 0.
\end{equation}
\item Let $Z=(Z_t)_{t\in \N}$ be a centered stationary process on $\R^d$ such as $\E\big [ \| Z_0 \|^2 \big ]<\infty$ and $0<u<1$.
\\
If $\frac 1 {\sqrt n} \, \sum_{t=1}^{n} Z_t \limiteloin {\cal N}_d \big (0\, , \, \Sigma\big )$ with $\Sigma$ a positive definite symmetric matrix, then we have
\begin{equation}\label{TLCnoyau} 
 \frac 1 {\sqrt{n\, h_n} } \, \sum_{t=i_n(u)}^{j_n(u)} Z_t \, K\Big(\frac{\frac tn-u}{h_n}\Big)  \limiteloin {\cal N} \Big ( 0 \, , \,\Big (\int_\R K^2(x)\,dx\Big ) \, \Sigma \Big ).
\end{equation}
\end{enumerate}
\end{lem}
\begin{proof}[Proof of Lemma \ref{TLCK}]
Let $\ell \in \N^*$, $[-c,c]$ be the compact support of $K$ and let $n$ be large enough such as $n(u-c\, h_n) \geq 1$ and $nu+c\,n h_n \leq n$.  Then,  for $j \in \{1,\ldots,\ell\}$, we denote $I^{(\ell)}_j=\big [-c+2 c \, \frac {j-1} \ell \, , \, -c+2 c \, \frac {j} \ell \big ]$ , $T^{(\ell)}_j=\Big \{ t \in \N, \, \frac{\frac tn-u}{h_n} \in I_j \Big \}$ 	and 
\begin{equation}\label{Snl}
S_n^{(\ell)}(u) = \frac 1 {n h_n} \, \sum_{t=i_n(u)}^{j_n(u)} Z_t(u)  \, K\Big(\frac{\frac tn-u}{h_n}\Big)\quad\mbox{and}\quad S^{(\ell)}_{n,j}(u) =\frac 1 {nh_n} \, \sum_{t\in T^{(\ell)}_j} Z_t(u)\,.
\end{equation}
Below, we will omit the reference to the dependence with respect to $u$ when no confusion will be possible.
 
1. We notice that $(Z_t(u))$ is a centered ergodic process on the Banach space $\L^1(\mathcal C([0,1],B))$. Thus for any fixed $\ell\in \N^*$ such that $\mbox{Card}(T^{(\ell)}_j)\simeq 2 c nh_n/ \ell \limiten \infty$ we apply the uniform ergodic theorem and since $\E[Z_0(u)]=0$ for any $0< u<1$ we obtain
\begin{equation}\label{convSl}
\sup_{0\le u\le 1}\|S^{(\ell)}_{n,j}(u)\| \limiteasn 0.
\end{equation}
Denote $t^{(\ell)}_j=-c+c \, \frac {2j-1} \ell$,  the midpoint of $T^{(\ell)}_j$, then we have
\begin{equation}\label{decoup}
S_n^{(\ell)}(u) = \sum_{j=1}^{\ell }  K\Big(\frac{\frac {t^{(\ell)}_j}n-u}{h_n}\Big)\, S^{(\ell)}_{n,j}(u) 
+ \sum_{j=1}^{\ell } \frac 1 {nh_n}\! \sum_{t\in T^{(\ell)}_j} Z_t(u)  \Big [ K\Big(\frac{\frac tn\!-\!u}{h_n}\Big)-K\Big(\frac{\frac {t^{(\ell)}_j}n\!-\!u}{h_n}\Big) \Big ].
\end{equation}
First, since  $K$ is a bounded function and from \eqref{convSl}, then for any $\ell \in \N^*$ we obtain
\begin{equation}\label{part1}
\sup_{0\le u\le 1}\Big\|\sum_{j=l}^{\ell }  K\Big(\frac{\frac {t^{(\ell)}_j}n-u}{h_n}\Big)\, S^{(\ell)}_{n,j}(u)\Big\|  \limiteasn 0.
\end{equation}
Second, since $K$ is a ${\cal C}^1$ function on $[-c,c]$ it holds  
$$
\sup_{0\le u\le 1}\ \max _{1\leq j \leq \ell} \ \sup _{t \in T^{(\ell)}_j} \Big | K\Big(\frac{\frac tn-u}{h_n}\Big)-K\Big(\frac{\frac {t^{(\ell)}_j}n-u}{h_n}\Big) \Big |\leq   \frac c \ell \|K'\|_\infty
$$ 
for any $\ell\in \N^*$ and any $n\in \N$. Then we obtain
\begin{multline*}
\sup_{0\le u\le 1} \Big \|\sum_{j=1}^{\ell } \frac 1 {n h_n} \sum_{t\in T^{(\ell)}_j} Z_t(u) \, \Big [ K\Big(\frac{\frac tn-u}{h_n}\Big)-K\Big(\frac{\frac {t^{(\ell)}_j}n-u}{h_n}\Big) \Big ] \Big \|\\ \leq \|K'\|_\infty \, \frac c \ell \cdot  \frac 1 {n h_n} \sum_{t=1}^{n} \sup_{0< u< 1}\|Z_t(u)\|. 
\end{multline*}
The ergodicity of $\big( \sup_{0<u< 1}\|Z_t(u)\|\big)_{t}$,  its stationarity  and $\E[\sup_{0< u< 1}\|Z_0(u) \|]<\infty$ together  yield
$$
\frac 1 {n  h_n} \sum_{t=i_n(u)}^{j_n(u)}  \sup_{0< u< 1}\|Z_t(u) \| \limiteasn \E\Big[\sup_{0< u< 1} \|Z_0(u) \|\Big]\,. 
$$ 
Thus, for any $\varepsilon>0$, there exists a.s. $(\ell_0,n_0)$ such as for any $\ell\geq \ell_0$ and $n\geq n_0$,  
\begin{equation}\label{part2}
\sup_{0< u< 1}\Big \|\sum_{j=1}^{\ell } \frac 1 {n  h_n} \sum_{t\in T^{(\ell)}_j} Z_t(u) \, \Big [ K\Big(\frac{\frac tn-u}{h_n}\Big)-K\Big(\frac{\frac {t^{(\ell)}_j}n-u}{h_n}\Big) \Big ] \Big \| \leq \varepsilon\,\qquad a.s.
\end{equation}
From \eqref{decoup}, \eqref{part1} and \eqref{part2}, we deduce \eqref{LGNnoyau}.\\
~\\
2. Consider first  $K= K_\ell$ the piecewise constant function $K_\ell(x)=\sum_{j=1}^\ell a_j \, \1 _{x \in I^{(\ell)}_j}$, and also assume first that $d=1$ with $\Sigma=\sigma^2>0$ (unidimensional case).
\begin{equation}\label{Sn}
S_n^{(\ell)}=\sum_{j=1}^\ell a_j \, S^{(\ell)}_{n,j},
\end{equation}
with $S^{(\ell)}_{n,j}$ defined in \eqref{Snl}. Using $\mbox{Card}(T^{(\ell)}_j)\sim 2 c nh_n/ \ell \limiten \infty$, for any $j\in \{1,\ldots,\ell\}$,
\begin{equation}\label{Snj}
\Big ( \frac \ell{2c n h_n} \Big )^{1/2} \sum_{t\in T^{(\ell)}_j} Z_t = \sqrt{\frac {n  h_n \ell}{2c}}\; S^{(\ell)}_{n,j} \limiteloin {\cal N} \Big ( 0 \, , \, \sigma^2\Big ),
\end{equation}
where $\sigma^2=\sum_{t\in \Z} \E [Z_0Z_t] $ is such as $0<\sum_{t\in \Z} \E [Z_0Z_t]<\infty$. 
Moreover, using the stationarity of $Z$, we have for any $j,j'\in \{1,\ldots,\ell\}$ such as $j\neq j'$,
\begin{eqnarray*}
n  h_n \, \cov \big (S^{(\ell)}_{n,j}\, , \, S^{(\ell)}_{n,j'} \big )&=&\frac 1 {n  h_n} \, \sum_{t\in T^{(\ell)}_j}\sum_{t'\in T^{(\ell)}_{j'}} \E [Z_tZ_{t'} ]
\,=\, \frac 1 {n  h_n} \, \sum_{t\in T^{(\ell)}_j}\sum_{t'\in T^{(\ell)}_{j'}} \E [Z_0Z_{t'-t} ]\\
\Longrightarrow \quad \Big |n  h_n \, \cov \big (S^{(\ell)}_{n,j}\, , \, S^{(\ell)}_{n,j'} \big ) \Big | & \leq & C \, \Big | \sum_{k>  T^{(\ell)}_j} \E [Z_0Z_{k} ]\Big | \limiten 0,
\end{eqnarray*}
as soon as $\ell=o(nh_n)$ since $\sum_{t\in \Z}\big | \E [Z_0Z_t]\big |<\infty$. Hence a central limit theorem holds for any linear combinations of $S^{(\ell)}_{n,j}$.  {E.g.}  
 for $S^{(\ell)}_n$ defined in \eqref{Sn},
\begin{multline}
\sqrt{\frac {n  h_n \ell}{2c}}\, S^{(\ell)}_n \limiteloin  {\cal N} \Big ( 0 \, , \, \sum_{j=1}^\ell a_j^2 \, \sum_{t\in \Z} \E [Z_0Z_t] \Big ) \\
\Longrightarrow \quad \sqrt{n  h_n }\, S^{(\ell)}_n \limiteloin  {\cal N} \Big ( 0 \, ,  \,\sigma^2 \,\int_\R K_\ell^2(x)\, dx\Big ),
\end{multline}
since for piecewise kernel we have $\displaystyle\int_\R K_\ell^2(x)\, dx= \frac {2c}{\ell}\sum_{j=1}^\ell a_j^2$.

Consider now a general piecewise differentiable kernel $K$ and denote $K_\ell$ such as $K_\ell(x)=\sum_{j=1}^\ell a_j\, \1 _{x \in I_j}$ for $x\in \R$ with $a_j=K\big (-c+c \, \frac {2j-1} \ell \big )$ such that
\[
 \int_\R K_\ell^2(x)\, dx \, \limitel  \int_\R K^2(x)\, dx .
\]
The  result will follow from  Theorem 3.2 in  \cite{bill}. For this we check  that for any $\varepsilon>0$ 
\[
\lim_{\ell\to \infty}\lim\sup_{n\to \infty} \P\Big(\Big|S_n^{(\ell)}-S_n \Big|\ge \varepsilon/\sqrt{nh_n}\Big) =0,\quad  \mbox{  with }\quad S_n=\frac 1 {n  h_n} \, \sum_{t=1}^{n} Z_t \, K\Big(\frac{\frac tn-u}{h_n}\Big).
\]
 From Markov inequality 
\[
\P\Big(\Big|S_n^{(\ell)}-S_n \Big|\ge \varepsilon /\sqrt{nh_n}\Big) \le nh_n \,\dfrac{\E[\Delta_n^2]}{\varepsilon^2}
\]
with
\begin{eqnarray*}
\Delta_n&=&  \frac 1 {n  h_n} \, \sum_{t=i_n(u)}^{j_n(u)}Z_t \, K\Big(\frac{\frac tn-u}{h_n}\Big)- \sum_{j=1}^\ell a_j \, S^{(\ell)}_{n,j}   \\
&=& \frac 1 {nh_n} \, \sum_{j=1}^\ell \sum_{t \in T_j}\underbrace{ \Big ( K\Big(\frac{\frac tn-u}{h_n}\Big)-K\big (-c+2 c \, \frac {j-1} \ell \big )\Big )}_{a_{t,j}(\ell) }\,   Z_t \,.
\end{eqnarray*}
Thus
\begin{eqnarray*}
\E[\Delta_n^2]&=& \frac 1 {( nh_n)^2} \sum_{j=1}^\ell \sum_{t \in T_j}\sum_{j'=1}^\ell \sum_{t' \in T_j'} a_{t,j}(\ell) \, a_{t',j'}(\ell)  \E[ Z_tZ_{t'}]\\
&\le & \frac 1 { (nh_n)^2}\sum_{j=1}^\ell \sum_{t \in T_j} \sum_{j'=1}^\ell \sum_{t' \in T_j'} |a_{t,j}(\ell) || a_{t',j'}(\ell)|  |\E[ Z_tZ_{t'}]|\,.
\end{eqnarray*}
Since $K$ is Lipschitz continuous, there exists a constant $C>0$ such that
\[
|a_{t,j}(\ell) |\le 2 \, \frac {c\,C}\ell\,,\qquad \mbox{ for }\  i_n\le t\le j_n\ 1\le j\le \ell\,.
\]
Thus
$$
\E[\Delta_n^2]\le \frac{C^2}{(nh_n)^2} \left(\frac {2c}\ell\right)^2  \sum_{1\le t, t'\le n}|\E[ Z_tZ_{t'}]|\le 8 \,\frac{C^2}{nh_n}\left(\frac c\ell\right)^2\sum_{t\ge 0}\big |\E[ Z_tZ_{t'}]\big |\,.
$$
Then we obtain, again from Markov inequality that,  for any $\varepsilon>0$,
$$\displaystyle
\P\Big(\Big|S_n^{(\ell)}-S_n \Big|\ge \varepsilon /\sqrt{nh_n}\Big)\le  2C^2 \left(\frac c\ell\right)^2\sum_{t\ge 0}\big |\E[ Z_tZ_{t'}] \big |\limitel 0.
$$
Now the extension to $d>1$ is standard: consider  $r=(r_1,\ldots,r_d)^T\in \R^d$ and a linear combination  ${\cal Z}_t= r_1Z_t^{(1)}+\cdots+r_dZ_t^{(d)}$ where  $Z_t=\big (Z_t^{(1)},\ldots,Z_t^{(d)}\big)^T$ and apply the result obtained for $d=1$. Then the asymptotic covariance matrix $r^T\,\Sigma \, r>0$ appears to be positive definite and this implies the multidimensional central limit theorem.
\end{proof}

\subsection{Proofs of the main results}

We will prove \eqref{convtheta2} in Theorem \ref{theo1} only since the consistency is achieved directly by simple arguments. We need the following Lemma that is a strong law of large number on the contrast  as if the stationary versions were observed:

\begin{lem}\label{lem:statwlln}
Under the assumptions of Theorem \ref{theo1} we have
\begin{multline}\label{as}
\sup_{\varepsilon\le u\le 1-\varepsilon} \sup_{\boldsymbol{\theta} \in \Theta}  \Big | \frac  1{nh_n}\sum_{t=i_n(u)}^{j_n(u)} \Phi(\widetilde X_{t-k}(u))_{k\in \N},\boldsymbol{\theta})K\big(\frac{\frac tn-u}{h_n}\big) -\E \big [\Phi \big ((\widetilde X_{-k}(u))_{k\geq 0},\boldsymbol{\theta}\big ) \big ]   \Big | \\
\limiteasn 0.
\end{multline}
\end{lem}

\begin{proof}
The expression $I$ in \eqref{as}  tends to $0$ a.s., if is it the case for  $I_1$ and $I_2$ such that, 
\begin{eqnarray*}
I_1\!\!&=&\!\!\!\!\sup_{\varepsilon\le u\le 1-\varepsilon} \sup_{\boldsymbol{\theta} \in \Theta} \, \Big | 
\frac  1{nh_n}\sum_{t=i_n(u)}^{j_n(u)}  \Big (\Phi(\widetilde X_{t-k}(u))_{k\in \N},\boldsymbol{\theta})-\E \big [\Phi \big ((\widetilde X_{-k}(u))_{k \in \N},\boldsymbol{\theta}\big ) \big ] \Big ) \, K\Big(\frac{\frac tn-u}{h_n}\Big) \Big |   \\
\nonumber
I_2\!\!&=&\!\!\!\!\sup_{\varepsilon\le u\le 1-\varepsilon}\sup_{\boldsymbol{\theta} \in \Theta} \,\Big |
\frac  1{nh_n}\sum_{t=i_n(u)}^{j_n(u)}   K\Big(\frac{\frac tn-u}{h_n}\Big) \, \E \big [\Phi \big ((\widetilde X_{-k}(u))_{k \in \N},\boldsymbol{\theta}\big )\big ]-\E \big [\Phi \big ((\widetilde X_{-k}(u))_{k \in \N},\boldsymbol{\theta}\big ) \big ] \Big |. 
\end{eqnarray*}
1. We use the part 1. of Lemma \ref{TLCK}  to control $I_1$. For this we define $Z(\boldsymbol{\theta},u)=(Z_t(\boldsymbol{\theta},u))_{t\in \Z}$ with $Z_t(\boldsymbol{\theta},u)=\Phi \big ((\widetilde X_{t-k}(u))_{k\in \N},\boldsymbol{\theta})-\E \big [\Phi \big ((\widetilde X_{-k}(u))_{k \in \N},\boldsymbol{\theta}\big )\big ]$; this is a centered  ergodic stationary process on the Banach space of the continuous function over $\Theta\times [0,1]$ equipped with the uniform norm. Using $\E \big [ \sup_{(\boldsymbol{\theta},u) \in \Theta\times [0,1]} |Z_0(\boldsymbol{\theta},u)| \big ] <\infty$ since $\Phi \in \Lip_{p}(\Theta)$, with Theorem 2.2.1. in \cite{Stra05} we  apply the part 1. of Lemma \ref{TLCK}  to get
\begin{equation}\label{suplgn}
I_1  \limiteasn 0.
\end{equation}
~\\
2. For the term $I_2$, notice that
\begin{eqnarray}
\nonumber I_2  &\leq & \sup_{\varepsilon\le u\le 1-\varepsilon}\ \sup_{\boldsymbol{\theta} \in \Theta} \, \Big | \Big (  1- \frac  1{nh_n}\sum_{t=i_n(u)}^{j_n(u)} K\Big(\frac{\frac tn-u}{h_n}\Big) \Big ) \,\E \big [\Phi \big ((\widetilde X_{-k}(u))_{k \in \N},\boldsymbol{\theta}\big ) \big ]  \Big | \\
\label{I2} &\leq & \, C \,\sup_{\varepsilon\le u\le 1-\varepsilon}\Big | 1- \frac  1{nh_n}\sum_{t=i_n(u)}^{j_n(u)} K\Big(\frac{\frac tn-u}{h_n}\Big) \Big | 
\leq \frac C{nh_n},
\end{eqnarray}
from the usual comparison of a Riemann sum and its integral: indeed $K$ is Lipschitz because it is  a piecewise differentiable function with a compact support.\\
As a consequence, the proof is complete from \eqref{suplgn} and \eqref{I2}.
\end{proof}

We also need the uniform approximation of the contrast with its stationary version stated in the next Proposition.
\begin{prop} \label{prop1}
Under the assumptions of Theorem \ref{theo1} with $(\widetilde X_{t}(u))_t$ denoting the stationary process defined in \eqref{eq:statu}, we obtain
\begin{equation}\label{P}
\sup_{u\in [\varepsilon,1-\varepsilon]}\  \sup _{\boldsymbol{\theta} \in \Theta} \Big | \frac 1{nh_n} \, \sum_{k=1}^n \Phi \big ((X^{(n)}_{k-t})_{t\geq 0},\boldsymbol{\theta} \big )  K\Big (\frac{\frac kn-u}{h_n}\Big)- \E \big [\Phi \big ((\widetilde X_{-k}(u))_{k\geq 0},\boldsymbol{\theta}\big ) \big ] \Big | \limiteproban 0.
\end{equation}
\end{prop}

\begin{proof}[Proof of Proposition \ref{prop1}]
Since $\Phi \in \Lip_{p}(\Theta)$ with  $p \geq 1$, we have
\begin{multline*}
\Big \|\sup_{u\in [\varepsilon,1-\varepsilon]}\  \sup_{\boldsymbol{\theta} \in \Theta}  \,  \frac  1{nh_n} \, \sum_{t=i_n(u)}^{j_n(u)} \big (\Phi\big ((X^{(n)}_{t-k}))_{k\geq 0}, \boldsymbol{\theta}\big )-\Phi \big ((\widetilde X_{t-k}(u))_{k\geq 0},\boldsymbol{\theta}\big ) \big )\, K\Big (\frac{\frac tn-u}{h_n}\Big ) \Big \|_1 \\
\le\ \frac C{nh_n} \, \sum_{t=0}^{2c\, nh_n}  \Big |K\Big (\frac{\frac {i_n(u)+t}n-u}{h_n}\Big )\Big | \,  g\big (\sup_{s \le j_n(u)} \big \{\|X^{(n)}_{ s} \|_p \vee \|\widetilde X_{ s}(u) \|_p \big \} \big ) \, \\
\times \sum_{s=1}^\infty \alpha_s(\Phi, \Theta)
 \big \|\sup_{u\in [\varepsilon,1-\varepsilon]} \big |X^{(n)}_{{i_n(u)+t}+1-s}-\widetilde X_{{i_n(u)+t}+1-s}(u) \big |\,\big \|_p.
\end{multline*}
From Assumption ${\bf (A_0(\Theta))}$ and $B_0(\Theta)\!<\!1$, 
$$
 \Big\|\sup_{u\in [\varepsilon,1-\varepsilon]}|X^{(n)}_{i_n(u)+j}-\widetilde X_{i_n(u)+j}(u) | \Big\|_p \leq  C  n^{1/p},
$$ 
for $j\leq  0$ using similar arguments as in the proof Lemma  \ref{lemtilde}. Moreover with ${\bf (A_1(\Theta))}$ and ${\bf (LS(\rho))}$, we apply \ref{lemtilde} in order to get
$$
\displaystyle \Big\|\sup_{u\in [\varepsilon,1-\varepsilon]}|X^{(n)}_{i_n(u)+j}-\widetilde X_{i_n(u)+j}(u) |\, \Big \|_p \leq n^{1/p} \big ( h_n^\rho +  \lambda_{j}\big ).
$$ 
for $j\geq   1$. Therefore,
\begin{multline}
\label{ineg1}
\Big \| \sup_{u\in [\varepsilon,1-\varepsilon]}\; \sup_{\boldsymbol{\theta} \in \Theta}  \,  \frac  1{nh_n} \sum_{t=0}^{2c\, nh_n}\big (\Phi\big ((X^{(n)}_{t-k}))_{k\geq 0}, \boldsymbol{\theta}\big )-\Phi \big ((\widetilde X_{t-k}(u))_{k\geq 0},\boldsymbol{\theta}\big ) \big ) K\Big (\frac{\frac tn-u}{h_n}\Big ) \Big \|_1 \\
 \le  \frac C {nh_n} \, \sum_{t=0}^{2c\, nh_n}C_K \, C^*  \, \Big ( \sum_{s=1}^t \alpha_s(\Phi, \Theta) \, C\, n^{1/p} \big ( h_n^\rho +  \lambda_{t+1-s}\big )+\sum_{s=t+1}^t \alpha_s(\Phi, \Theta) \, C \, n^{1/p} \Big )\,.
\end{multline}
Then we deduce
\begin{align}
\nonumber
 \Big \|\sup_{u\in [\varepsilon,1-\varepsilon]}&  \sup_{\boldsymbol{\theta} \in \Theta}  \,  \frac  1{nh_n} \,\sum_{t=0}^{2c\, nh_n}\big (\Phi\big ((X^{(n)}_{t-k}))_{k\geq 0}, \boldsymbol{\theta}\big )-\Phi \big ((\widetilde X_{t-k}(u))_{k\geq 0},\boldsymbol{\theta}\big ) \big )\, K\Big (\frac{\frac tn-u}{h_n}\Big ) \Big \|_1    \\
\nonumber 
& \le \frac  { C \, n^{1/p}}{nh_n} \, \sum_{t=i_n(u)}^{j_n(u)} \Big (\sum_{s=1}^{t-i_n}   \,  \alpha_s(\Phi, \Theta) \,\big ( \lambda_{t-s-i_n}+h_n^\rho\big ) +  \,  \sum_{s=t-i_n+1}^\infty \alpha_s(\Phi, \Theta)  \Big ) \\
\nonumber 
& \le \frac  { C \, n^{1/p}}{nh_n} \, \Big (    \sum_{k=1}^{j_n-i_n} \lambda_k \sum_{i=1}^{k} \alpha_i(\Phi, \Theta) +    h_n^\rho \! \sum_{k=1}^{j_n-i_n}\!\!\!  \sum_{i=1}^{k} \alpha_i(\Phi, \Theta)+\!\! \sum_{i=1}^{\infty} i  \alpha_i(\Phi, \Theta)  \Big )\\
\nonumber 
& \le \frac  {C  n^{1/p}}{nh_n} \, \Big (   \sum_{k=1}^{\infty} \lambda_k \sum_{i=1}^{\infty} \alpha_i(\Phi, \Theta) + \big (    h_n^\rho +1 \big ) \sum_{i=1}^{\infty} i  \alpha_i(\Phi, \Theta)  \Big )\\
\label{ineg2}   & \le C \,   \frac  { n^{1/p}}{nh_n}.
\end{align}
Here we made  use of  the assumption $
\sum_{s=1}^\infty s \alpha_s(\Phi, \Theta) <\infty$ and of the fact that $
\sum_{k=1}^{\infty} \lambda_k <\infty$.  This last bound
 holds if $
  \sum_{t=1}^\infty t\log(t) b_t(\Theta)<\infty$ and  follows from Lemma \ref{lem:lambda}. Finally, using \eqref{ineg1}, \eqref{ineg2} and the almost sure convergence \eqref{as} we obtain the weak consistency result  \eqref{P}.
\end{proof}
\begin{proof}[Proof of Theorem \ref{theo1}]
From the  Assumption {\bf(Co}($\Phi,\Theta$), we have $$\displaystyle \boldsymbol{\theta}^*(u)=\mbox{Arg}\! \min_{\! \! \! \!\boldsymbol{\theta}\in \Theta}\E \big [\Phi \big ((\widetilde X_{-t}(u))_{t\geq 0},\boldsymbol{\theta}\big ) \big ].$$
The uniform weak law of large numbers implies  the uniform convergence, and we need: 
$$
\sup_{u\in [\varepsilon,1-\varepsilon]}\;\sup _{\boldsymbol{\theta} \in \Theta} \Big | \frac 1{nh_n} \, \sum_{k=1}^n \Phi \big ((X^{(n)}_{k-t})_{t\geq 0},\boldsymbol{\theta} \big ) \, K\Big (\frac{\frac kn-u}{h_n}\Big)- \E \big [\Phi \big ((\widetilde X_{-k}(u))_{k\geq 0},\boldsymbol{\theta}\big ) \big ] \Big | \limiteproban 0\,,
$$
see the discussion in the Appendix of \cite{DRW}. From an application of the approximation in Proposition \ref{prop1}, this uniform weak law of large number follows from the uniform weak law of large number on the stochastic version of the contrast, namely
$$
\sup_{u\in [\varepsilon,1-\varepsilon]}\sup _{\boldsymbol{\theta} \in \Theta} \Big | \frac 1{nh_n} \, \sum_{k=1}^n \Phi \big ((\widetilde X_{k-t}(u))_{t\geq 0},\boldsymbol{\theta} \big ) \, K\Big (\frac{\frac kn-u}{h_n}\Big)- \E \big [\Phi \big ((\widetilde X_{-k}(u))_{k\geq 0},\boldsymbol{\theta}\big ) \big ] \Big | \limiteproban 0\,.
$$
From usual arguments, see for instance \cite{bill}, this uniform version of the weak law of large number obtained in Lemma \ref{lem:statwlln} will follow from the equicontinuity of the family
$$
\left(\sup _{\boldsymbol{\theta} \in \Theta} \Big |\frac 1{nh_n} \, \sum_{k=1}^n \Phi \big ((\widetilde X_{k-t}(u))_{t\geq 0},\boldsymbol{\theta} \big ) \, K\Big (\frac{\frac kn-u}{h_n}\Big)- \E \big [\Phi \big ((\widetilde X_{-t}(u))_{t\geq 0},\boldsymbol{\theta}\big ) \big ] \Big|\right)_{u\in [\varepsilon, 1-\varepsilon]}\,.
$$
This holds from Markov inequality as $\Phi\in \Lip_p(\Theta)$ and from the relation
\begin{equation}\label{equic}
\|\widetilde X_{t}(u)-\widetilde X_{t}(u')\|_p\le  \frac {\big\|{\boldsymbol{\theta}}^*(u')-{\boldsymbol{\theta}}^*(u)\big\|  }{1-B_0(\Theta)} \,\Big ( \frac { B_1(\Theta) \,C_0(\Theta)}{1-B_0(\Theta)}+C_1(\Theta) \Big )\,,
\end{equation}
as ${\boldsymbol{\theta}}^*$ is equicontinuous.
Indeed, under {\bf A$_k(\Theta)$}, $k=1,2$, we have
\begin{align*}
\|\widetilde X_{t}(u)-\widetilde X_{t}(u')\|_p\le& \|F_{\boldsymbol{\theta}^*(u)}\big ((\widetilde X_{t-k}(u))_{k\geq 1},\xi_t\big )-F_{\boldsymbol{\theta}^*(u')}\big ((\widetilde X_{t-k}(u'))_{k\geq 1},\xi_t\big )\|_p\\
\le&\; \|F_{\boldsymbol{\theta}^*(u)}\big ((\widetilde X_{t-k}(u))_{k\geq 1},\xi_t\big )-F_{\boldsymbol{\theta}^*(u)}\big ((\widetilde X_{t-k}(u'))_{k\geq 1},\xi_t\big )\|_p\\
&\;\;+  \|F_{\boldsymbol{\theta}^*(u)}\big ((\widetilde X_{t-k}(u'))_{k\geq 1},\xi_t\big )-F_{\boldsymbol{\theta}^*(u)}\big ((\widetilde X_{t-k}(u'))_{k\geq 1},\xi_t\big )\|_p\\
 \le&\;\sum_{k=1}^\infty b_k^{(0)}(\Theta) \, \|\widetilde X_{t-k}(u')-\widetilde X_{t-k}(u)\|_p\\
 &\quad +\big\|{\boldsymbol{\theta}}^*(u')-{\boldsymbol{\theta}}^*(u)\big\| \times \sup_{{\boldsymbol{\theta}} \in \Theta}  \big \| \partial_{\boldsymbol{\theta}}^1 F_{\boldsymbol{\theta}}(\widetilde X_{t-1}(u),\widetilde X_{t-2}(u),\ldots,\xi_t) \big \|_p\,.
\end{align*}
We upper-bound 
\begin{multline*}
 \sup_{{\boldsymbol{\theta}} \in \Theta}  \big \| \partial_{\boldsymbol{\theta}}^1 F_{\boldsymbol{\theta}}(\widetilde X_{t-1}(u),\widetilde X_{t-2}(u),\ldots,\xi_t) \big \|_p
 \\
 \le \sum_{k=1}^\infty b^{(1)}_k(\Theta) \,\|\widetilde X_{t-k}(u)\|_p+\sup_{{\boldsymbol{\theta}} \in \Theta}  \big \| \partial_{\boldsymbol{\theta}}^1 F_{\boldsymbol{\theta}}(0,0,\ldots,\xi_t) \big \|_p\,. 
\end{multline*}
By a similar argument than in the proof of Lemma \ref{lem*}, we deduce that \eqref{equic} holds.
\end{proof}
Now we are  in position  to prove Theorem \ref{theo3}. 
\begin{proof}[Proof of Theorem \ref{theo3}]
We follow the usual proof of asymptotic normality of a M-estimator. This will follow from the  3 forthcoming  steps:
~\\
$\bullet$ I/\ We establish that the family $\partial_{\boldsymbol{\theta}} H_t\big (\boldsymbol{\theta}^*(u)\big )=\Big (\frac {\partial}{\partial \theta_i} H_t \big (\boldsymbol{\theta}^*(u)\big )  \Big )_{1\leq i \leq d}$ for  $i_n\leq t \leq j_n$ satisfies a multidimensional central limit theorem, where we denote $H_t\big (\boldsymbol{\theta}\big )= \Phi \big ((\widetilde X_{t-k}(u))_{k\in \N},\boldsymbol{\theta}\big )$. We notice first that
$\displaystyle
\partial _{\boldsymbol{\theta}}\E \big [\Phi \big ((\widetilde X_{t-k}(u))_{k\in \N},\boldsymbol{\theta}^*(u)\big )~| ~{\cal F}_0 \big ]=0
$ as $\boldsymbol{\theta}^*(u)$ is the unique minimizer of $\E \big [\Phi \big ((\widetilde X_{t-k}(u))_{k\in \N},\boldsymbol{\theta}^*(u)\big )~| ~{\cal F}_0 \big ]$ over the open set $\stackrel{o}{\Theta}$. \\
The function $\boldsymbol{\theta} \in \Theta \mapsto \E \big [\Phi \big ((\widetilde X_{t-k}(u))_{k\in \N},\boldsymbol{\theta}^*(u)\big )~| ~{\cal F}_0 \big ]$ is differentiable under the condition $\big \| \partial _{\boldsymbol{\theta}} \Phi \big \|\in \Lip_{p}(\Theta)$. Thus $ \partial_{\boldsymbol{\theta}} H_t\big (\boldsymbol{\theta}^*(u)\big )$ constitutes a differences of martingale sequence. 
We also have $\E\big[\big \| \partial _{\boldsymbol{\theta}} \Phi \big \|^2\big]< \infty$ and we can apply the CLT for differences of martingale sequences (See for instance \cite{bill}). As a consequence we can apply the point 2. of Lemma \ref{TLCK} and  obtain the multidimensional central limit theorem
\begin{multline}\label{tlcmulti}
 \frac 1 {\sqrt{n  h_n} } \, \sum_{t=i_n(u)}^{j_n(u)}  \partial_{\boldsymbol{\theta}} \Phi \big ((\widetilde X_{t-k}(u))_{k\in \N},\boldsymbol{\theta}^*(u)\big )\, K\Big(\frac{\frac tn-u}{h_n}\Big) 
 \limiteloin {\cal N} \Big ( 0 \, , \, \Sigma\big (\boldsymbol{\theta}^*(u)\big )\Big )
 \end{multline}
 \begin{multline*}
\mbox{with }\quad \Sigma\big (\boldsymbol{\theta}^*(u)\big )=  \int_\R K^2(x)dx  \\
\times   \sum_{t\in \Z}\Big(\cov  \big [\frac {\partial}{\partial \theta_i} \Phi \big ((\widetilde X_{-k}(u))_{k\in \N},\boldsymbol{\theta}^*(u)\big ) \, , \, \frac {\partial}{\partial \theta_j} \Phi \big ((\widetilde X_{t-k}(u))_{k\in \N},\boldsymbol{\theta}^*(u)\big ) \big ] \Big )_{1\leq i,j \leq d}.
\end{multline*}
$\bullet$ II/\ We use a Taylor-Lagrange expansion for establishing
\begin{multline}\label{taylor}
 \frac 1 {\sqrt{n  h_n} } \, \sum_{t=i_n(u)}^{j_n(u)}  \partial_{\boldsymbol{\theta}} \Phi \big ((\widetilde X_{t-k}(u))_{k\in \N},\widehat {\boldsymbol{\theta}}(u)\big )\, K\Big(\frac{\frac tn-u}{h_n}\Big) \\
 \hspace*{-2cm}=\frac 1 {\sqrt{n  h_n} } \, \sum_{t=i_n(u)}^{j_n(u)}  \partial_{\boldsymbol{\theta}} \Phi \big ((\widetilde X_{t-k}(u))_{k\in \N},\boldsymbol{\theta}^*(u)\big )\, K\Big(\frac{\frac tn-u}{h_n}\Big) \\
+ \sqrt{n  h_n} \cdot \frac 1 {n  h_n } \, \, \sum_{t=i_n(u)}^{j_n(u)}  \partial^2_{\boldsymbol{\theta}^2} \Phi \big ((\widetilde X_{t-k}(u))_{k\in \N},\overline {\boldsymbol{\theta}}(u) \big )  K\Big(\frac{\frac tn-u}{h_n} \Big )\big ( \widehat{\boldsymbol{\theta}}(u)- \boldsymbol{\theta}^*(u) \big ),
\end{multline}
   where $\overline {\boldsymbol{\theta}}(u)$ belongs to the segment with extremities $\boldsymbol{\theta}^*(u) $ and $ \widehat{\boldsymbol{\theta}}(u)$. From Theorem \ref{theo1}, we have $\overline {\boldsymbol{\theta}} (u)\limiteproban \boldsymbol{\theta}^*(u)$. Moreover, since $\E \big [ \big \| \partial^2 _{\boldsymbol{\theta}^2} \Phi\big ((\widetilde X_{t-k}(u))_{k\in \N},\boldsymbol{\theta} \big )\big \|\big ]<\infty$ for any $\boldsymbol{\theta} \in \Theta$ and $\boldsymbol{\theta}\in \Theta \mapsto \partial^2 _{\boldsymbol{\theta}^2} \Phi\big ((\widetilde X_{t-k}(u))_{k\in \N},\boldsymbol{\theta}\big )$ is uniformly continuous because $\Theta$ is a bounded set included in $\R^d$, we can apply Lemma \ref{TLCK} and then:
\begin{multline}\label{cond2}\!\!
\frac 1 {nh_n } \!\! \sum_{t=i_n(u)}^{j_n(u)}\!\! \big (\!  \partial^2_{\boldsymbol{\theta}^2} \Phi  ((\widetilde X_{t\!-\!k}(u))_{k\!\in\! \N},\!\overline {\boldsymbol{\theta}}(u)  )\!\! -\!\!   \E \big [ \! \partial^2_{\boldsymbol{\theta}^2} \Phi  ((\widetilde X_{t\!-\!k}(u))_{k\!\in\! \N},\overline {\boldsymbol{\theta}}(u)  ) \big ]\big )\!
 K\big(\frac{\frac tn\!-\!u}{h_n}\big)\! \limiteproban\! 0 .
\end{multline}
Thus we get, with $\Gamma(\boldsymbol{\theta}^*(u))=\E \big [  \partial^2_{\boldsymbol{\theta}^2} \Phi \big ((\widetilde X_{t-k}(u))_{k\in \N},\boldsymbol{\theta}^* (u)\big ) \big ]$,
\begin{multline*}
\frac 1 {n  h_n } \, \, \sum_{t=i_n(u)}^{j_n(u)}  \partial^2_{\boldsymbol{\theta}^2} \Phi \big ((\widetilde X_{t-k}(u))_{k\in \N},\overline {\boldsymbol{\theta}}(u)\big )\, K\Big(\frac{\frac tn-u}{h_n}\Big)
  \limiteproban \Gamma(\boldsymbol{\theta}^*(u)).
\end{multline*}
Moreover, since $\widehat{\boldsymbol{\theta}}(u)$ minimizes the contrast function we have
\begin{equation}\label{deriv0}
 \frac  1{nh_n} \, \sum_{t=i_n(u)}^{j_n(u)}  \partial_{\boldsymbol{\theta}}  \Phi \big ((X^{(n)}_{t-k}(u))_{k\in \N},\widehat{\boldsymbol{\theta}} (u)\big )\, K\Big(\frac{\frac tn-u}{h_n}\Big)  =0.
\end{equation}
Using the assumptions on the Lipschitz coefficients of $\partial_{\boldsymbol{\theta}} \Phi$, the same inequalities  as \eqref{ineg1} and \eqref{ineg2} in the proof of Proposition \ref{prop1} lead for a convenient constant $C>0$ to:
\begin{eqnarray*}
\Big \|  \frac  1{nh_n}  \sum_{t=i_n(u)}^{j_n(u)}   \big (\partial_{\boldsymbol{\theta}} \Phi \big (( X^{(n)}_{t-k}(u))_{k\in \N},\widehat{\boldsymbol{\theta}} (u) \big ) -\partial_{\boldsymbol{\theta}}\Phi \big ((X_{t-k}^*(u))_{k\geq 0},\widehat{\boldsymbol{\theta}} (u)\big ) \big )K\big(\frac{\frac tn-u}{h_n}\big) \big \|_1& \leq & C  {h^\rho_n},\\ 
\Big \| \frac  1{nh_n}  \sum_{t=i_n(u)}^{j_n(u)}   \big (\partial_{\boldsymbol{\theta}}\Phi \big ((X_{t-k}^*(u))_{k\geq 0},\widehat{\boldsymbol{\theta}} (u)\big ) -\partial_{\boldsymbol{\theta}} \Phi \big ((\widetilde X_{t-k}(u))_{k\in \N},\widehat{\boldsymbol{\theta}} (u) \big ) \big )K\big(\frac{\frac tn-u}{h_n}\big) \big \|_1& \leq & \frac  {C }{nh_n}.
\end{eqnarray*}
As a consequence we deduce that:
\begin{multline} \label{mincon}
\Big \|\frac 1 {\sqrt{n  h_n} } \, \sum_{t=i_n(u)}^{j_n(u)}  \partial_{\boldsymbol{\theta}} \Phi \big ((\widetilde X_{t-k}(u))_{k\in \N},\widehat {\boldsymbol{\theta}}(u)\big )\, K\Big(\frac{\frac tn-u}{h_n}\Big)  \Big \|_1 \leq C \,\Big (\frac 1 {\sqrt{nh_n}} + h^\rho_n \sqrt{nh_n} \Big )\\
\Longrightarrow ~\frac 1 {\sqrt{n  h_n} } \, \sum_{t=i_n(u)}^{j_n(u)}  \partial_{\boldsymbol{\theta}} \Phi \big ((\widetilde X_{t-k}(u))_{k\in \N},\widehat {\boldsymbol{\theta}}(u)\big )\, K\Big(\frac{\frac tn-u}{h_n}\Big)  \limiteproban 0,
\end{multline}
by using  \eqref{condbn}.  
Finally, from \eqref{taylor}, using \eqref{cond2}, \eqref{mincon}, Slutsky Lemma and \eqref{tlcmulti}, we deduce:
$$
 \sqrt{n  h_n}\;\;\Gamma(\boldsymbol{\theta}^*(u)) \big ( \widehat{\boldsymbol{\theta}}(u)- \boldsymbol{\theta}^*(u) \big ) \limiteloin {\cal N} \Big ( 0 \, , \, \Sigma\big (\boldsymbol{\theta}^*(u)\big )\Big ),
$$
and this leads to Theorem \ref{theo3}.
\end{proof}

\begin{proof}[Proof of Proposition \ref{propLARCH}]
We already proved in Section \ref{LARCH} that $\Phi_{LARCH} \in \Lip_4(\Theta)$ as well as Assumption $\mbox{\bf Co}(\Phi_{LARCH},\Theta)$ when condition \eqref{identLARCH} holds. 
We assumed that $\boldsymbol{\theta} \in \Theta \mapsto a_i(\boldsymbol{\theta})$ are ${\cal C}^2 (\Theta)$ functions for any $i\in \N$. 
Thus in order to
 check the conditions of Theorem \ref{theo3}, we first have to prove that $\partial _{\boldsymbol{\theta}} \Phi_{LARCH} \in \Lip_4(\Theta)$. 
Indeed we use the estimates
\begin{eqnarray*}
&& \big \|\partial_{\boldsymbol{\theta}} \Phi_{LARCH}(U,\boldsymbol{\theta})  -\partial_{\boldsymbol{\theta}} \Phi_{LARCH}(V,\boldsymbol{\theta}) \big \|  \\ 
&&  \leq  8 \, \Big ( |U_1|+|V_1|+ 2a_0(\boldsymbol{\theta}) + \sum_{i=1}^\infty |a_i(\boldsymbol{\theta}) | 
 \big ( |U_{i+1}|+|V_{i+1}| \big ) \Big ) ^2\\
 &&\quad \times \Big (\big \|\partial_{\boldsymbol{\theta}} a_0(\boldsymbol{\theta})\big \| + \sum_{i=1}^\infty\big \|\partial_{\boldsymbol{\theta}}a_i(\boldsymbol{\theta})\big \| \,|U_{i+1}|\Big )
 \Big ( |U_1-V_1|+  \sum_{i=1}^\infty |a_i(\boldsymbol{\theta}) | \,  |U_{i+1}-V_{i+1}|\Big ) \\
&&  +4  \Big ( |U_1|+|V_1|+ 2a_0(\boldsymbol{\theta}) + \sum_{i=1}^\infty |a_i(\boldsymbol{\theta}) | \, \big ( |U_{i+1}|\!+\!|V_{i+1}| \big ) \Big ) ^3 \,\sum_{i=1}^\infty\big \|\partial_{\boldsymbol{\theta}}a_i(\boldsymbol{\theta})\big \| \,|U_{i+1}\!-\!V_{i+1}|\,.
\end{eqnarray*}
Therefore, using H\"older and Minkowski Inequalities, we obtain
\begin{multline*}
\E \big [ \sup_{\boldsymbol{\theta}\in \Theta} \big \|\partial_{\boldsymbol{\theta}} \Phi_{LARCH}(U,\boldsymbol{\theta})  -\partial_{\boldsymbol{\theta}} \Phi_{LARCH}(V,\boldsymbol{\theta}) \big \| \big ] \\ 
 \leq g\big (\sup_{i \geq 1} \big \{ \big \|U_i\|_4 \vee \big \|V_i\|_4 \}\big ) \, \Big ( |U_1-V_1|+  \sum_{i=1}^\infty \big (  \sup_{\boldsymbol{\theta}\in \Theta} |a_i(\boldsymbol{\theta}) |+  \sup_{\boldsymbol{\theta}\in \Theta}\big \|\partial_{\boldsymbol{\theta}}a_i(\boldsymbol{\theta})\big \|  \big )\, \big  \|U_{i+1}-V_{i+1}\big \|_4\Big ).
\end{multline*}
This inequality implies that $\partial _{\boldsymbol{\theta}} \Phi_{LARCH} \in \Lip_4(\Theta)$ under the assumptions of Proposition \ref{propLARCH}. \\
~\\
We also have to establish that under conditions of Proposition \ref{propLARCH}, 
$$\E \big [\big \| \partial _{\boldsymbol{\theta}} \Phi_{LARCH}\big ((\widetilde X_k(u))_{k\leq 0}, \boldsymbol{\theta}^*(u) \big )\big \|^2 \big ]<\infty$$ and 
$\E \big [\big \| \partial^2 _{\boldsymbol{\theta}^2} \Phi_{LARCH}\big ((\widetilde X_{-k}(u))_{k\in \N}, \boldsymbol{\theta}^*(u) \big )\big \| \big ]<\infty$. 
Indeed we have
\begin{eqnarray}
\nonumber
\partial _{\boldsymbol{\theta} } \Phi_{LARCH}\big ((\widetilde X_k(u))_{k\leq 0}, \boldsymbol{\theta}^*(u) \big )
&=&-4 
 (\xi_0^2-1  )
 \Big ( a_0(\boldsymbol{\theta}^*(u) )+\sum_{i=1}^\infty a_i(\boldsymbol{\theta}^*(u))\,\widetilde X_{-i}(u) ) \Big ) ^3 
\\
\label{partialLARCH}
&&\qquad \times \Big ( \partial_{\boldsymbol{\theta}} a_0(\boldsymbol{\theta}^*(u) )
+\sum_{i=1}^\infty \partial_{\boldsymbol{\theta}}  a_i(\boldsymbol{\theta}^*(u))\,\widetilde X_{-i}(u) ) \Big ) \,.
\end{eqnarray}
Therefore, using H\"older and Minkowski Inequalities and independence of $\xi_0$ and $(\widetilde X_k(u))_{k\leq -1}$ together with
${\cal E}_u\equiv \E \Big [ \big \| \partial _{\boldsymbol{\theta}} \Phi_{LARCH}\big ((\widetilde X_k(u))_{k\leq 0}, \boldsymbol{\theta}^*(u) \big )  \big \|^2 \Big ]$ we derive that
\begin{multline*}
{\cal E}_u  \leq  16 \, \E \big [ \big (\xi_0^2-1 \big )^2\big ] \, \Big (\E \Big [\Big ( a_0(\boldsymbol{\theta}^*(u) )+\sum_{i=1}^\infty a_i(\boldsymbol{\theta}^*(u))\,\widetilde X_{-i}(u) ) \Big ) ^8 \Big ] \Big) ^{3/4} \\
 \times \Big (\E \Big [\big \| \partial_{\boldsymbol{\theta}} a_0(\boldsymbol{\theta}^*(u) )+\sum_{i=1}^\infty \partial_{\boldsymbol{\theta}}  a_i(\boldsymbol{\theta}^*(u))\,\widetilde X_{-i}(u) \big \| ^8 \Big ] \Big) ^{1/4} \\
 \leq  C \, \Big ( \sup _{\boldsymbol{\theta}\in \Theta} | a_0(\boldsymbol{\theta} )|+\| \widetilde X_{0}(u)\|_8 \, \sum_{i=1}^\infty \sup _{\boldsymbol{\theta}\in \Theta}  |a_i(\boldsymbol{\theta})| \Big ) ^6  \\
 \times \Big ( \sup _{\boldsymbol{\theta}\in \Theta} \| \partial_{\boldsymbol{\theta}} a_0(\boldsymbol{\theta} )\|+\| \widetilde X_{0}(u)\|_8 \, \sum_{i=1}^\infty \sup _{\boldsymbol{\theta}\in \Theta}  \|\partial_{\boldsymbol{\theta}} a_i(\boldsymbol{\theta})\| \Big ) ^2\,.
\end{multline*}
Thus we obtain $\E \Big [ \big \| \partial _{\boldsymbol{\theta}} \Phi_{LARCH}\big ((\widetilde X_k(u))_{k\leq 0}, \boldsymbol{\theta}^*(u) \big )  \big \|^2 \Big ] <\infty$ with $r=8$ under suitable conditions on $(a_j)_j$. 
The expression for the second derivatives is also derived
\begin{multline*}
\partial^2 _{\boldsymbol{\theta}^2} \Phi_{LARCH}\big ((\widetilde X_k(u))_{k\leq 0}, \boldsymbol{\theta}^*(u) \big )=-4 \Big ( a_0(\boldsymbol{\theta}^*(u) )+\sum_{i=1}^\infty a_i(\boldsymbol{\theta}^*(u))\,\widetilde X_{-i}(u) ) \Big ) ^2  \\
\times\! \Big \{
 (\xi_0^2\!-\!3  ) \big (\partial_{\boldsymbol{\theta}} a_0(\boldsymbol{\theta}^*\!(u) )\!+\!\sum_{i=1}^\infty\! \partial_{\boldsymbol{\theta}}  a_i(\boldsymbol{\theta}^*(u))\widetilde X_{-i}(u) ) \big ) 
 \big (\partial_{\boldsymbol{\theta}} a_0(\boldsymbol{\theta}^*(u) )\!+\!\sum_{i=1}^\infty\! \partial_{\boldsymbol{\theta}}  a_i(\boldsymbol{\theta}^*\!
(u))\widetilde X_{-i}(u) ) \big )' \\
+ \big ( a_0(\boldsymbol{\theta}^*(u) )\!+\!\sum_{i=1}^\infty \!a_i(\boldsymbol{\theta}^*(u))\widetilde X_{-i}(u) ) \big ) \big (\partial^2_{\boldsymbol{\theta}^2} a_0(\boldsymbol{\theta}^*\!(u) )+\sum_{i=1}^\infty \partial^2_{\boldsymbol{\theta}^2}  a_i(\boldsymbol{\theta}^*(u))\widetilde X_{-i}(u) ) \big ) \Big \}.
\end{multline*}
As a consequence,  similar arguments as previously entail
$$
\E \Big [ \big \| \partial^2 _{\boldsymbol{\theta}^2} \Phi_{LARCH}\big ((\widetilde X_k(u))_{k\leq 0}, \boldsymbol{\theta}^*(u) \big ) \big \| \Big ]<\infty
$$ from Hausdorff and Minkowski inequalities.  
Finally we checked the conditions of Theorem \ref{theo3} since the asymptotic covariance matrix $\Sigma\big (\boldsymbol{\theta}^*(u)\big )$ and $\Gamma(\boldsymbol{\theta}^*(u))$ are positive definite matrix from \eqref{sigmaLARCH} using \eqref{partialLARCH}. 
\end{proof}

\begin{proof}[Proof of Proposition \ref{propcausal}]
We proved in Section \ref{subExamples} that $\alpha_k(\Phi_G, \Theta)= b_k^{(0)}(\Theta) $ and $\Phi_G \in \Lip_3(\Theta)$ when $f_{\boldsymbol{\theta}}$ and  $M_{\boldsymbol{\theta}}$ satisfy Lipschitz inequalities \eqref{lipsh}. But we also have:
\begin{multline*}
\partial_{\boldsymbol{\theta}} \Phi_G(x,\boldsymbol{\theta}) =  \frac {\partial_{\boldsymbol{\theta}}M_{\boldsymbol{\theta} }\cdot \big ((x_k)_{k\geq 2}  \big )}{M_{\boldsymbol{\theta} }\big ((x_k)_{k\geq 2}  \big )} + 2 \, \partial_{\boldsymbol{\theta}}f_{\boldsymbol{\theta} } \big ((x_k)_{k\geq 2}  \big ) \cdot  \frac {f_{\boldsymbol{\theta} } \big ((x_k)_{k\geq 2}  \big )-x_1}{M^2_{\boldsymbol{\theta} }\big ((x_k)_{k\geq 2}  \big )}\\ -2 \, \partial_{\boldsymbol{\theta}}M_{\boldsymbol{\theta} } \big ((x_k)_{k\geq 2}  \big )\cdot  \frac {\big (f_{\boldsymbol{\theta} } \big ((x_k)_{k\geq 2}  \big )-x_1\big )^2}{M^3_{\boldsymbol{\theta} }\big ((x_k)_{k\geq 2}  \big )}.
\end{multline*}
After computations and using $M_{\boldsymbol{\theta} } \geq \underline M$ as well as H\"older Inequalities, we obtain
\begin{multline*}
\E \big [ \sup_{\boldsymbol{\theta}\in \Theta} \big \|\partial_{\boldsymbol{\theta}} \Phi_G(U,\boldsymbol{\theta})  -\partial_{\boldsymbol{\theta}} \Phi_G(V,\boldsymbol{\theta}) \big \| \big ] 
 \leq  \frac 1 {\underline M} \,  \big \|\partial_{\boldsymbol{\theta}}M_{\boldsymbol{\theta} } \big ((U_k)_{k\geq 2}  \big )-\partial_{\boldsymbol{\theta}}M_{\boldsymbol{\theta} } \big ((V_k)_{k\geq 2}  \big ) \big \|_1
\\
+ \frac { \| \partial_{\boldsymbol{\theta}}M_{\boldsymbol{\theta} } \big ((V_k)_{k\geq 2}  \big )  \|_2}{\underline M}  
\big \|M_{\boldsymbol{\theta} } \big ((U_k)_{k\geq 2}  \big )-M_{\boldsymbol{\theta} } \big ((V_k)_{k\geq 2}  \big ) \big \|_2  \\
+ 2 \Big ( \frac { \| f_{\boldsymbol{\theta} } \big ((U_k)_{k\geq 2}  \big )-U_1   \|_2}{\underline M^2} \,  
\big \|\partial_{\boldsymbol{\theta}}f_{\boldsymbol{\theta} } \big ((U_k)_{k\geq 2}  \big )-\partial_{\boldsymbol{\theta}}f_{\boldsymbol{\theta} } \big ((V_k)_{k\geq 2}  \big ) \big \|_2 \\
+ 2 \cdot  \frac { \| f_{\boldsymbol{\theta} } \big ((V_k)_{k\geq 2}  \big )-V_1   \|_3\| \partial_{\boldsymbol{\theta}}f_{\boldsymbol{\theta} } \big ((V_k)_{k\geq 2}  \big ) \|_3}{\underline M^3} \,  
\big \|M_{\boldsymbol{\theta} } \big ((U_k)_{k\geq 2}  \big )-M_{\boldsymbol{\theta} } \big ((V_k)_{k\geq 2}  \big ) \big \|_3 \\
 +\frac { \| \partial_{\boldsymbol{\theta}}f_{\boldsymbol{\theta} } \big ((V_k)_{k\geq 2}  \big ) \|_2}{\underline M^2} \,   \big ( \big \|f_{\boldsymbol{\theta} } \big ((U_k)_{k\geq 2}  \big )-f_{\boldsymbol{\theta} } \big ((V_k)_{k\geq 2}  \big ) \big \|_2 + \|U_1-V_1\|_2 \big )  \Big )\\
+ 2 \Big (\frac { \| \partial_{\boldsymbol{\theta}}M_{\boldsymbol{\theta} } \big ((V_k)_{k\geq 2}  \big ) \|_3\, \|f_{\boldsymbol{\theta} } \big ((U_k)_{k\geq 2}  \big )+f_{\boldsymbol{\theta} } \big ((V_k)_{k\geq 2}  \big )-U_1 -V_1\|_3}{\underline M^3}  
\\
\times  \big ( \big \|f_{\boldsymbol{\theta} } 
\big ((U_k)_{k\geq 2}  \big )-f_{\boldsymbol{\theta} } \big ((V_k)_{k\geq 2}  \big ) \big \|_3 + \|U_1-V_1\|_3 \big ) \\
  +3 \cdot \frac { \| \partial_{\boldsymbol{\theta}}M_{\boldsymbol{\theta} } \big ((V_k)_{k\geq 2}  \big ) \|_4\, \|f_{\boldsymbol{\theta} } \big ((V_k)_{k\geq 2}  \big ) -V_1\|^2_4}{\underline M^4} \,   \big ( \big \|M_{\boldsymbol{\theta} } \big ((U_k)_{k\geq 2}  \big )-M_{\boldsymbol{\theta} } \big ((V_k)_{k\geq 2}  \big ) \big \|_4  \big ) \\
 + \frac { \| f_{\boldsymbol{\theta} } \big ((U_k)_{k\geq 2}  \big )-U_1   \|^2_3}{\underline M^3} \,  
\big \|\partial_{\boldsymbol{\theta}}M_{\boldsymbol{\theta} } \big ((U_k)_{k\geq 2}  \big )-\partial_{\boldsymbol{\theta}}M_{\boldsymbol{\theta} } \big ((V_k)_{k\geq 2}  \big ) \big \|_3  \Big ).
\end{multline*}
Thus\begin{multline*}
\E \big [ \sup_{\boldsymbol{\theta}\in \Theta} \big \|\partial_{\boldsymbol{\theta}} \Phi_G(U,\boldsymbol{\theta})  -\partial_{\boldsymbol{\theta}} \Phi_G(V,\boldsymbol{\theta}) \big \| \big ] 
 \  \leq \ g\big (\sup_{i \geq 1} \big \{ \big \|U_i\|_4 \vee \big \|V_i\|_4 \}\big )
\\
 \times \Big (  \big \|\partial_{\boldsymbol{\theta}}M_{\boldsymbol{\theta} } \big ((U_k)_{k\geq 2}  \big )-\partial_{\boldsymbol{\theta}}M_{\boldsymbol{\theta} } \big ((V_k)_{k\geq 2}  \big ) \big \|_4+\big \|\partial_{\boldsymbol{\theta}}f_{\boldsymbol{\theta} } \big ((U_k)_{k\geq 2}  \big )-\partial_{\boldsymbol{\theta}}f_{\boldsymbol{\theta} } \big ((V_k)_{k\geq 2}  \big ) \big \|_4  \\
 + \big \|M_{\boldsymbol{\theta} } \big ((U_k)_{k\geq 2}  \big )-M_{\boldsymbol{\theta} } \big ((V_k)_{k\geq 2}  \big ) \big \|_4+\big \|f_{\boldsymbol{\theta} } \big ((U_k)_{k\geq 2}  \big )-f_{\boldsymbol{\theta} } \big ((V_k)_{k\geq 2}  \big ) \big \|_4 + \| U_1-V_1 \|_4\Big ) 
\end{multline*}
from Jensen inequality and since we assume that $f_{\boldsymbol{\theta}}$, $M_{\boldsymbol{\theta}}$, $\partial_{\boldsymbol{\theta}}f_{\boldsymbol{\theta} }$ and $\partial_{\boldsymbol{\theta}}M_{\boldsymbol{\theta} }$ satisfy Lipschitz inequalities \eqref{lipsh}. As a consequence we derive
\begin{multline*}
\E \big [ \sup_{\boldsymbol{\theta}\in \Theta} \big \|\partial_{\boldsymbol{\theta}} \Phi_G(U,\boldsymbol{\theta})  -\partial_{\boldsymbol{\theta}} \Phi_G(V,\boldsymbol{\theta}) \big \|\; \big ]  \leq \; g\big (\sup_{i \geq 1} \big \{ \big \|U_i\|_4 \vee \big \|V_i\|_4 \}\big )\qquad\qquad \\ 
\times \Big (\|U_1-V_1\|_4 \!+\!\sum_{i=2}^\infty \big ( \beta_i(f,\Theta)\!+\!\beta_i(M,\Theta)+\beta_i(\partial_{\boldsymbol{\theta}}f ,\Theta)\!+\!\beta_i(\partial_{\boldsymbol{\theta}}M ,\Theta)\big ) \|U_i-V_i\|_4 \Big ),
\end{multline*}
 therefore $\partial_{\boldsymbol{\theta}} \Phi_G \in \Lip_4(\Theta)$. 
From these computations and with the inequality \eqref{lipab} we also deduce that condition \eqref{conddPhi} implies $B_0(\Theta)<1$,  and $\sum_{t=2}^\infty t\log(t) b_t(\Theta)<\infty$ follows from  $\sum_{s\ge 0}s \, \alpha_s(\Phi, \Theta)<\infty$,  required in Theorems \ref{theo1} and   
 \ref{theo3}.
Similar calculations also entail $$\E\Big [\big \|\partial^2 _{\boldsymbol{\theta}^2} \Phi \big ((\widetilde X_{-k}(u))_{k\in \N},\boldsymbol{\theta}\big ) \big \|^2\Big ]< \infty ,\quad \mbox{ for any }\quad \boldsymbol{\theta}\in \Theta,$$ since $p=4$ and $\partial^2_{\boldsymbol{\theta}^2}f_{\boldsymbol{\theta} }$ and $\partial^2_{\boldsymbol{\theta}^2}M_{\boldsymbol{\theta} }$ satisfy Lipschitz inequalities \eqref{lipsh}. 
We also have 
\begin{multline*}
\E \big [\big \| \partial _{\boldsymbol{\theta}} \Phi\big ((\widetilde X_k(u))_{k\leq 0}, \boldsymbol{\theta}^*(u) \big )\big \|^2 \big ]\leq 
\frac {12}{(1 \vee \underline{M})^2} \,  \Big ( \big \| \partial_{\boldsymbol{\theta}}M_{\boldsymbol{\theta}^*(u) }((\widetilde X_k(u))_{k\leq -1})\big \|_2^2 \\+  \big \| \partial_{\boldsymbol{\theta}}f_{\boldsymbol{\theta}^*(u) }((\widetilde X_k(u))_{k\leq -1})\big \|_2^2\,  \big \| \xi_0 \big \|_2^2 
+ \big \| \partial_{\boldsymbol{\theta}}M_{\boldsymbol{\theta}^*(u) }((\widetilde X_k(u))_{k\leq -1})\big \|_2^2 \,  \big \| \xi_0 \big \|_4^4\Big ),
\end{multline*}  
since $\widetilde X_0(u)-f_{\boldsymbol{\theta}^*(u) }((\widetilde X_k(u))_{k\leq -1})=M_{\boldsymbol{\theta}^*(u) }((\widetilde X_k(u))_{k\leq -1}) \, \xi_0$, with $M_{\boldsymbol{\theta}^*\!(u) }((\widetilde X_k(u))_{k\leq -1})$ and $\xi_0$  which are independent. Therefore, we obtain
\[
\E \big [\big \| \partial _{\boldsymbol{\theta}} \Phi\big ((\widetilde X_k(u))_{k\leq 0}, \boldsymbol{\theta}^*\!(u) \big )\big \|^2 \big ] <\infty
\]
since $p=4$. Finally \eqref{definite} ensures that asymptotic covariance matrix $\Sigma\big (\boldsymbol{\theta}^*\!(u)\big )$ and $\Gamma(\boldsymbol{\theta}^*(u))$ are positive definite matrix (see \cite{BW}) and \eqref{ident}  implies the existence and the uniqueness of $ \boldsymbol{\theta}^*\!(u)$ as the minimum of $\boldsymbol{\theta} \in \Theta \mapsto \E \big [\Phi \big ((\widetilde X_k(u))_{k\leq 1}),\boldsymbol{\theta} \big ) ~| ~{\cal F}_0\big ]$ defined in \eqref{min} (see also \cite{BW}). This ends the checking of the conditions of Theorem \ref{theo3}.
\end{proof}
\paragraph*{Aknowledgments} The authors wish to thank Rainer Dahlhaus, Thomas Mikosch and  Lionel Truquet for many inspiring discussions. 
\\
The work of the second author was funded by CY Initiative of Excellence
(grant "Investissements d'Avenir"ANR- 16-IDEX-0008),
Project "EcoDep" PSI-AAP2020-0000000013.
\bibliographystyle{imsart-number} 

\bibliography{bib}

\end{document}